\documentclass[12pt, reqno]{amsart}
\usepackage{amsmath,amssymb,amsthm}
\usepackage{amsmath, amssymb}
\usepackage{amsthm, amsfonts, mathrsfs}
\usepackage{mathptmx}
\usepackage{fullpage}
\usepackage{amsfonts,graphicx}
\usepackage{dutchcal}
\numberwithin{equation}{section}
\usepackage[colorlinks=true, pdfstartview=FitV, linkcolor=blue, citecolor=blue, urlcolor=blue]{hyperref}

\newtheorem {theorem}{Theorem}[section]
\newtheorem {lemma}[theorem]{{\bf Lemma}}

\newtheorem {proposition}[theorem]{{\bf Proposition}}
\theoremstyle{remark}
\newtheorem {remark}{{\bf Remark}}[section]

\newtheorem {definition}{{\bf Definition}}[section]
\theoremstyle{plain} \numberwithin {equation}{section}

\newcommand{\norm}[2]{\left\lVert #1 \right\rVert_{#2}}
\newcommand{\p}{{\mathbb P}}
\newcommand{\q}{{\mathbb Q}}
\def\N{{\mathbb N}}
\newcommand{\R}{{\mathbb R}}
\def\div{ \hbox{\rm div}\,  }
\newcommand\Z{{\mathbb{Z}}}
\def\la{\Lambda}
\def\aa{ \varphi }
\def\ze{ \zeta }
\def\nn{\nonumber}
\def\ta{\theta}
\def\aa{\phi}

\def\ddk{\dot \Delta_k}
\def\f{\frac}

\def\G{ \mathbf{G} }
\def\u{ \mathbf{u} }
\def\v{ \mathbf{v} }
\def\T{ \mathbb{T} }
\def\ea{ \mathcal E_\infty(t) }
\def\eat{ \mathcal E_\infty(t') }
\def\eb{  \mathcal E_1(t) }
\def\ebt{  \mathcal E_1(t') }
\begin{document}
\title{ { Global well-posedness and  large time behavior of solutions to   the  compressible Oldroyd-B model \\without stress diffusion}}

\author[Y. Zhao]{Yajuan Zhao}
\address[Y. Zhao]{School of Mathematics and Statistics, ZhengZhou University,
Zhengzhou, 450001, China}
\email{zhaoyj\_91@zzu.edu.cn}

\author[Y. Li]{Yongsheng Li}
\address[Y. Li]{School of Mathematics,
South China University of Technology,
Guangzhou, 510640, China}
\email{yshli@scut.edu.cn}

\author[T. Liang]{Tao Liang}
\address[T. Liang]{ School of Mathematics,
	South China University of Technology,
	Guangzhou, 510640, China}
\email{taolmath@163.com}

\author[X. Zhai]{Xiaoping Zhai}
\address[X. Zhai]{School  of Mathematics and Statistics, Guangdong University of Technology,
Guan-gzhou, 510520,  China}
\email{pingxiaozhai@163.com (Corresponding author)}

\footnote{\today}



\begin{abstract}
We consider the Cauchy problem ($\R^d, d=2,3$) and the initial boundary values problem ($\T^d, d=2,3$)
associated to the compressible Oldroyd-B model which is first derived by  Barrett,
Lu and S\"{u}li {\it[Existence of large-data finite-energy global weak solutions to a compressible Oldroyd-B model, Commun. Math. Sci., {15} (2017), 1265--1323]} through micro-macro-analysis of
the compressible Navier-Stokes-Fokker-Planck system.
 Due to lack of stress diffusion, the problems
considered here are very difficult. Exploiting tools from harmonic analysis,
 notably the Littlewood Paley theory,
 we first establish the global well-posedness and time-decay rates for solutions of the
model with small initial data in Besov spaces with critical regularity.
Then, through deeply exploring and fully utilizing the structure of the perturbation system,
we obtain the global well-posedness and exponential decay rates for solutions of
the model with small initial data in the Soboles spaces $H^3(\T^d)$.
Our obtained results improve
considerably the recent results  by Lu, Pokorn\'{y} {\it[Anal. Theory Appl., {36} (2020), 348--372]},
Wang, Wen {\it[Math. Models Methods Appl. Sci., {30} (2020), 139--179]},
and Liu, Lu, Wen {\it[SIAM J. Math. Anal., {53} (2021), 6216--6242]}.
\end{abstract}

\maketitle

\tableofcontents
\section{Introduction and the main results}
\subsection {Model and synopsis of result}
The Oldroyd-B model is a typical prototypical model for viscoelastic fluids, which describes the motion of some viscoelastic flows:
\begin{eqnarray}\label{sys}
\left\{\begin{aligned}
&\partial_t\rho+\div(\rho \v)=0,\\
&\partial_t\eta+\div(\eta \v)-\varepsilon\Delta\eta=0,\\
&\partial_t{\sigma}+(\v\cdot\nabla){\sigma}+ {\sigma} \div \v  -
(\nabla \v{\sigma}+{\sigma}\nabla^{\top} \v)-\varepsilon\Delta{\sigma}= \frac{K\,A_0}{2\lambda_1}\eta  \,\mathbb{Id} - \frac{A_0}{2\lambda_1}{\sigma} ,\\
&\rho(\partial_t\v+\v\cdot\nabla \v)-\mu\Delta \v-(\lambda+\mu)\nabla\div \v+\nabla
P(\rho)=\div({\sigma} - (K L\eta + \zeta\, \eta^2)\,\mathbb{Id}\, \big),
\end{aligned}\right.
\end{eqnarray}
for $(t,x)\in \mathbb{R}_+\times\Omega$ with $\Omega=\R^d$ or $\mathbb{T}^d\,(d= 2,3)$. Here $\rho=\rho(t,x)\in \mathbb{R}_+$ is the density function of the fluid, ${ \v}=(v^1(t,x),v^2(t,x),\cdot\cdot\cdot v^d(t,x))$ is the velocity. The symmetric matrix function ${\sigma} = ({\sigma}_{i,j})$, $1\leq i,j \leq d$ is the extra stress tensor
and ${\eta}={\eta}(t,x)\in\mathbb{R}_+$ represents the polymer number density defined as the
integral of a probability density function   with respect to the conformation vector, which is a microscopic variable  in the modeling of dilute polymer chains.

 The viscosities constant $\mu$ and $\lambda$ are supposed to  satisfy $\mu>0$ and $d\lambda+2\mu\ge 0$.
 In particular, the parameters $K$, $\varepsilon$,   $A_0$, $\lambda_1$ are all positive numbers, whereas $\zeta \geq 0$ and $L\geq 0$ with $\zeta + L \neq 0$. The term $K L \eta + \zeta \eta^2$ in the momentum equation \eqref{sys} can be seen as the {polymer pressure}, compared to the fluid pressure $P(\rho)=R\rho^\gamma$ for some $R>0$ and $\gamma>1$.

If the additional stress tensor $\sigma$ and the polymer number density $\eta$ vanish, the system \eqref{sys} simplifies to the classical compressible Navier-Stokes equations, which have been extensively investigated by numerous researchers, see \cite{D2000,HLX2012,XZ2021} and references therein.
When the density $\rho$ is constant,
the system \eqref{sys} reduces to the incompressible Oldroyd-B model,
which has garnered considerable attention in the literature.
For comprehensive studies on this model, we refer to \cite{BS2011,BS2012,BS2012JDE,CM2001,CZ2023,ER2015,L2012,wang2022sharp,WZ2022,zhaijmp,Z2018}, along with the accompanying references therein.

The mathematical investigation of the Oldroyd-B model has garnered significant attention due to its physical significance and wide-ranging mathematical applications. We highlight several notable results concerning the compressible Oldroyd-B model \eqref{sys}, which was initially derived from the formal macroscopic closure of the compressible Navier-Stokes-Fokker-Planck system by Barrett, Lu, and S\"{u}li \cite{BLS2017}.
Initially, Barrett, Lu, and S\"{u}li \cite{BLS2017} established the existence of global-in-time finite-energy weak solutions in a bounded open domain $\Omega\subset\mathbb{R}^2$ with stress diffusion.
However, the uniqueness of the global weak solution remained an open question. Subsequently, Lu and Zhang \cite{LZ2018} proved weak-strong uniqueness and provided a refined blow-up criterion,
relying solely on the upper bound of the fluid density.
Wang and Wen \cite{WW2020} demonstrated the global well-posedness of \eqref{sys} and associated time-decay estimates in $\mathbb{R}^3$, particularly when the initial data is near a nonzero equilibrium state. Recently, Zhai and Li \cite{ZL2021} enhanced the findings of \cite{WW2020}, proving global well-posedness and optimal decay rates of solutions with small initial data in Besov spaces in $\mathbb{R}^d$ (where $d=2,3$). Moreover, in contrast to \cite{WW2020}, Zhai and Li \cite{ZL2021} allowed the polymer number density to vanish and considered the stress tensor to be near zero equilibrium.
Additionally, there exist intriguing results on other macroscopic models of compressible Oldroyd type, as evidenced by \cite{BS2016,BS2016JDE,HWZ2019,HL2016,HWWZ2022,LLZ2005,LWW2023} and the references therein.

It is noteworthy that in \eqref{sys}, there exist stress diffusion terms $\varepsilon\Delta\eta$ and $\varepsilon\Delta \sigma$, which significantly contribute to the mathematical analysis of existence theory. However, in standard derivations of bead-spring models, the center-of-mass diffusion term is often neglected, given its significantly smaller magnitude compared to other terms in the equations, as discussed in \cite{GS1990,LM2000,R1990}.
For compressible models without stress diffusion, few results concerning global well-posedness exist. Lu and Pokorn\'{y} \cite{LP2020} obtained global-in-time finite-energy weak solutions without any restrictions on the size of the data. Liu {\it et al.} \cite{LLW2021} established the global well-posedness of strong solutions and provided optimal decay rates for the highest-order derivatives of the solutions in $\mathbb{R}^3$.

\vskip .1in
\subsection{The first main result}
Our first aim in this paper is to study the Cauchy problem  of the  compressible Oldroyd-B model without stress diffusion (i.e., $\varepsilon=0$ in \eqref{sys}). More precisely, we  focus on the following system which is derived by introducing a transformation $ \tau_{ij} = \sigma_{ij} - K \eta \mathbb{I}_{ij}$  in \cite{LLW2021}:
\begin{eqnarray}\label{m2}
\left\{\begin{aligned}
&\partial_t \rho + \div (\rho \v) = 0,\\
&\rho(\partial_t \v+ \v\cdot\nabla \v) - \mu \Delta \v - (\lambda + \mu)\nabla \div \v + \nabla (P(\rho)+K (L-1)\eta + \zeta \eta^2 )  = \div \tau  ,\\
& \partial_t \tau  + \v \cdot \nabla \tau +  \frac{A_0}{2\lambda_1} \tau    = \big( \nabla \v \tau + \tau (\nabla \v)^{\top}\big) + K \eta \big( \nabla \v + (\nabla \v)^{\top} \big)- \tau \div \v,\\
& \partial_t \eta + \div (\eta \v)  = 0.
\end{aligned}\right.
\end{eqnarray}
We consider the Cauchy
problem of \eqref{m2} in $\R^d$ with the initial data
\begin{align*}
(\rho, \v, \tau, \eta)|_{t=0}=(\rho_0(x, t), \v_0(x,t), \tau_0(x,t), \eta_0(x,t))\to(\bar{\rho}, { 0},0,0) \quad \mathrm{ for }
\quad|x|\to \infty, \  t>0,
\end{align*}
where  $\bar{\rho} $ is a given  positive constant.

Moreover,
to simply the system,  for some positive constant $\alpha$ which will be chosen later, we  introduce some new unknowns
\begin{align*}
a = \rho - \bar{\rho},\quad n = R\rho^{\gamma}-R\bar{\rho}^\gamma  + K (L-1)\eta + \zeta \eta^2,\quad\u= \frac{ \v}{\alpha}.
\end{align*}
After an elementary calculation, we find that the new variables
 $ (n, \u, \tau, \eta)$  satisfies the system
\begin{eqnarray}\label{m3}
\left\{\begin{aligned}
&\partial_t n + \alpha \u \cdot \nabla n + \alpha  R\gamma \bar{\rho}^\gamma  \div \u = - \alpha R(\gamma -1) \big((a + \bar{\rho})^\gamma - \bar{\rho}^\gamma \big) \div \u - \alpha n \div \u - \alpha \zeta \eta^2 \div \u,\\
&  \partial_t \u+  \alpha \u\cdot\nabla \u - \frac{\mu } {a+\bar{\rho}} \Delta \u - \frac{ (\lambda + \mu)}{a+\bar{\rho}}\nabla \div \u + \frac{1}{\alpha (a+\bar{\rho})} \nabla n  = \frac{1}{\alpha (a+\bar{\rho})}\div \tau  ,\\
& \partial_t \tau  + \alpha \u \cdot \nabla \tau +  \frac{A_0}{2\lambda_1} \tau    = \alpha \big( \nabla \u \tau + \tau (\nabla \u)^{\top}\big) + \alpha K \eta \big( \nabla \u + (\nabla \u)^{\top} \big)- \alpha \tau \div \u,\\
& \partial_t \eta + \alpha \div (\eta \u)  = 0.
\end{aligned}\right.
\end{eqnarray}
The system \eqref{m3} is augmented by the inclusion of the following initial conditions:
\begin{align}\label{initial}
 (n, \u, \tau, \eta)|_{t=0}=&(n_0(x), \u_0(x),\tau_0(x), \eta_0(x))\\
\stackrel{\mathrm{def}}{=}&
(R\rho^{\gamma}_0(x)-R\bar{\rho}^\gamma_0(x)  + K(L-1)\eta_0(x) + \zeta \eta^2_0(x), \frac{ \v_0(x)}{\alpha},\tau_0(x), \eta_0(x))
, \qquad x \in \R^d.\nn
\end{align}

Let $\mathcal{S}(\R^d)$ be the space of
rapidly decreasing functions over $\R^d$ and $\mathcal{S}'(\R^d)$ be its dual
space. For any $z \in\mathcal{S}'(\R^d)$,
the lower and higher frequency parts are expressed as
\begin{align*}
z^\ell\stackrel{\mathrm{def}}{=}\sum_{k\leq k_0}\dot{\Delta}_k z\quad\hbox{and}\quad
z^h\stackrel{\mathrm{def}}{=}\sum_{k>k_0}\dot{\Delta}_k z
\end{align*}
for some fixed   integer $k_0\ge 1$ (the value of $k_0$ is dependent on  the proof of the main theorem). 
The corresponding  truncated semi-norms are defined  as follows:
\begin{align*}\|z\|^{\ell}_{\dot B^{s}_{2,1}}
\stackrel{\mathrm{def}}{=}  \|z^{\ell}\|_{\dot B^{s}_{2,1}}
\ \hbox{ and }\   \|z\|^{h}_{\dot B^{s}_{2,1}}
\stackrel{\mathrm{def}}{=}  \|z^{h}\|_{\dot B^{s}_{2,1}}.
\end{align*}
Now, the first main result of this paper is stated as follows.
\begin{theorem}
\label{dingli1}
Let $d=2,3$ and $\Lambda\stackrel{\mathrm{def}}{=}\sqrt{-\Delta}$. For any $( n^\ell_0, \u^\ell_0, \eta^\ell_0 )\in \dot B^{\frac{d}{2}-1}_{2,1}(\R^d)$,\ $ \tau^\ell_0 \in \dot B^{\frac{d}{2}}_{2,1}(\R^d)$,\  $(\Lambda n^h_0,\u^h_0,\Lambda\tau^h_0, \Lambda\eta^h_0) \in \dot B^{\frac{d}{2} + 1}_{2,1}(\R^d)$. There exist a small positive constant $c_0$ such that, if
\begin{align}\label{X0}
& \|(n_0, \u_0, \eta_0)\|^\ell_{\dot B^{\frac{d}{2}-1}_{2,1}} + \| \tau_0\|^\ell_{\dot B^{\frac{d}{2}}_{2,1}}   +  \|(\Lambda n_0,\u_0,\Lambda\tau_0, \Lambda\eta_0)\|^h_{\dot B^{\frac{d}{2} + 1}_{2,1}}  \leq c_0,
\end{align}
then the system \eqref{m3} with \eqref{initial} admits a unique global-in-time solution $( n, \u, \tau, \eta)$ satisfying
\begin{align*}
&n^\ell \in C(\mathbb{R}^+;\dot B^{\frac{d}{2}-1}_{2,1}) \cap L^1(\mathbb{R}^+;\dot B^{\frac{d}{2}+1}_{2,1}),\quad \hspace{0.08cm} \Lambda n^h \in C(\mathbb{R}^+;\dot B^{\frac{d}{2}+1}_{2,1}) \cap L^1(\mathbb{R}^+;\dot B^{\frac{d}{2}+1}_{2,1});\\
&\u^\ell \in C(\mathbb{R}^+;\dot B^{\frac{d}{2}-1}_{2,1}) \cap L^1(\mathbb{R}^+;\dot B^{\frac{d}{2}+1}_{2,1}),\quad \u^h \in C(\mathbb{R}^+;\dot B^{\frac{d}{2}+1}_{2,1}) \cap L^1(\mathbb{R}^+;\dot B^{\frac{d}{2}+3}_{2,1});\\
&\tau^\ell \in C(\mathbb{R}^+;\dot B^{\frac{d}{2}}_{2,1})  \cap L^1(\mathbb{R}^+;\dot B^{\frac{d}{2}}_{2,1}) , \quad\hspace{0.36cm}  \Lambda\tau^h \in C(\mathbb{R}^+;\dot B^{\frac{d}{2} + 1}_{2,1}) \cap L^1(\mathbb{R}^+;\dot B^{\frac{d}{2}+1}_{2,1}); \nn\\
&\eta^\ell \in C(\mathbb{R}^+;\dot B^{\frac{d}{2}-1}_{2,1}), \quad \quad \quad \quad \quad \quad  \hspace{0.45cm} \Lambda\eta^h \in C(\mathbb{R}^+;\dot B^{\frac{d}{2}+1}_{2,1}).
\end{align*}
Moreover, there exists a constant $C$ such that
\begin{align}\label{xiaonorm}
&\|(n, \u ,\eta)\|^\ell_{\widetilde{L}^{\infty}_{t}(\dot B^{\frac{d}{2} - 1}_{2,1})} + \|\tau\|^\ell_{\widetilde{L}^{\infty}_{t}(\dot B^{\frac{d}{2} }_{2,1})} + \|(\Lambda n,\u,\Lambda\tau ,\Lambda\eta)\|^h_{\widetilde{L}^{\infty}_{t}(\dot B^{\frac{d}{2} +1}_{2,1})}\nonumber\\
&\quad +\|( n,\u)\|^\ell_{{L}^1_t(\dot{B}_{2,1}^{\frac{d}{2} +1} )} +\|\tau\|^\ell_{{L}^1_t(\dot{B}_{2,1}^{\frac{d}{2}} )} + \| (\Lambda n,\Lambda\tau)\|^h_{{L}^1_t(\dot{B}_{2,1}^{\frac{d}{2}+1})}+\| \u\|^h_{{L}^1_t(\dot{B}_{2,1}^{\frac{d}{2}+3})} \leq Cc_0.
\end{align}
Furthermore,  if additionally the initial data satisfying
$(n_0^{\ell},\u_0^{\ell}, \eta_0^\ell)\in {\dot{B}^{-{s}}_{2,\infty}}(\R^d),  \tau_0^{\ell}\in \dot{B}_{2,\infty}^{-{s}+1}(\R^d),$ with $1-\frac{d}{2} < {s} \le \frac{d}{2},$
then we have the following time decay rates:
\begin{align*}
	\|\Lambda^{\beta_1} (n,\u)\|_{L^2}
	\le C(1+t)^{  - \frac{\beta_1+{s}}{2}},\quad\quad &\forall\ - {s} < \beta_1 < \frac{d}{2} -1,\\
	\|\Lambda^{\beta_2} \tau\|_{L^2}
	\le C(1+t)^{  - \frac{\beta_2+{s}-1}{2}},\quad\quad &\forall\ 1- {s} < \beta_2 < \frac{d}{2}.
	\end{align*}
\end{theorem}

\begin{remark}
As far as we know, our theorem is the  first result about the global existence and convergence rates of solutions to the two-dimensional compressible  Oldroyd-B model without stress diffusion. Moreover, our
work can be viewed as an extension of \cite{LP2020},  \cite{WW2020} and \cite{ZL2021}, where the stress diffusion is included.

\end{remark}

\begin{remark}
Let $d=3$,  a similar result has been obtained by Liu {\it et al.}  \cite{LLW2021} in Sobolev spaces. However, the regularity of the initial data  is much more lower  than  \cite{LLW2021}, and is optimal  in some sense, here.
\end{remark}

\begin{remark}
Due to lack of dissipations on the fluid density $\rho$ and the polymer number density $\eta$,  thus $\rho$ and  $\eta$ have no
decay-in-time property. Moreover, if we take  $d=3$, $\beta_1=0$, and $s=\frac 32$ in Theorem \ref{dingli1},
 then we can get
 \begin{align*}
\|(n,\u)\|_{L^2}
\le& C(1+t)^{-\frac{3}{4}}
\end{align*}
which coincides with the  heat flows, thus our decay rate is optimal in some sense.
\end{remark}
\subsubsection*{\bf Scheme of the proof of Theorem \ref{dingli1}}
We shall  prove  Theorem \ref{dingli1} in Section 3. The proof contains  global well posedness and decay rates respectively.
In the first step,  we use six subsections to
 obtain the {\it a priori} estimates of the solutions in the low frequency and high frequency, respectively. Then, in the seventh  subsection, we  complete the proof of Theorem \ref{dingli1} by a continuous argument.
With the existence of the global solutions has been proved,  our   aim  in the second step is to prove the large time asymptotic behavior of the solutions. Throughout the proof, the crucial point is to establish a Lyapunov-type inequality in time for energy norms (see \eqref{sa59}) by using the pure energy argument (independent of spectral analysis) which is inspired of  the previous works \cite{xujiang2021}, \cite{ZL2021}.
To arrive at \eqref{sa59}, we need first to prove that the solutions constructed in Section 4 can propagate the regularity of the initial data with  negative   index (see Proposition \ref{propagate}).
Then, exploiting the interpolation inequality, we can  obtain the desired Lyapunov-type inequality for energy norms, which leads to the time-decay estimates.

\vskip .2in
\subsection{The second main result}
Our second aim  is to study the initial boundary value problem  of the  compressible Oldroyd-B model without stress diffusion (i.e., $\varepsilon=0$ in \eqref{sys}). More precisely,
inspired by \cite{LP2020},  we are concerned with the global well-posedness  and time-decay rates of the compressible Oldroyd-B model without stress diffusion which satisfies the following form:
\begin{eqnarray}\label{zhouqim2}
\left\{\begin{aligned}
&\partial_t\rho+\div(\rho \u)=0,\\
&\partial_t\eta+\div(\eta \u)=0,\\
&\partial_t\tau+ \tau+\div(\tau \u) = 0,\\
&\rho(\partial_t\mathbf{u} + \u\cdot \nabla \u) -  \mu \Delta \u -(\lambda+\mu) \nabla \div \u
+\nabla( P(\rho) +q(\eta)-\tau)=0,\\
&(\rho,\eta,\tau,\u)|_{t=0}=(\rho_{0},\eta_{0},\tau_{0},\u_{0}),
\end{aligned}\right.
\end{eqnarray}
where $q(\eta)=K (L-1)\eta + \zeta\, \eta^2$ and
$\tau$ is supposed to be a positive scalar function, see Section 2 of  \cite{LP2020} for more details about deriving \eqref{zhouqim2}.

From now on, attention is focused on the system \eqref{zhouqim2}  for
$(t, x) \in  [0, \infty)\times\T^d$ with the volume of $\T^d$ normalized to unity.
For notational convenience, we write
\begin{align}\label{conserint0}
\int_{\T^d}\rho_0\u_0\,dx=0.
\end{align}
Owing to the conservation of
total momentum, the property in \eqref{conserint0} is preserved in time. That is  to say, for any $t\ge0,$   there holds
\begin{align}\label{conserint1}
\int_{\mathbb T^d}\rho \u \,dx =0.
\end{align}

In order to overcome the lack of the stress diffusion, we will rewrite system  \eqref{zhouqim2}  in terms of variables $P,\u,\eta$ and $\tau$ as
 \begin{eqnarray}\label{zhouqim3}
\left\{\begin{aligned}
&\partial_t P+\div(P \u)+(\gamma-1)P\div \u=0,\\
&\partial_t\eta+\div(\eta \u)=0,\\
&\partial_t\tau + \tau+\div(\tau \u) =0,\\
&\rho(\partial_t\mathbf{u} + \u\cdot \nabla \u) -  \mu \Delta \u -(\lambda+\mu) \nabla \div \u+\nabla ( P(\rho) +q(\eta)-\tau)=0,\\
&(P,\eta,\tau,\u)|_{t=0}=(P_{0},\eta_{0},\tau_{0},\u_{0})\stackrel{\mathrm{def}}{=}
(R\rho_0^\gamma,\eta_{0},\tau_{0},\u_{0}).
\end{aligned}\right.
\end{eqnarray}

The second main result of the paper is stated as follows.
\begin{theorem}\label{dingli2}
Let  $d=2,3,$   and
\begin{align}\label{pingjunwei0}
	 c_0\le\rho_0,\eta_0\le c_0^{-1}\quad\int_{\T^d}\rho_0\u_0\,dx=0
\end{align}
for some constant $c_0>0$.
 Assume that the initial data satisfy
 $(P_0-\bar{P},\u_0, \eta_0-\bar{\eta},\tau_0)\in H^{3}(\T^d)$ for some constants $\bar{\rho}_0>0$, $\bar{P}=R\bar{\rho}_0^\gamma>0,$ and $\bar{\eta}_0>0$.
 There exists a small constant $\varepsilon$ such that, if
\begin{align*}
\norm{(P_0-\bar{P},\u_0, \eta_0-\bar{\eta},\tau_0)}{H^{3}}\le\varepsilon,
\end{align*}
then the system \eqref{zhouqim3} admits a unique  global  solution
$(P-\bar{P},\u, \eta-\bar{\eta},\tau)$ such that
\begin{align*}
&(P-\bar{P},\eta-\bar{\eta} ,\tau)\in C(\R^+;\ H^{3}),\quad
\u \in C(\R^+;\ H^{3})\cap L^{2} (\R^+;\ H^{4}).
\end{align*}
Moreover, for any $t\ge 0$ and for some pure constant $C_1 >0$,
there holds
\begin{align*}
\norm{(\u,\tau)}{H^{{3}}}\le C_1 e^{-C_1t}.
\end{align*}
\end{theorem}
\begin{remark}
As far as we know, Theorem \ref{dingli2} is the  first result about the global existence and convergence rates of solutions to the
 compressible  Oldroyd-B model without stress diffusion on bounded domains.
	\end{remark}

\subsubsection*{\bf Scheme of the proof of Theorem \ref{dingli2}}
Now, let us explain the difficulty and our main idea to prove Theorem \ref{dingli2}. Due to the lack of the  dissipations  on both the density and the polymer number density, the stability and large-time behavior problem concerned here is very difficult.
 To make up for the missing regularization, we consider a  small perturbation of the equilibrium $\bar{P}$, $\bar{\eta}$ for the density and the polymer number density, respectively.
In the  framework of the perturbation, the local well-posedness of (\ref{zhouqim3}) can be shown via a procedure that is now standard (see, e.g.,  \cite{Kawashima}). The focus of the proof is on the global bound of
$( P-\bar{P},\u,\eta-\bar{\eta},\tau)$ in $H^3(\mathbb T^d)$. We use the bootstrapping argument and start by making the
ansatz that
\begin{align*}
	\sup_{t\in[0,T]}(\norm{P-\bar{P}}{H^{3}}+\norm{\u}{H^{3}}
+\norm{\eta-\bar{\eta}}{H^{3}}+\norm{\tau}{H^{3}})\le \delta,
\end{align*}
for suitably chosen $0<\delta<1$. The main efforts are devoted to proving that, if the initial norm is taken to be sufficiently small, namely
$$
\|P_0-\bar{P}\|_{H^3}+\|\u_0\|_{H^3}+\|\eta_0-\bar{\eta}\|_{H^3} +\norm{\tau_0}{H^{3}}\le \varepsilon
$$
with sufficiently small $\varepsilon>0$, then
\begin{align} \label{task}
\sup_{t\in[0,T]}(\norm{P-\bar{P}}{H^{3}}+\norm{\u}{H^{3}}
+\norm{\eta-\bar{\eta}}{H^{3}}+\norm{\tau}{H^{3}})\le \frac{\delta}{2}.
\end{align}
It is not trivial to prove (\ref{task}). Now, let us explain  our main idea. Without loss of generality, we set $R=1$ in
$P(\rho)=R\rho^\gamma$ and take $\bar \rho=\bar \eta=1$ in the paper.
 The starting point is to write the term $\nabla (P+q(\eta)-K(L-1)-1-\zeta)$ as  a new variable rather than the nonlinear term and define
 \begin{align*}
p\stackrel{\mathrm{def}}{=}P -1,\quad b\stackrel{\mathrm{def}}{=}\eta-1.
\end{align*}
Then, we can rewrite \eqref{zhouqim3} into the following form
\begin{eqnarray}\label{mmm}
\left\{\begin{aligned}
&\partial_tp+\gamma\div \u+\u\cdot\nabla p+\gamma p\div \u=0,\\
&\partial_t \u-\mu\Delta\u-(\lambda+\mu)\nabla\div\u
   +\nabla (p+q(\eta)-K(L-1)-\zeta)=\text{Nonlinear terms},\\
&\partial_t{b}+\div \u+\u\cdot\nabla {b}+{b}\div \u=0,\\
&\partial_t\tau + \tau+\div(\tau \u) =  0,
\end{aligned}\right.
\end{eqnarray}
By the standard energy method, we can show that
 \begin{align}\label{celue1}
&\frac12\frac{d}{dt}\norm{(\frac{1}{\sqrt{\gamma}}\, p,\u,
\sqrt{K(L-1)+2\zeta}\,b,\sqrt{1+c_1}\,\tau)}{H^{3}}^2+\mu\norm{\nabla\u}{H^{3}}^2
+(\lambda+\mu)\norm{\div\u}{H^{3}}^2+\norm{\tau}{H^{3}}^2\nn\\
&\quad\le C\left(\norm{\u}{H^3}+\norm{(p,\tau)}{H^3}^2
   +\norm{b}{H^3}^2(1+\norm{b}{H^3}^2)\right)\norm{(p,\u,b,\tau)}{H^3}^2,
\end{align}
 from which we can see that \eqref{celue1} does not close under small initial
 data unless some norms of $p$, $b$ such as  $\norm{p}{H^{3}}^2$ and
 $\norm{b}{H^{3}}^2$ occurs on the left.
 In order to capture the dissipation arising from
the complicated coupling between $p$ and $b$ , our idea is to introduce the new
pressure $\aa$ as
\begin{align}\label{}
\aa\stackrel{\mathrm{def}}{=}P-1 +q(\eta)-K(L-1)-\zeta.
\end{align}
After an elementary calculation, we find that the new variable $(\aa,\u)$
satisfies  the following equations
\begin{eqnarray}\label{celuem5}
\left\{\begin{aligned}
&\partial_t \aa+ (\gamma+2\ze+K(L-1))\div\u  =\text{Nonlinear terms},\\
&\partial_t \u-\mu\Delta\u-(\lambda+\mu)\nabla\div\u+\nabla \aa
=\text{Nonlinear terms}.
\end{aligned}\right.
\end{eqnarray}
Especially, the linearized system of \eqref{celuem5} has   the same  structure as the compressible Navier-Stokes equations. Hence, by exploiting   delicate energy analysis, we can capture the damping effect of  $\aa$ and smoothing effect of $\u$ in \eqref{celuem5}.

\vskip .1in

Although we have obtained the damping effect of $\aa$, another difficulty to prove (\ref{task}) is that  we still cannot get any damping effect of $p$ and $b$, respectively. So, the energy estimate like \eqref{celue1} is invalid to our bootstrap argument. We need to make a more dedicated energy estimate as follows
\begin{align}\label{celueyiming1}
&\frac12\frac{d}{dt}\norm{\left(\frac{1}{\sqrt{\gamma}}\, p,\u, \sqrt{K(L-1)+2\zeta}\,b,\sqrt{1+c_1}\,\tau\right)}{H^{3}}^2
     -\frac{1}{2\gamma}\frac{d}{dt}\int_{\T^d}\frac{p}{1+p}(\Lambda^3p)^2dx\nn\\
     &\qquad
      -\frac{K(L-1)}{2}\frac{d}{dt}\int_{\T^d}\frac{b}{1+b}(\Lambda^3b)^2dx
      +\mu\norm{\nabla\u}{H^{3}}^2
     +(\lambda+\mu)\norm{\div\u}{H^{3}}^2+\norm{\tau}{H^{3}}^2\nn\\
& \quad\leq C\left(\norm{\u}{H^3}+(\norm{p}{H^3}^2+\norm{b}{H^3}^2)\norm{\u}{H^3}
       +\norm{(\u,\aa,\tau)}{H^3}^2\right)
      \|( p, \u, b,\tau)\|_{H^{3}}^2.
\end{align}

Compared to \eqref{celue1}, the advantage of the refined energy estimates
\eqref{celueyiming1}
is that the time integral of
$\norm{\u}{H^3}+(\norm{p}{H^3}^2+\norm{b}{H^3}^2)\norm{\u}{H^3}
       +\norm{(\u,\aa,\tau)}{H^3}^2$
 in front of $\|( p, \u, b,\tau)\|_{H^{3}}^2$ is time integrable. This is because
the damping effect  and smoothing effect on $\aa$  and $\u$ in \eqref{celuem5}.
With the above {\it a priori }estimates,  finally, we  use the continuity argument to close the energy argument in  the framework of small initial data.
\section{Preliminaries}
\label{ineq}
Throughout the paper, $C > 0$ stands for a generic harmless ``constant''. For brevity,
we sometime write $f\lesssim g$ instead of $f \le Cg$.
 Let $A$, $B$ be two operators, we denote $[A, B] = AB - BA$, the commutator
between $A$ and $B$. Denote $\langle f,g\rangle$ the $L^2(\R^d)$ inner product of $f$ and $g$.  For $X$ a Banach space and $I$ an interval of $\mathbb{R}$, for any $f,g\in X$, we agree that
$\left\|\left(f,g\right)\right\| _{X}\stackrel{\mathrm{def}}{=} \left\|f\right\| _{X}
+\left\|g\right\|_{X}$
and denote by $C(I; X)$ the set of
continuous functions on $I$ with values in $X$.

\subsection{Littlewood-Paley decomposition and Besov spaces}
Let us briefly recall the
Littlewood-Paley decomposition and Besov spaces for convenience. More details may be
found for example in Chap. 2 and Chap. 3 of \cite{bcd}.
\begin{definition}
\label{def2.1}
Considering two smooth functions $\varphi$ and $\chi$ on $ \R$ with the supports $supp \varphi \subset [ \frac{3}{4} ,\frac{8}{3} ]$ and  $supp \chi \subset [0, \frac{4}{3}]$ such that
\begin{align*}
& \sum_{k \in \Z} \varphi (2^{-k}\xi) = 1 \quad for \quad \xi > 0 \quad and \quad \chi (\xi) \stackrel{\mathrm{def}}{=} 1-\sum_{k \ge 0} \varphi (2^{-k}\xi) \quad for \quad \xi \in \R.
\end{align*}
Then we define homogeneous dyadic blocks $ \dot{\Delta}$
\begin{align*}
& \dot{\Delta}_k f = \mathscr{F}^{-1} (\varphi (2^{-k}|\xi|)\hat{f}) \quad and \quad \dot{S}_{k}f = \mathscr{F}^{-1} (\chi (2^{-k}|\xi|)\hat{f}).
\end{align*}
If A(D) is a 0 order Fourier multiplier, then we have
\begin{align*}
& \| \dot{\Delta}_k (A(D))f \|_{L^p} \le C\|  \ddk f \|_{L^p},  \qquad \forall \ p \in [1, \infty].
\end{align*}
\end{definition}

\begin{definition}
\label{def2.2}
Let $p,r \in [1, \infty]$, $ s \in \R$ and $ f \in \mathcal{S}'(\R^d)$. We define following Besov norm by
\begin{align*}
& \| f\|_{\dot{B}^s_{p,r}} \stackrel{\mathrm{def}}{=} \big\| (2^{ks}\| \dot{\Delta}_k f\|_{L^p})_k \big\|_{\ell^r(\Z)}
\end{align*}
and the Besov space as follows
\begin{align*}
& \dot{B}^s_{p,r}(\R^d) \stackrel{\mathrm{def}}{=}  \Big\{ f \in \mathcal{S}_{h}'(\R^d),  \big| \| f\|_{\dot{B}^s_{p,r}}<\infty  \Big\},
\end{align*}
where $\mathcal{S}_{h}'(\R^d)$ denotes $ f \in \mathcal{S}'(\R^d)$  and \ $ \lim\limits_{k \rightarrow\infty} \| \dot{S}_{k}f\|_{L^{\infty}} = 0$.
\end{definition}

In this paper, we use the ``time-space" Besov spaces or Chemin-Lerner space introduced by Chemin and Lerner.
\begin{definition}
\label{chemin}
Let $s \in \R $ , and $0 < T \le +\infty$, we define
\begin{align*}
& \| f\|_{\widetilde{L}^{q}_{T}(\dot B^{s}_{p,r})} \stackrel{\mathrm{def}}{=}   \big\|2^{ks} \| \dot{\Delta}_k f\|_{L^q(0,T;L^p)}   \big\|_{\ell ^r}
\end{align*}
for $p,q \in [1, \infty]$ and with the standard modification for $p,q = \infty$.
\end{definition}
By Minkowski's inequality, we have the following inclusions between the Chemin-Lerner space $ \widetilde{L}^{q}_{T}(\dot B^{s}_{p,r})$ and the Bochner space $ L^{q}_{T}(\dot B^{s}_{p,r})$:
\begin{align*}
& \| f\|_{\widetilde{L}^{q}_{T}(\dot B^{s}_{p,r})} \le \| f\|_{L^{q}_{T}(\dot B^{s}_{p,r})} \quad \text{if} \ q \le r,\qquad \| f\|_{L^{q}_{T}(\dot B^{s}_{p,r})} \le \| f\|_{\widetilde{L}^{q}_{T}(\dot B^{s}_{p,r})} \quad \text{if} \ q \ge r.
\end{align*}

\subsection{Analysis tools in Besov spaces}
Let us first recall classical Bernstein's lemma of Besov spaces.
\begin{lemma}\label{bernstein}
Let $\mathcal{B}$ be a ball and $\mathcal{C}$ a ring of $\mathbb{R}^d$. A constant $C$ exists so that for any positive real number $\lambda$, any
non-negative integer k, any smooth homogeneous function $\sigma$ of degree m, and any couple of real numbers $(p, q)$ with
$1\le p \le q\le\infty$, there hold
\begin{align*}
&\mathrm{Supp} \,\hat{u}\subset\lambda \mathcal{B}\Rightarrow\sup_{|\alpha|=k}\|{\partial^{\alpha}f}\|_{L^q}\le C^{k+1}\lambda^{k+d(\frac1p-\frac1q)}\|{f}\|_{L^p},\\
&\mathrm{Supp} \,\hat{f}\subset\lambda \mathcal{C}\Rightarrow C^{-k-1}\lambda^k\|{f}\|_{L^p}\le\sup_{|\alpha|=k}\|{\partial^{\alpha}f}\|_{L^p}
\le C^{k+1}\lambda^{k}\|{f}\|_{L^p},\\
&\mathrm{Supp} \,\hat{f}\subset\lambda \mathcal{C}\Rightarrow \|{\sigma(D)f}\|_{L^q}\le C_{\sigma,m}\lambda^{m+d(\frac1p-\frac1q)}\|{f}\|_{L^p},
\end{align*}
where $\hat{f}$ denotes the Fourier transform of $f$.
\end{lemma}

The following like-Bernstein inequality will be used frequently.
\begin{lemma}\label{like-bernstein}(see \cite{danxu})
If supp $ \hat{f} \subset \big\{ \xi \in \R^d : R_1 \lambda \le |\xi| \le R_2 \lambda\big\}$, then there exists C depending only on $ d$, $ R_1$ and $R_2$ such that for all $1<p<\infty$,
\begin{align*}
& C \lambda^2 (\frac{p-1}{p}) \int_{\R^d} |f|^p dx \le (p-1) \int_{\R^d}|\nabla f|^2 |f|^{p-2} dx = -\int_{\R^d}\Delta f |f|^{p-2}f dx.
\end{align*}
\end{lemma}

The following embedding inequality and interpolation inequality are also often used in this paper.
\begin{lemma}\label{embedding}(see \cite{bcd})
Let $ 1\le p, r, r_1, r_2 \le \infty$.
\begin{itemize}

  \item
  Complex interpolation: if $ f \in \dot{B}^{s_1}_{p, r_1} \cap \dot{B}^{s_2}_{p, r_1}(\R^d)$ and $ s_1 \neq s_2$, then $ f \in \dot{B}^{\ta s_1 + (1 - \ta)s_2}_{p, r}(\R^d)$ for all $ \ta \in (0, 1)$ and
\begin{align*}
\|f \|_{\dot{B}^{\ta s_1 + (1 - \ta)s_2}_{p, r}}\le C \| f\|^{\ta}_{\dot{B}^{s_1}_{p, r_1}} \|f \|^{1-\ta}_{\dot{B}^{s_2}_{p, r_1}}
\end{align*}
with $ \frac{1}{r} = \frac{\ta}{r_1} + \frac{1-\ta}{r_2}$.

  \item
Real interpolation: if $ f \in \dot{B}^{s_1}_{p, \infty} \cap \dot{B}^{s_2}_{p, \infty}(\R^d)$ and $ s_1 < s_2$, then $ f \in \dot{B}^{\ta s_1 + (1 - \ta)s_2}_{p, 1}(\R^d)$ for all $ \ta \in (0, 1)$ and
\begin{align*}
\|f \|_{\dot{B}^{\ta s_1 + (1 - \ta)s_2}_{p, 1}} \le \frac{C}{\ta(1-\ta)(s_2-s_1)} \| f\|^{\ta}_{\dot{B}^{s_1}_{p, \infty}} \|f \|^{1-\ta}_{\dot{B}^{s_2}_{p, \infty}}
\end{align*}
\item
Embedding: if $ s \in \R$, $ 1 \le p_1 \le p_2 \le \infty$ and $ 1 \le r_1 \le r_2 \le \infty$, then we have the continuous embedding $\dot{B}^s_{p_1,r_1} (\R^d)\hookrightarrow \dot{B}^{s-d(\frac{1}{p_1} - \frac{1}{p_2})}_{p_2,r_2}(\R^d)$.
\end{itemize}
\end{lemma}

Next we recall a few nonlinear estimates in Besov spaces, we need para-differential decomposition of Bony in the homogeneous context:
\begin{align}
\label{bony}
& fg = \dot{T}_f g + \dot{T}_g f + \dot{R}(f,g),
\end{align}
where
\begin{align*}
& \dot{T}_f g \stackrel{\mathrm{def}}{=} \sum_{k\in \Z} \dot{S}_{k-1}f \dot{\Delta}_k g,\qquad \dot{R}(f,g) \stackrel{\mathrm{def}}{=} \sum_{k\in \Z} \dot{\Delta}_k f \tilde{\dot{\Delta}}_k g, \qquad \tilde{\dot{\Delta}}_k g \stackrel{\mathrm{def}}{=} \sum_{|k-k'| \le 1} \dot{\Delta}_{k'}g.
\end{align*}
The following lemma gives some classical properties of the paraproduct $\dot{T}$ and the remainder $\dot{R}$ operators.
\begin{lemma}\label{le2.5}(see \cite{bcd})
For all $s,s_1,s_2 \in \R$, $\sigma \ge 0$, and $1 \le p, p_1, p_2 \le \infty$, the paraproduct $\dot{T}$ is a bilinear, continuous operator from $ \dot B^{-\sigma}_{p_1,1} \times \dot B^{s}_{p_2,1}$ to $ \dot B^{s -\sigma}_{p,1}$ with $ \frac{1}{p} = \frac{1}{p_1} + \frac{1}{p_2}$. The remainder $\dot{R}$ is bilinear continuous from $ \dot B^{s_1}_{p_1,1} \times \dot B^{s_2}_{p_2,1}$ to $ \dot B^{s_1 + s_2}_{p,1}$ with $ s_1 + s_2 > 0$, and $ \frac{1}{p} = \frac{1}{p_1} + \frac{1}{p_2}$.
\end{lemma}
To deal with the nonlinear terms in this paper, we also need the
following product estimates in Besov spaces.
\begin{lemma} \label{le2.6}(see \cite{danxu})
Let $ 1 \le p,q \le \infty$, $ s_1 \le \frac{d}{q}, s_2 \le d \min \Big\{ \frac{1}{p} ,\frac{1}{q}\Big\}$, and $ s_1 + s_2 > d \max\Big\{ 0, \frac{1}{p} + \frac{1}{q} -1\Big\}$. For $\forall (f, g) \in \dot B^{s_1}_{q,1}(\R^d) \times \dot B^{s_2}_{p,1}(\R^d)$, then we have
\begin{align*}
\| fg\|_{\dot{B}^{s_1 + s_2 - \frac{d}{q}}_{p, 1}} \le C \| f\|_{\dot{B}^{s_1}_{q, 1}} \| g\|_{\dot{B}^{s_2}_{p, 1}}.
\end{align*}
\end{lemma}

\begin{lemma}\label{le_product}(see \cite{bcd})
	Let $ s>0, \ 1 \le p, r \le \infty, \ f \in L^{\infty}(\R^3) \cap \dot{B}^s_{p,r}(\R^3)$, and $ g \in L^{\infty}(\R^3) \cap \dot{B}^s_{p,r}(\R^3) $, then there holds
	\begin{align}
	\| fg\|_{\dot{B}^s_{p,r}} \le \frac{C^{s+1}}{s} ( \| f\|_{L^{\infty}} \| g\|_{\dot{B}^s_{p,r}} + \| g\|_{L^{\infty}} \| f\|_{\dot{B}^s_{p,r}} ).
	\end{align}
\end{lemma}

\begin{lemma}\label{le2.7}(see \cite{ZL2021})
Let $ d \ge 2$ and $ 2\leq p \leq \min \Big\{4, \frac{2d}{d-2}\Big\}$, additionally, $ p \neq 4 $ if $ d = 2$. For any $ f \in \dot{B}^{\frac{d}{p}}_{p,1}(\R^d)$, $ g^\ell \in \dot{B}^{\frac{d}{2}-1}_{2,1}(\R^d)$ and $ g^h \in \dot{B}^{\frac{d}{p}-1}_{p,1}(\R^d)$, then we have
\begin{align}
\| fg\|^{\ell}_{\dot{B}^{\frac{d}{2}-1}_{2,1}} \le C \Big(\| g\|^\ell_{\dot{B}^{\frac{d}{2}-1}_{2,1}} + \| g\|^h_{\dot{B}^{\frac{d}{p}-1}_{p,1}}\Big)\| f\|_{\dot{B}^{\frac{d}{p}}_{p,1}}.
\end{align}
\end{lemma}

Next, we introduce a classical commutator's estimate.
\begin{lemma}\label{commutator}(see \cite{bcd})
Let $ 1 \le p \le \infty$, $ -d \min \Big\{ \frac{1}{p}, 1- \frac{1}{p}\Big\} < s \le \frac{d}{p} +1$, for any $ g \in \dot{B}^s_{p,1}(\R^d)$ and $ \nabla f \in \dot{B}^{\frac{d}{p}}_{p,1}(\R^d)$, then we have
\begin{align*}
& \big\| [\dot{\Delta}_k , f \cdot \nabla]g\big\|_{L^p} \le C d_k 2^{-ks} \| \nabla f\|_{\dot{B}^{\frac{d}{p}}_{p,1}} \| g\|_{\dot{B}^s_{p,1}},
\end{align*}
where $ (d_k)_{k \in \Z}$ denotes a sequence such that $ \| (d_k)\|_{\ell^1} \le 1$.
\end{lemma}
Systems \eqref{m3} and \eqref{zhouqim3}  also involves multivariate compositions of functions  that are bounded thanks to the following result:
\begin{lemma}\label{I(a)}(see \cite{Runst})
	Let $ m\in \N$ and $ s>0$. Let $G$ be a function in $C^{\infty}(\R^m \times \R^3)$ such that $ G(0,...,0) = 0$. Then for every real value functions $ f_1,...,f_m \in L^{\infty}(\R^d) \cap \dot{B}^s_{p,r}(\R^d)$, the function $ G(f_1, ..., f_m) \in L^{\infty}(\R^d) \cap \dot{B}^s_{p,r}(\R^d)$, then we have
	\begin{align}
	\| G(f_1, ..., f_m)\|_{\dot{B}^s_{p,r}} \le C\| (f_1,...,f_m)\|_{\dot{B}^s_{p,r}}
	\end{align}
	with $C$ depending only on $ \| f_i\|_{L^{\infty}} (i = 1,...,m)$,  their high derivatives, $s$ and  $p$.\\
	In the case $ s> -\min(\frac{d}{p}, \frac{d}{p^*})$, then $ f_1,...,f_m \in \dot{B}^s_{p,r} \cap \dot{B}^{\frac{d}{p}}_{p,r}$ implies that  $ G(f_1,...,f_m) \in \dot{B}^s_{p,r} \cap \dot{B}^{\frac{d}{p}}_{p,r}$ and there holds
	\begin{align}
	\| G(f_1, ..., f_m)\|_{\dot{B}^s_{p,r}} \le C \Big( 1+ \| f_1\|_{\dot{B}^{\frac{d}{p}}_{p,r}} + ...+ \| f_m\|_{\dot{B}^{\frac{d}{p}}_{p,r}} \Big) \| (f_1,...,f_m)\|_{\dot{B}^s_{p,r}}.
	\end{align}
\end{lemma}
\vskip .1in

\subsection{Analysis tools in Sobolev spaces}
We first recall a weighted Poincar\'e inequality first established by Desvillettes and Villani in \cite{DV05}.
\begin{lemma}\label{lem-Poi}
	Let $\Omega$  be a bounded connected Lipschitz domain and $\bar{\varrho}$ be a positive constant.  There exists a positive constant $C$, depending on $\Omega$ and $\bar{\varrho}$, such that  for any nonnegative function $\varrho$ satisfying
	\begin{align*}
		\int_{\Omega}\varrho dx=1, \quad \varrho\le\bar{\varrho},
	\end{align*}
	and any $\u\in H^1(\Omega)$, there holds
	\begin{align}\label{wPoi}
		\int_{\Omega}\varrho \left(\u-\int_{\Omega}\varrho \u \,dx\right)^2\,dx\le C\|\nabla \u\|_{L^2}^2.
	\end{align}
\end{lemma}
\vskip .1in
In order to remove the weight function $\varrho$ in  \eqref{wPoi} without resorting to the lower bound of  $\varrho$, we need  another variant of Poincar\'e inequality (see Lemma 3.2 in \cite{F04}).
\begin{lemma}\label{lem2.2}
	Let $\Omega$ be a bounded connected Lipschitz domain  and
	$p>1$ be a constant. Given positive constants $M_0$ and $E_0$, there is a constant $C=C(E_0,M_0)$ such that for any
	non-negative function $\varrho$ satisfying
	$$
	M_0\leq\int_{\Omega}\varrho dx\quad\mbox{and}\quad  \int_{\Omega}\varrho^{p}dx\leq
	E_0,
	$$
	and for any $\u\in H^1(\Omega)$, there holds
	$$
	\|\u\|_{L^2}^2\leq C\left[\|\nabla
	\u\|_{L^2}^2+\left(\int_{\Omega}\varrho|\u|\,dx\right)^2\right].
	$$
\end{lemma}
\vskip .1in
\begin{lemma}\label{daishu}{\rm(\cite{kato})}
Let $s\ge 0$, $f,g\in {H^{s}}(\T^d)\cap {L^\infty}(\T^d)$, it holds that
\begin{equation*}
\|fg\|_{H^{s}}\le C(\|f\|_{L^\infty}\|g\|_{H^{s}}+\|g\|_{L^\infty}\|f\|_{H^{s}}).
\end{equation*}
\end{lemma}
\vskip .1in
\begin{lemma}\label{jiaohuanzi}{\rm(\cite{kato})}
Let $s> 0$. Then there exists a constant $C$ such that, for any $f\in {H^{s}}(\T^d)\cap W^{1,\infty}(\T^d)$, $g\in {H^{s-1}}(\T^d)\cap {L^\infty}(\T^d)$, there holds
\begin{align*}
\norm{[\la^s,f\cdot\nabla ]g}{L^2}\le C(\norm{\nabla f}{L^\infty}\norm{\la^sg}{L^2}+\norm{\la^s f}{L^2}\norm{\nabla g}{L^\infty}).
\end{align*}
\end{lemma}
\vskip .1in
\begin{lemma}\label{fuhe}{\rm(\cite{Triebel})}
Let $s>0$ and $f\in H^s(\T^d)\cap L^\infty(\T^d)$. Assume that $F$ is a smooth function  on $\R$ with $F(0)=0$. Then we
have
$$
\|F(f)\|_{H^s}\le C(1+\|f\|_{L^\infty})^{[s]+1}\|f\|_{H^s},
$$
where the constant $C$ depends on
$\sup\limits_{k\le{[s]+2},\, t\le\|f\|_{L^\infty}} \left\|F^{(k)}(t)\right\|_{L^\infty}.$
\end{lemma}

\section{The proof of Theorem \ref{dingli1}}
We are going to  prove Theorem \ref{dingli1} by using energy argument. For clarity, the proof is divided into  two main steps. The first step is to establish the global well-posedness, and the second step is to prove the decay rates of the solutions.

\subsection{The proof  of the global well-posedness of Theorem \ref{dingli1}}
To begin with, we give some notations.
We shall denote by $\p\stackrel{\mathrm{def}}{=}\mathbb{I}-\nabla\Delta^{-1}\div$ the Leray projector which projects a
vector to its divergence free (or rotational) part. Denote also $\mathbb{Q}=\mathbb{I}-\mathbb{P}$, which
projects a vector field to its curl-free (or potential) part.
Define
 \begin{align*}
 \alpha = \sqrt{\frac{1}{R\gamma \bar{\rho}^{\gamma + 1}}}, \quad \alpha_1 = \sqrt{R\gamma \bar{\rho}^{\gamma-1}}, \quad\mu_1 = \frac{\mu}{\bar{\rho}}, \quad \mu_2 = \frac{\lambda + \mu}{\bar{\rho}},
\end{align*}
 we can further rewrite \eqref{m3} into the following form:
\begin{eqnarray}\label{reformulate}
\left\{\begin{aligned}
&\partial_t n  + \alpha_1  \div \u + \alpha \u \cdot \nabla n= f_1,\\
&  \partial_t \mathbb Q \u   - (\mu_1 + \mu_2) \Delta \mathbb Q \u  + \alpha_1 \nabla n - \alpha_1 \mathbb Q \div \tau = \mathbb Q f_2  ,\\
&  \partial_t \mathbb P \u   -  \mu_1 \Delta \mathbb P \u   - \alpha_1 \mathbb P \div \tau = \mathbb P f_2  ,\\
& \partial_t \tau  + \alpha \u \cdot \nabla \tau +  \frac{A_0}{2\lambda_1} \tau    = f_3,\\
& \partial_t \eta + \alpha  \u \cdot \nabla \eta  = f_4,
\end{aligned}\right.
\end{eqnarray}
where
\begin{align*}
&f_1 \stackrel{\mathrm{def}}{=}  - \alpha R(\gamma -1) I(a) \div \u  - \alpha n \div \u - \alpha \zeta \eta^2 \div \u,\\
&  f_2 \stackrel{\mathrm{def}}{=} - \alpha \u\cdot\nabla \u + \mu k(a) \Delta \u + (\lambda + \mu) k(a) \nabla \div \u
 - \frac{1}{\alpha} k(a) \nabla n+ \frac{1}{\alpha}k(a) \div \tau,\\
&f_3 \stackrel{\mathrm{def}}{=}  \alpha \big( \nabla \u \tau + \tau (\nabla \u)^{\top}\big) + \alpha K \eta \big( \nabla \u + (\nabla \u)^{\top} \big)- \alpha \tau \div \u,\\
&f_4 \stackrel{\mathrm{def}}{=} - \alpha \eta \div \u,
\end{align*}
with
$$ I(a)\stackrel{\mathrm{def}}{=} (a + \bar{\rho})^\gamma - \bar{\rho}^\gamma \quad\hbox{and}\quad  k(a) \stackrel{\mathrm{def}}{=} \frac{1} {a+\bar{\rho}} - \frac{1} {\bar{\rho}}.$$

To simply the notations in the energy argument, we also define the following equality:
\begin{align*}
 \mathcal E_\infty(t)\stackrel{\mathrm{def}}{=}&\| (n,\u,\eta )\|^\ell_{\dot B^{\frac{d}{2} -1}_{2,1}} + \| \tau \|^\ell_{\dot B^{\frac{d}{2}}_{2,1}}+ \|(\Lambda n,\u,\Lambda \tau,\Lambda \eta) \|^h_{\dot B^{\frac{d}{2} +1 }_{2,1}},\nonumber\\
\mathcal E_1(t)\stackrel{\mathrm{def}}{=}&\| (n,\u) \|^\ell_{\dot B^{\frac{d}{2} + 1}_{2,1}} + \| \tau \|^\ell_{\dot B^{\frac{d}{2} }_{2,1} } + \| (\Lambda n,\Lambda\tau) \|^h_{\dot B^{\frac{d}{2} + 1}_{2,1}} + \| \u \|^h_{\dot B^{\frac{d}{2} + 3}_{2,1}}.
\end{align*}
Then by using  the embedding relation in Besov spaces, one can obtain the following inequalities:
\begin{align}
\| ( n,\eta)\|_{\dot B^{\frac{d}{2}-1}_{2,1}}+\| ( n,\tau,\eta)\|_{\dot B^{\frac{d}{2}}_{2,1}}
 \lesssim&\| ( n,\eta)\|_{\dot B^{\frac{d}{2}-1}_{2,1}}^\ell+\| \tau\|_{\dot B^{\frac{d}{2}}_{2,1}}^\ell +\| ( \Lambda n,\Lambda\tau,\Lambda\eta)\|_{\dot B^{\frac{d}{2}+1}_{2,1}}^h
 \lesssim&\mathcal E_\infty(t);\label{buxiu1}\\
\|  n\|_{\dot B^{\frac{d}{2}+1}_{2,1}} +\|  n\|_{\dot B^{\frac{d}{2}+2}_{2,1}}\lesssim&\|  n\|_{\dot B^{\frac{d}{2}+2}_{2,1}}^\ell+ \| \Lambda n\|_{\dot B^{\frac{d}{2}+1}_{2,1}}^h
 \lesssim\mathcal E_1(t);\label{buxiu2}\\
\| \u \|_{\dot B^{\frac{d}{2} -1}_{2,1}}+\| \u \|_{\dot B^{\frac{d}{2} }_{2,1}}\lesssim&\| \u \|^\ell_{\dot B^{\frac{d}{2} - 1}_{2,1}}+ \| \u \|^h_{\dot B^{\frac{d}{2} + 1}_{2,1}}\lesssim\mathcal E_\infty(t);\label{buxiu3}\\
 \| \u \|_{\dot B^{\frac{d}{2} +1}_{2,1}}+\| \u \|_{\dot B^{\frac{d}{2} +2}_{2,1}}\lesssim&\| \u \|^\ell_{\dot B^{\frac{d}{2} + 1}_{2,1}}+ \| \u \|^h_{\dot B^{\frac{d}{2} + 3}_{2,1}}\lesssim\mathcal E_1(t).\label{buxiu4}
\end{align}

Moreover, throughout we make the assumption that
\begin{equation}\label{eq:smalladd}
\sup_{t\in\R_+,\, x\in\R^d} |(n,\eta)(t,x)|\leq \frac{1}{100}
\end{equation}
which will enable us to use freely the composition estimate stated in Lemma \ref{I(a)}.
Note that as $\dot B^{\frac d2}_{2,1}(\R^d)\hookrightarrow L^\infty(\R^d),$ Condition \eqref{eq:smalladd} will be ensured by the fact that the constructed solution about $(n,\eta)$ has small norm in $\dot B^{\frac d2}_{2,1}(\R^d)$.

\subsubsection{The estimate of $( n, \u)$ in the low frequency}\label{sublow_fre}
In this subsection, we are concerned with the
estimates of $( n, \u)$ in the low frequency. More precisely, we have the following lemma.
\begin{lemma}\label{le_aut_low}
Under the condition in Theorem \ref{dingli1}, there exists a constant $ C$ such that
\begin{align}\label{aut_low}
& \| (n, \u)\|^\ell_{\widetilde{L}^{\infty}_{t}(\dot B^{\frac{d}{2} - 1}_{2,1})}  +  \| (n, \u)\|^\ell_{{L}^{1}_{t}(\dot B^{\frac{d}{2} + 1}_{2,1} )} \nn\\
&\quad \le  \| (n_0, \u_0) \|^\ell_{\dot{B}^{\frac{d}{2}-1}_{2,1}}  + C\| \tau \|^\ell_{{L}^{1}_{t}(\dot B^{\frac{d}{2}}_{2,1})}  + C \int_{0}^{t} \eat\big( 1+ \eat\big)\ebt \,dt'.
\end{align}
\begin{proof}
At first, we get by
applying $ \dot{\Delta}_k$ to the first equation of \eqref{reformulate} that
\begin{align*}
	\partial_t \dot{\Delta}_k n  + \alpha_1 \dot{\Delta}_k  \div \mathbb{Q} \u =  \dot{\Delta}_k f_1 - \alpha \dot{\Delta}_k ( \u \cdot \nabla n),
\end{align*}
from which and
taking  $ L^2$ inner product with $ \dot{\Delta}_k n$, we have
\begin{align}\label{ak_l}
 \frac{1}{2} \frac{d}{dt} \| \dot{\Delta}_k n \|^2_{L^2}  + \alpha_1 \langle \dot{\Delta}_k  \div \mathbb{Q} \u,  \dot{\Delta}_k n \rangle = \langle \dot{\Delta}_k f_1 - \alpha \dot{\Delta}_k ( \u \cdot \nabla n), \dot{\Delta}_k n \rangle.
\end{align}
For convenience, we introduce the following new definition:
\begin{align*}
  \delta \stackrel{\mathrm{def}}{=} \Lambda^{-1} \div \mathbb{Q}\u.
\end{align*}
Then applying operator $\ddk\la^{-1} \div$ to the second equation of \eqref{reformulate} gives rise to
\begin{align}\label{eq_delta}
	\partial_t \ddk\delta - (\mu_1 + \mu_2)\Delta  \ddk\delta  -  \alpha_1\ddk \Lambda n  =  \ddk \Lambda^{-1} \div \mathbb Q f_2 + \ddk\alpha_1 \la^{-1} \div \mathbb{Q} \div \tau.
\end{align}
Taking the $L^2$ inner product with $ \ddk \delta$ on both hand side of \eqref{eq_delta} implies
\begin{align}\label{deltak_l}
		&\frac{1}{2} \frac{d}{dt} \| \dot{\Delta}_k \delta \|^2_{L^2}  + (\mu_1 + \mu_2) \| \nabla \dot{\Delta}_k \delta  \|^2_{L^2} -  \alpha_1 \big\langle   \dot{\Delta}_k  n, \dot{\Delta}_k \div \q \u \big\rangle \nn\\
	&\quad = \big\langle \Lambda^{-1} \div \q (\mathbb Q f_2 + \alpha_1 \div \tau)  , \dot{\Delta}_k \delta \big\rangle
\end{align}
where we have used the fact that
\begin{align*}
	\langle   \dot{\Delta}_k \Lambda n, \dot{\Delta}_k \delta \rangle =  \langle \dot{\Delta}_k  n,  \dot{\Delta}_k \div \q \u \rangle.
\end{align*}
In order to find the smoothing effect of $n$ in the low frequency part, we need to introduce an
unknown good quantity,
\begin{align*}
	\Gamma  \stackrel{\mathrm{def}}{=}  (\mu_1 + \mu_2) \Lambda n - \alpha_1 \delta.
\end{align*}
By a simple computation, we can deduce that
\begin{align*}
	\partial_t \Gamma +  \alpha_1^2 \Lambda n  = (\mu_1 + \mu_2) \big(\Lambda f_1 - \alpha \Lambda (\u \cdot \nabla n)\big) - \alpha_1 \Lambda^{-1} \div (\mathbb Q f_2 + \alpha_1 \q \div \tau).
\end{align*}
Hence, we get by a similar derivation of \eqref{ak_l} or \eqref{deltak_l} that
\begin{align}\label{gammak_l}
	& \frac{1}{2} \frac{d}{dt} \| \dot{\Delta}_k \Gamma \|^2_{L^2}  +(\mu_1 + \mu_2)\alpha_1^2  \|\dot{\Delta}_k \Lambda n \|^2_{L^2} - \alpha_1^3 \big\langle \dot{\Delta}_k n,  \dot{\Delta}_k \div \q \u \big \rangle \nn\\&\qquad  =  \big\langle  (\mu_1 + \mu_2) \big(\Lambda f_1 - \alpha \Lambda (\u \cdot \nabla n)\big) - \alpha_1 \Lambda^{-1} \div (\mathbb Q f_2 + \alpha_1 \q \div \tau), \dot{\Delta}_k \Gamma \big\rangle.
\end{align}
Let $0<\beta <1$ be determined later, multiplying \eqref{gammak_l} by $  \frac{\beta}{\alpha_1^2}$, \eqref{deltak_l} by $  (1-\beta)$, and then adding to \eqref{ak_l}, we can get
\begin{align}\label{udeltagammak_l}
	&\frac{1}{2} \frac{d}{dt} (\| \dot{\Delta}_k n\|^2_{L^2} + (1-\beta)\| \dot{\Delta}_k \delta\|^2_{L^2} + \frac{\beta}{\alpha_1^2} \| \dot{\Delta}_k \Gamma\|^2_{L^2}) \nn\\
	&\quad \quad  + (1-\beta)(\mu_1 + \mu_2) \| \nabla \dot{\Delta}_k \delta \|^2_{L^2}  +  \beta (\mu_1 + \mu_2)  \|\dot{\Delta}_k \Lambda n \|^2_{L^2}  \\
	&\quad   \lesssim \| \big(\dot{\Delta}_k f_1 ,  \dot{\Delta}_k ( \u \cdot \nabla n) \big)\|_{L^2} \| \dot{\Delta}_k n \|_{L^2} +  \| (\dot{\Delta}_k \mathbb Q f_2, \dot{\Delta}_k \div \tau)\|_{L^2}  \| \dot{\Delta}_k \delta\|_{L^2}\nn\\
	&\qquad + \| \big(\dot{\Delta}_k \la f_1 ,  \dot{\Delta}_k \la ( \u \cdot \nabla n),  \dot{\Delta}_k \mathbb Q f_2, \dot{\Delta}_k \div \tau \big)\|_{L^2} \|\dot{\Delta}_k \Gamma_{\lambda} \|_{L^2}.\nn
\end{align}
For any $ k\le k_0$, where $ k_0$ will be determined in the high frequency. On the one hand, applying Lemma \ref{bernstein} leads to
\begin{align}\label{left_eq}
\| \dot{\Delta}_k n\|^2_{L^2} + (1-\beta)\| \dot{\Delta}_k \delta\|^2_{L^2}
    + \frac{\beta}{\alpha_1^2} \| \dot{\Delta}_k \Gamma\|^2_{L^2}
& \le C(\| \dot{\Delta}_k n \|^2_{L^2} + \| \dot{\Delta}_k \delta \|^2_{L^2}) + C2^{2k} \| \dot{\Delta}_k n\|^2_{L^2}\nn\\
	& \le C(\| \dot{\Delta}_k n \|^2_{L^2} + \| \dot{\Delta}_k \delta \|^2_{L^2}).
\end{align}
On the other hand,  choosing $ \beta $ is sufficiently small such that $ C_1 \frac{\beta}{\alpha_1^2}  2^{2k_0} \le \frac{1}{2} $ , then we have
\begin{align*}
	&\| \dot{\Delta}_k n\|^2_{L^2} + (1-\beta)\| \dot{\Delta}_k \delta\|^2_{L^2} + \frac{\beta}{\alpha_1^2} \| \dot{\Delta}_k \Gamma\|^2_{L^2} \nn\\
	 &\quad \ge \| \dot{\Delta}_k n \|^2_{L^2} + (1-\beta)\| \dot{\Delta}_k \delta \|^2_{L^2} - \frac{\beta}{\alpha_1^2} (C_1 2^{2k}\| \dot{\Delta}_k n \|^2_{L^2} + \| \dot{\Delta}_k \delta \|^2_{L^2} )\nn\\
	& \quad \ge c(\| \dot{\Delta}_k n \|^2_{L^2} + \| \dot{\Delta}_k \delta \|^2_{L^2}),
\end{align*}
from which and  \eqref{left_eq} implies that
\begin{align}\label{equal}
	\| \dot{\Delta}_k n\|^2_{L^2} + (1-\beta)\| \dot{\Delta}_k \delta\|^2_{L^2} + \frac{\beta}{\alpha_1^2} \| \dot{\Delta}_k \Gamma\|^2_{L^2} \approx \| \dot{\Delta}_k n \|^2_{L^2} + \| \dot{\Delta}_k \delta \|^2_{L^2}.
\end{align}
Hence,
substituting \eqref{equal} into \eqref{udeltagammak_l} yields
\begin{align}\label{udeltagamma_l}
	&\frac{1}{2} \frac{d}{dt} (\| \dot{\Delta}_k n \|^2_{L^2} +\| \dot{\Delta}_k \delta \|^2_{L^2}) + c2^{2k} (\| \dot{\Delta}_k n \|^2_{L^2} + \| \dot{\Delta}_k \delta \|^2_{L^2}  ) \nn\\
  &\quad\lesssim \| \big(\dot{\Delta}_k f_1 ,  \dot{\Delta}_k ( \u \cdot \nabla n) \big)\|_{L^2} \| \dot{\Delta}_k n \|_{L^2} +  \| (\dot{\Delta}_k \mathbb Q f_2, \dot{\Delta}_k \div \tau)\|_{L^2}  \| \dot{\Delta}_k \delta\|_{L^2} \\
  &\qquad + \| \big(\dot{\Delta}_k \la f_1 ,  \dot{\Delta}_k \la ( \u \cdot \nabla n),  \dot{\Delta}_k \mathbb Q f_2, \dot{\Delta}_k \div \tau \big)\|_{L^2} \|\dot{\Delta}_k \Gamma_{\lambda} \|_{L^2}.\nn
\end{align}
Multiplying by $ 2^{(\frac{d}{2}- 1)k}\frac{1}{\| (\dot{\Delta}_k n, \dot{\Delta}_k \delta) \|_{L^2}}$ on both hand side of \eqref{udeltagamma_l} formally, then integrating from $0$ to $t$,  taking sum with respect to $k$ over $ k\le k_0$, we can obtain
\begin{align}\label{udelta_l}
	&\| (n, \delta ) \|^\ell_{\widetilde{L}^\infty_t(\dot{B}^{\frac{d}{2}-1}_{2,1})}  + \int_{0}^{t} \| (n, \delta )\|^\ell_{\dot B^{\frac{d}{2} + 1}_{2,1}} \,dt'\nn\\
	&\quad \lesssim  \| (n_0, \delta_0) \|^\ell_{\dot{B}^{\frac{d}{2}-1}_{2,1}} + \| \tau \|^\ell_{{L}^{1}_{t}(\dot B^{\frac{d}{2}}_{2,1})} +  \int_{0}^{t} \| (f_1, f_2, \u \cdot \nabla n )\|^\ell_{\dot B^{\frac{d}{2} - 1}_{2,1}} \,dt'
\end{align}
where we have used the fact
$$ \| \Lambda f_1 \|^\ell_{\dot B^{\frac{d}{2} - 1}_{2,1}} \lesssim \|  f_1 \|^\ell_{\dot B^{\frac{d}{2} - 1}_{2,1}}, \  \| \Lambda (\u \cdot \nabla n) \|^\ell_{\dot B^{\frac{d}{2} - 1}_{2,1}} \lesssim \|  \u \cdot \nabla n \|^\ell_{\dot B^{\frac{d}{2} - 1}_{2,1}}, \   \| \div \tau \|^\ell_{{L}^{1}_{t}(\dot B^{\frac{d}{2}-1}_{2,1})} \lesssim  \| \tau \|^\ell_{{L}^{1}_{t}(\dot B^{\frac{d}{2}}_{2,1})},$$
 and
\begin{align*}
	\frac{\| \dot{\Delta}_k n \|^2_{L^2} + \| \dot{\Delta}_k \delta \|^2_{L^2} } {\| \dot{\Delta}_k n  \|_{L^2} + \| \dot{\Delta}_k \delta \|_{L^2}}
	\ge c(\| \dot{\Delta}_k n  \|_{L^2} + \| \dot{\Delta}_k \delta \|_{L^2} ).
\end{align*}
 In the following, we need to deal with  nonlinear terms in \eqref{udelta_l}. At first, applying Lemma \ref{le2.6}, interpolation inequality and \eqref{buxiu1}--\eqref{buxiu4}, we have
\begin{align}\label{uda_l}
	\| \u \cdot \nabla n\|^\ell_{\dot B^{\frac{d}{2} - 1}_{2,1}}+\| n \div \u\|^\ell_{\dot B^{\frac{d}{2} - 1}_{2,1}} &\lesssim \| \u \|_{\dot B^{\frac{d}{2} }_{2,1}}  \|  n\|_{\dot B^{\frac{d}{2}}_{2,1}} \lesssim \|  n\|^2_{\dot B^{\frac{d}{2}}_{2,1}} + \| \u \|^2_{\dot B^{\frac{d}{2} }_{2,1}}\nn\\
	& \lesssim \|  n\|_{\dot B^{\frac{d}{2}-1}_{2,1}} \|  n\|_{\dot B^{\frac{d}{2}+1}_{2,1}} + \| \u \|_{\dot B^{\frac{d}{2} -1}_{2,1}} \| \u \|_{\dot B^{\frac{d}{2} +1}_{2,1}}  \nn\\
	& \lesssim  \ea\eb.
\end{align}
Bounding nonlinear terms involving composition functions in $f_1,f_2$ is more elaborate. Remembering $n = R\rho^{\gamma}-R\bar{\rho}^\gamma  + K (L-1)\eta + \zeta \eta^2 $ in mind, we can write $I(a)$ into the following form:
 $$I(a) \stackrel{\mathrm{def}}{=} (a+\bar{\rho})^\gamma - \bar{\rho}^\gamma = \frac{1}{R}n - \frac{K (L-1)}{R} \eta - \frac{\zeta}{R} \eta^2.$$
As a consequence, we decompose the term $ I(a) \div \mathbf{u}$ into three separate components: $ n \div \mathbf{u}$, $ \eta \div \mathbf{u}$,  and $ \eta^2\div \mathbf{u}$. The estimation of the term $ n\div \mathbf{u}$ has been previously addressed in \eqref{uda_l}; thus, we focus solely on the remaining two terms.
Thanks to Lemmas \ref{le2.6} and \ref{I(a)}, we infer from \eqref{buxiu1}--\eqref{buxiu4} that
\begin{align}\label{aIaedu_l}
\| \eta \div \u\|^\ell_{\dot B^{\frac{d}{2} - 1}_{2,1}}
	  \lesssim&\| \eta \|_{\dot B^{\frac{d}{2} - 1}_{2,1}} \| \u \|_{\dot B^{\frac{d}{2} + 1}_{2,1}}
	\lesssim\ea\eb.
\end{align}
Similarly, there holds
\begin{align}\label{eIaedu_l}
& \|  \eta^2 \div \u \|^\ell_{\dot B^{\frac{d}{2} - 1}_{2,1}} \lesssim\ea\big( 1+ \ea\big)\eb.
\end{align}

Combining \eqref{uda_l},  \eqref{aIaedu_l} and \eqref{eIaedu_l} gives
\begin{align}\label{h1_l}
	\| f_1 \|^\ell_{{L}^{1}_t ({\dot B^{\frac{d}{2} - 1}_{2,1}})} \lesssim  \big( 1 + \ea \big)\ea\eb.
\end{align}
Noticing the fact that operator $\la^{-1}\div $ is a zero order Fourier multiplier, we infer from Lemma \ref{le2.6} and \eqref{buxiu3}, \eqref{buxiu4}  that
\begin{align}\label{udu_l}
	\| \u \cdot \nabla \u\|^\ell_{\dot B^{\frac{d}{2} - 1}_{2,1}} &\lesssim \| \u \|_{\dot B^{\frac{d}{2} -1}_{2,1}} \| \u \|_{\dot B^{\frac{d}{2} +1}_{2,1}}
	\lesssim \ea\eb.
\end{align}
For the term $k(a) \stackrel{\mathrm{def}}{=}  k(n, \eta)$, we can similarly handle it like $I(a)$, and it can be easily verified that $k(0,0)=0$. However, we opt to employ the concept of Taylor expansion for binary functions. Consequently, a rephrasing of this term is necessary, taking the form:
\begin{align}\label{ine}
k(n, \eta) = k_1(0, 0) n + k_2(0, 0) \eta + n \widetilde{k}_1(n, \eta) + \eta \widetilde{k}_2(n, \eta).
\end{align}
Here, $k_1(0, 0), k_2(0, 0)$ are two constants and
 $ \widetilde{k}_1(n, \eta)$ consists of terms involving higher-order partial derivatives of $ k(n, \eta)$ with respect to the first variable evaluated at $ (0,0)$, while $ \widetilde{k}_2(n, \eta)   $ comprises terms involving higher-order partial derivatives of $ k(n, \eta)$ with respect to the second variable and mixed partial derivatives (though this combination is not unique). Moreover, it is straightforward to verify that
 $$ \widetilde{k}_1(0, 0) = 0, \qquad \widetilde{k}_2(0, 0) = 0.$$
 Therefore,  we  infer from Lemma \ref{le2.6} and \eqref{buxiu1}, \eqref{buxiu2}  that
\begin{align}\label{ada_l}
	\| n \nabla n\|^\ell_{\dot B^{\frac{d}{2} - 1}_{2,1}}+\| \eta \nabla n\|^\ell_{\dot B^{\frac{d}{2} - 1}_{2,1}}  \lesssim \| (n,\eta )\|_{\dot B^{\frac{d}{2} - 1}_{2,1}} \| n \|_{\dot B^{\frac{d}{2} + 1}_{2,1}} \lesssim \ea\eb,
\end{align}
and
\begin{align}\label{akaeda_l}
&  \| n \widetilde{k}_1(n, \eta) \nabla n\|^\ell_{\dot B^{\frac{d}{2} - 1}_{2,1}}+\| \eta \widetilde{k}_2(n, \eta) \nabla n\|^\ell_{\dot B^{\frac{d}{2} - 1}_{2,1}}\nn\\
&\quad\lesssim \| n \|_{\dot B^{\frac{d}{2} - 1}_{2,1}} \| \widetilde{k}_1(n, \eta) \|_{\dot B^{\frac{d}{2} }_{2,1}} \| n \|_{\dot B^{\frac{d}{2} + 1}_{2,1}}+\| \eta \|_{\dot B^{\frac{d}{2} - 1}_{2,1}} \| \widetilde{k}_2(n, \eta) \|_{\dot B^{\frac{d}{2} }_{2,1}} \| n \|_{\dot B^{\frac{d}{2} + 1}_{2,1}} \nn\\
& \quad\lesssim\| (n,\eta )\|_{\dot B^{\frac{d}{2} - 1}_{2,1}} \| (n, \eta) \|_{\dot B^{\frac{d}{2} }_{2,1}} \| n\|_{\dot B^{\frac{d}{2} + 1}_{2,1}}\nn\\
 & \quad\lesssim  \big( 1 + \ea \big)\ea\eb.
\end{align}
Combining \eqref{ada_l} with  \eqref{akaeda_l} gives rises to
\begin{align}\label{kada_l}
\| k(n, \eta) \nabla n\|^\ell_{\dot B^{\frac{d}{2} - 1}_{2,1}} \lesssim \ea \big(1+\ea\big)\eb.
\end{align}
Similarly,
for the term $ k(n, \eta) \nabla^2 \u$,  there hold
\begin{align*}
&\| n \nabla^2 \u\|^\ell_{\dot B^{\frac{d}{2} - 1}_{2,1}}+\| \eta \nabla^2 \u\|^\ell_{\dot B^{\frac{d}{2} - 1}_{2,1}}  \lesssim \|( n,\eta )\|_{\dot B^{\frac{d}{2} - 1}_{2,1}} \| \u \|_{\dot B^{\frac{d}{2} + 2}_{2,1}} \lesssim \ea\eb,\nn\\
&  \| n \widetilde{k}_1(n, \eta) \nabla^2 \u\|^\ell_{\dot B^{\frac{d}{2} - 1}_{2,1}}+\| \eta \widetilde{k}_2(n, \eta) \nabla^2 \u\|^\ell_{\dot B^{\frac{d}{2} - 1}_{2,1}}\nn\\
 &\quad\lesssim \| n \|_{\dot B^{\frac{d}{2} - 1}_{2,1}} \| \widetilde{k}_1(n, \eta) \|_{\dot B^{\frac{d}{2} }_{2,1}} \| \u \|_{\dot B^{\frac{d}{2} + 2}_{2,1}} +\| \eta \|_{\dot B^{\frac{d}{2} - 1}_{2,1}} \| \widetilde{k}_1(n, \eta) \|_{\dot B^{\frac{d}{2} }_{2,1}} \| \u \|_{\dot B^{\frac{d}{2} + 2}_{2,1}}\nn\\
& \quad\lesssim\| (n,\eta )\|_{\dot B^{\frac{d}{2} - 1}_{2,1}} \| (n, \eta) \|_{\dot B^{\frac{d}{2} }_{2,1}} \| \u \|_{\dot B^{\frac{d}{2} + 2}_{2,1}}\nn\\
 & \quad\lesssim \ea \big(1+\ea\big)\eb
\end{align*}
which implies that
\begin{align}\label{kaddu_l}
\| k(n, \eta) \nabla^2 \u\|^\ell_{\dot B^{\frac{d}{2} - 1}_{2,1}} \lesssim \ea \big(1+\ea\big)\eb.
\end{align}
For the term $ k(n, \eta)\div \tau$, we can also utilize a similar approach as  $ k(n ,\eta) \nabla^2 \u$ to get
\begin{align}\label{kadt_l}
\| k(n, \eta) \div \tau \|^\ell_{\dot B^{\frac{d}{2} - 1}_{2,1}} \lesssim \ea \big(1+\ea\big)\eb.
\end{align}
Together with \eqref{udu_l}, \eqref{kada_l}, \eqref{kaddu_l}, and \eqref{kadt_l}, we have
\begin{align}\label{h2_l}
\| \mathbb Q f_2 \|^\ell_{\dot B^{\frac{d}{2} - 1}_{2,1}} \lesssim \ea \big(1+\ea\big)\eb.
\end{align}

Concerning the low-frequency estimation of the incompressible component of velocity $ \p \u$, employing a treatment analogous to that in \eqref{udeltagamma_l}, we obtain
\begin{align*}
	 	\frac{1}{2} \frac{d}{dt} \| \dot{\Delta}_k  \mathbb P \u \|_{L^2}+  \mu_12^{2k}\| \dot{\Delta}_k  \mathbb P \u \|_{L^2}     \lesssim  \alpha_1 \|\dot{\Delta}_k \mathbb P \div \tau\|_{L^2} + \| \dot{\Delta}_k\mathbb P f_2\|_{L^2},
	 \end{align*}
from which we  further get
\begin{align}\label{pu_l}
&\|\p \u\|^\ell_{\widetilde{L}^{\infty}_{t}(\dot B^{\frac{d}{2}-1}_{2,1})}  + \|  \p\u \|^\ell_{{L}^{1}_{t}(\dot B^{\frac{d}{2} + 1}_{2,1})}\lesssim  \| \mathbb{P}\u_0\|^\ell_{\dot B^{\frac{d}{2}-1}_{2,1}} + \|  (\mathbb P f_2, \div \tau) \|^\ell_{{L}^{1}_t ({\dot B^{\frac{d}{2} - 1}_{2,1}})}.
\end{align}
Subsequently, by substituting \eqref{uda_l}, \eqref{h1_l} and \eqref{h2_l} into \eqref{udelta_l} and then simultaneously combining \eqref{pu_l}, we obtain \eqref{aut_low}.
\end{proof}
\end{lemma}

\subsubsection{The estimate of $\tau $ in the low frequency}
In this subsection, we are concerned with the
estimates of $\tau$ in the low frequency. More precisely, we have the following lemma.
\begin{lemma}\label{le_qt_low}
	Under the condition in Theorem \ref{dingli1}, there exists a constant $ C$ such that
	\begin{align}\label{t_low}
		\| \tau \|^\ell_{\widetilde{L}^{\infty}_{t}(\dot B^{\frac{d}{2}}_{2,1})}  + \frac{A_0}{2\lambda_1} \| \tau \|^\ell_{{L}^{1}_{t}(\dot B^{\frac{d}{2} }_{2,1} )} \le  \| \tau_0 \|^\ell_{\dot{B}^{\frac{d}{2}}_{2,1}}   + C\int_{0}^{t} \eat\ebt \, \,dt'.
	\end{align}
	\begin{proof}
Applying operator $ \dot{\Delta}_k$ to the forth equation of \eqref{reformulate}, we have
		\begin{align*}
			\partial_t  \dot{\Delta}_k \tau +\frac{A_0}{2\lambda_1}\dot{\Delta}_k \tau+ \u \cdot \nabla \dot{\Delta}_k \tau + [\dot{\Delta}_k, \u \cdot \nabla ]\tau = \dot{\Delta}_k f_3.
		\end{align*}
	 Taking the $ L^2$ inner product with $ \dot{\Delta}_k \tau$  gives
	 \begin{align}\label{eta}
	 	\frac{1}{2} \frac{d}{dt} \| \dot{\Delta}_k \tau \|^2_{L^2} +\frac{A_0}{2\lambda_1}\| \dot{\Delta}_k \tau \|^2_{L^2}- \int_{\R^d} \div \u |\dot{\Delta}_k \tau|^2 \,dx +  \langle [\dot{\Delta}_k, \u \cdot \nabla ]\tau,  \dot{\Delta}_k \tau \rangle  = \langle  \dot{\Delta}_k f_3,  \dot{\Delta}_k \tau \rangle
	 \end{align}
	 where we have used the fact
	 $$ \int_{\R^d} \u \cdot \nabla \dot{\Delta}_k \tau \cdot \dot{\Delta}_k \tau  \,dx = -\frac{1}{2} \int_{\R^d} \div \u |\dot{\Delta}_k \tau|^2 \,dx.$$
	 Thanks to the H\"{o}lder  inequality, we can infer \eqref{eta} that
\begin{align}\label{taoguji}
	 	\frac{1}{2} \frac{d}{dt} \| \dot{\Delta}_k \tau\|_{L^2}  +\frac{A_0}{2\lambda_1}   \| \dot{\Delta}_k \tau\|_{L^2}   \lesssim \| \div \u\|_{L^{\infty}} \|\dot{\Delta}_k \tau\|_{L^2} + \| [\dot{\Delta}_k, \u \cdot \nabla ]\tau\|_{L^2} +  \| \dot{\Delta}_k f_3\|_{L^2}.
	 \end{align}
Then multiplying \eqref{taoguji} by $2^{\frac{d k}{2}}$, integrating from $0$ to $t$,  taking sum with respect to $k$ over $ k\le k_0$, there holds
		\begin{align}\label{t_l}
		\| \tau \|^\ell_{\widetilde{L}^{\infty}_{t}(\dot B^{\frac{d}{2} }_{2,1})} + \frac{A_0}{2\lambda_1} \| \tau \|^\ell_{{L}^{1}_{t}(\dot B^{\frac{d}{2} }_{2,1} )} &\lesssim  \| \tau_0 \|^\ell_{\dot{B}^{\frac{d}{2}}_{2,1}} + \int_{0}^{t} \| \div \u \|_{L^{\infty}} \| \tau \|^\ell_{\dot B^{\frac{d}{2}}_{2,1}} \,dt' \nn\\
		&\quad +\int_{0}^{t} \sum_{k\leq k_0} 2^{\frac{d}{2}k} \| [\dot{\Delta}_k, \u \cdot \nabla ]\tau\|_{L^2} \,dt' + \int_{0}^{t} \| f_3 \|^\ell_{\dot B^{\frac{d}{2} }_{2,1}} \,dt'.
		\end{align}
Thanks to
 $ \dot B^{\frac{d}{2}}_{2,1}(\R^d) \hookrightarrow L^{\infty}(\R^d)$,  it follow from Lemma \ref{commutator}  and \eqref{buxiu1}, \eqref{buxiu4} that
\begin{align}\label{dut_l}
	& \| \div \u \|_{L^{\infty}} \| \tau \|^\ell_{\dot B^{\frac{d}{2}}_{2,1}} + \sum_{k\leq k_0} 2^{\frac{d}{2}k} \| [\dot{\Delta}_k, \u \cdot \nabla ]\tau\|_{L^2}
	\lesssim \| \u  \|_{\dot B^{\frac{d}{2} + 1}_{2,1}}  \| \tau  \|_{\dot B^{\frac{d}{2}}_{2,1}} \lesssim \ea\eb.
\end{align}
Similarly, for the last term in \eqref{t_l}, there holds

\begin{align}\label{h4_l}
\| f_3\|^\ell_{\dot B^{\frac{d}{2}}_{2,1}}
&\lesssim \| (\tau,\eta)\|_{\dot B^{\frac{d}{2}}_{2,1}} \| \u\|_{\dot B^{\frac{d}{2}+1}_{2,1}} \lesssim \ea\eb.
\end{align}
Inserting \eqref{h4_l} and \eqref{dut_l} into \eqref{t_l}, we can get \eqref{t_low}. Consequently, we complete the proof of Lemma \ref{le_qt_low}.
	\end{proof}
\end{lemma}
\subsubsection{The estimate of $\eta$ in the low frequency}
In this subsection, we are concerned with the
estimates of $\eta$ in the low frequency. More precisely, we have the following lemma.
\begin{lemma}\label{le_eta_low}
	Under the condition in Theorem \ref{dingli1}, there exists a constant $ C$ such that
	\begin{align}\label{eta_low}
		 \| \eta \|^\ell_{\widetilde{L}^{\infty}_{t}(\dot B^{\frac{d}{2} - 1}_{2,1})}   \le \| \eta_0 \|^\ell_{\dot{B}^{\frac{d}{2}-1}_{2,1}}   + C \int_{0}^{t} \eat\ebt \,dt'.
	\end{align}
	\begin{proof}
	 We get by a similar derivation of \eqref{taoguji} that
	 \begin{align}\label{dketa_l}
	 	\frac{1}{2} \frac{d}{dt} \| \dot{\Delta}_k \eta \|_{L^2}      \lesssim \| \div \u\|_{L^{\infty}} \|\dot{\Delta}_k \eta\|_{L^2} + \| [\dot{\Delta}_k, \u \cdot \nabla ]\eta\|_{L^2} +  \| \dot{\Delta}_k f_4\|_{L^2}.
	 \end{align}
	 Then multiplying by $2^{(\frac{d}{2} - 1)k}$, integrating from $0$ to $t$,  taking sum with respect to $k$ over $ k\le k_0$, there holds
		\begin{align}\label{eta_l}
			\| \eta  \|^\ell_{\widetilde{L}^{\infty}_{t}(\dot B^{\frac{d}{2} - 1}_{2,1})}
&\lesssim  \| \eta_0 \|^\ell_{\dot{B}^{\frac{d}{2}-1}_{2,1}} + \int_{0}^{t} \| \div \u \|_{L^{\infty}} \| \eta  \|^\ell_{\dot B^{\frac{d}{2} - 1}_{2,1}} \,dt' \nn\\
			&\quad + \int_{0}^{t} \sum_{k\leq k_0} 2^{(\frac{d}{2} -1)k} \| [\dot{\Delta}_k, \u \cdot \nabla ]\eta\|_{L^2} \,dt' + \int_{0}^{t} \| f_4 \|^\ell_{\dot B^{\frac{d}{2} - 1}_{2,1}} \,dt'.
		\end{align}
		For the term $ f_4$, we have already considered in \eqref{aIaedu_l},
		\begin{align}\label{h5_l}
			\| f_4\|^\ell_{\dot B^{\frac{d}{2} - 1}_{2,1}}=\| \eta \div \u\|^\ell_{\dot B^{\frac{d}{2} - 1}_{2,1}}  \lesssim \ea \eb.
		\end{align}
		Since $ \dot B^{\frac{d}{2}}_{2,1}(\R^d) \hookrightarrow L^{\infty}(\R^d)$, we can get by applying Lemma \ref{commutator} and \eqref{buxiu1}, \eqref{buxiu4} that
		\begin{align}\label{due_l}
			& \| \div \u \|_{L^{\infty}} \| \eta  \|^\ell_{\dot B^{\frac{d}{2} - 1}_{2,1}} + \sum_{k\leq k_0} 2^{(\frac{d}{2} -1)k} \| [\dot{\Delta}_k, \u \cdot \nabla ]\eta\|_{L^2}
			\lesssim \| \u  \|_{\dot B^{\frac{d}{2} + 1}_{2,1}}  \| \eta  \|_{\dot B^{\frac{d}{2} - 1}_{2,1}} \lesssim\ea\eb.
		\end{align}
		The combination of \eqref{h5_l} and \eqref{due_l} imply that \eqref{eta_low}.  Consequently, we complete the proof of Lemma \ref{le_eta_low}.
	\end{proof}
\end{lemma}

Next we shall
establish the high-frequency estimates of $( n, \u) $.
\subsubsection{The estimate of $( n, \u) $  in the high frequency}\label{sub_high} In this subsection, we aim to find the damping effect of the pressure  $n$ in the high frequency.
\begin{lemma}\label{le_aut_h}
	Under the condition in Theorem \ref{dingli1}, there exists a constant $ C$ such that
	\begin{align}\label{au_h}
	&\| (\Lambda n, \u)\|^h_{\widetilde{L}^{\infty}_{t}(\dot B^{\frac{d}{2}+1}_{2,1})}  + \| \Lambda n \|^h_{{L}^{1}_{t}(\dot B^{\frac{d}{2}+1}_{2,1})} + \|  \u \|^h_{{L}^{1}_{t}(\dot B^{\frac{d}{2} + 3}_{2,1})} \nn\\
	&\quad\le  \| (\Lambda n_0,\u_0)\|^h_{\dot B^{\frac{d}{2}+1}_{2,1}} + C\| \Lambda\tau \|^h_{{L}^{1}_t(\dot B^{\frac{d}{2} +1 }_{2,1})} +  C \int_{0}^{t}  \big( 1+\eat \big)\eat\ebt \, \,dt'.
	\end{align}
\begin{proof}
	In order to find the damping effect of the pressure in high frequency, we need to introduce the new unknown $$ \G \stackrel{\mathrm{def}}{=} \mathbb{Q} \u - \frac{\alpha_1}{ \mu_1 + \mu_2} \Delta^{-1} \nabla n,$$
from which and the second equation of \eqref{reformulate}, we can derive the equation of $\G$
\begin{align}\label{gfg}
\partial_t \G - (\mu_1 + \mu_2) \Delta \G =& \frac{\alpha_1^2}{\mu_1 + \mu_2} \G + \alpha_1 \q \div \tau \nn\\
&+ \mathbb Q f_2 - \frac{\alpha_1}{\mu_1 + \mu_2} \Delta^{-1} \nabla  \Big( f_1  - \alpha \u \cdot \nabla n -  \frac{\alpha_1^2}{\mu_1 + \mu_2} n\Big).
\end{align}
Thus, we can get   by a similar derivation of \eqref{eta_l} that
	\begin{align}\label{wk}
		&\frac{1}{2} \frac{d}{dt} \| \dot{\Delta}_k \G\|_{L^2}  + 2^{2k}(\mu_1 + \mu_2) \| \dot{\Delta}_k \G \|_{L^2} \nn\\
		&\quad\lesssim 2^{-k} \big(\| \dot{\Delta}_k f_1 \|_{L^2} + \| \dot{\Delta}_k n \|_{L^2} + \| \dot{\Delta}_k (\u \cdot \nabla n) \|_{L^2}\big) + \| \dot{\Delta}_k \G \|_{L^2} + \| \dot{\Delta}_k \mathbb Q f_2 \|_{L^2} + 2^k \| \dot{\Delta}_k \tau \|_{L^2}.
	\end{align}
Plugging $\G \stackrel{\mathrm{def}}{=} \mathbb{Q} \u - \frac{\alpha_1}{ \mu_1 + \mu_2} \Delta^{-1} \nabla n$ into the first equation of \eqref{reformulate} gives
	\begin{align*}
	\partial_t n  + \frac{\alpha_1^2}{\mu_1 + \mu_2 } n   + \alpha \u \cdot \nabla n= f_1 - \alpha_1  \div \G.
	\end{align*}
	Applying the operator $ \dot{\Delta}_k \Lambda $  to the above equation, we have
	\begin{align}\label{a_k}
		&\partial_t \dot{\Delta}_k \Lambda n + \frac{\alpha_1^2}{\mu_1 + \mu_2 } \dot{\Delta}_k \Lambda n + \alpha \u \cdot \nabla  \dot{\Delta}_k \Lambda n \nn\\
		&\quad = - \alpha [\dot{\Delta}_k, \u \cdot \nabla ] \Lambda n - \alpha \dot{\Delta}_k ( \Lambda \u \cdot \nabla n) - \dot{\Delta}_k \Lambda f_1   - \alpha_1 \dot{\Delta}_k \div \Lambda  \G.
	\end{align}
We proceed with a treatment similar to that of Equation \eqref{dketa_l} that
	\begin{align}\label{duk}
		&\frac{1}{2} \frac{d}{dt} \| \dot{\Delta}_k \Lambda n \|_{L^2}  + \frac{\alpha_1^2}{\mu_1 + \mu_2 } \| \dot{\Delta}_k \Lambda n \|_{L^2}\nn\\
 &\quad\lesssim \|\big([\dot{\Delta}_k, \u \cdot \nabla ] \Lambda n,  \dot{\Delta}_k ( \Lambda \u \cdot \nabla n), \dot{\Delta}_k \Lambda f_1 ,  \dot{\Delta}_k \Lambda  \div \G \big) \|_{L^2} + \| \div \u\|_{L^{\infty}} \| \dot{\Delta}_k \Lambda n\|_{L^2}.
	\end{align}

Let $\varepsilon>0$ be a small fixed constant which  will be determined later, multiplying \eqref{duk} by $  2\varepsilon$ and then adding to \eqref{wk}, we can get
	\begin{align}\label{duw_k}
		&\frac{1}{2} \frac{d}{dt} ( 2\varepsilon \| \dot{\Delta}_k \Lambda n\|_{L^2} + \| \dot{\Delta}_k \G \|_{L^2} ) +  \varepsilon \|\Lambda \dot{\Delta}_k  n\|_{L^2} + 2^{2k} (\mu_1 + \mu_2) \| \dot{\Delta}_k \G \|_{L^2} \nn\\
 &\quad\le C(\varepsilon 2^{2k}+1) \| \dot{\Delta}_k \G \|_{L^2}   + C2^{-k} \big( \| \dot{\Delta}_k n\|_{L^2} + \| \dot{\Delta}_k f_1 \|_{L^2} + \| \dot{\Delta}_k (\u \cdot \nabla n) \|_{L^2} \big)  + \| \dot{\Delta}_k \mathbb Q f_2 \|_{L^2}\nn\\
		&\quad \quad   + C\|\big([\dot{\Delta}_k, \u \cdot \nabla ] \Lambda n , \dot{\Delta}_k ( \Lambda \u \cdot \nabla n) , \dot{\Delta}_k \Lambda f_1  \big) \|_{L^2} + C\| \div \u\|_{L^{\infty}} \| \dot{\Delta}_k \Lambda n\|_{L^2} + 2^k \| \dot{\Delta}_k \tau \|_{L^2}.
	\end{align}
	It is easy to see that
	\begin{align*}
		C2^{-k} \| \dot{\Delta}_k n \|_{L^2} \le C2^{-2k_0} \| \Lambda
		 \dot{\Delta}_k n \|_{L^2},  \quad  C2^{-k} \| \dot{\Delta}_k f_1 \|_{L^2} \le C2^{-2k_0} \|
		 \dot{\Delta}_k \Lambda f_1 \|_{L^2},    \quad   \text{for all}\  k\ge k_0 -1.
	\end{align*}

	Choosing $ k_0$ is large enough and $ \varepsilon $ is sufficiently small such that $ \frac{\varepsilon}{2} \ge C2^{-2k_0}$ and  $ \frac{1}{2} \ge C\varepsilon $, we infer from \eqref{duw_k} that
	\begin{align}\label{duwk}
		&\frac{1}{2} \frac{d}{dt} ( 2\varepsilon \| \dot{\Delta}_k \Lambda n \|_{L^2} + \| \dot{\Delta}_k \G \|_{L^2} )  +  \varepsilon \|\Lambda \dot{\Delta}_k  n \|_{L^2} + 2^{2k} \| \dot{\Delta}_k \G \|_{L^2}\nn\\
		&\quad \le   C2^{-k}  \| \dot{\Delta}_k (\u \cdot \nabla n) \|_{L^2}   + \| \dot{\Delta}_k \mathbb Q f_2 \|_{L^2} + 2^k \| \dot{\Delta}_k \tau \|_{L^2}\nn\\
		&\quad \quad   + C\|\big([\dot{\Delta}_k, \u \cdot \nabla ] \Lambda n , \dot{\Delta}_k ( \Lambda \u \cdot \nabla n) , \dot{\Delta}_k \Lambda f_1  \big) \|_{L^2} + C\| \div \u\|_{L^{\infty}} \| \dot{\Delta}_k \Lambda n\|_{L^2}.
	\end{align}
	On the one hand, there holds
	\begin{align*}
		2\varepsilon \| \dot{\Delta}_k \Lambda n \|_{L^2} + \| \dot{\Delta}_k \G \|_{L^2} \lesssim \| \dot{\Delta}_k \Lambda n\|_{L^2} + \| \dot{\Delta}_k \mathbb{Q}\u \|_{L^2}.
	\end{align*}
	On the other hand, due to $ \frac{\varepsilon}{2} \ge C2^{-2k_0}$, there holds
	 \begin{align*}
	 	2\varepsilon \| \dot{\Delta}_k \Lambda n \|_{L^2} + \| \dot{\Delta}_k \G \|_{L^2} \gtrsim  \| \dot{\Delta}_k \Lambda n\|_{L^2} + \| \dot{\Delta}_k \mathbb{Q}\u \|_{L^2}.
	 \end{align*}

	Hence, we can get from the fact  $ \|\dot{\Delta}_k \delta \|_{L^2} \approx \|\dot{\Delta}_k \mathbb{Q} \u \|_{L^2} $ that
$$ 2\varepsilon \| \dot{\Delta}_k \Lambda n\|_{L^2} + \| \dot{\Delta}_k \G \|_{L^2} \approx \| \dot{\Delta}_k \Lambda n\|_{L^2} + \| \dot{\Delta}_k \delta \|_{L^2},$$
	 from which and multiplying the equation \eqref{duwk} by $ 2^{(\frac{d}{2}+1)k}$, taking the summation of $ k \ge k_0 $, and then integrating over $ [0, t]$, we can finally  get
	\begin{align}\label{le_uqt_h}
		&\| (\Lambda n, \delta )\|^h_{\widetilde{L}^{\infty}_{t}(\dot B^{\frac{d}{2}+1}_{2,1})}    + \|\Lambda n \|^h_{{L}^{1}_{t}(\dot B^{\frac{d}{2}+1}_{2,1})} + \|  \delta  \|^h_{{L}^{1}_{t}(\dot B^{\frac{d}{2} + 3}_{2,1})} \nn\\
		&\quad\lesssim \| (\Lambda n_0,\mathbb{Q}\u_0)\|^h_{\dot B^{\frac{d}{2}+1}_{2,1}} +\| \Lambda\tau \|^h_{{L}^1_{t} (\dot{B}^{\frac{d}{2} +1 }_{2,1} )} + \| f_1 \|^h_{{L}^1_{t} (\dot{B}^{\frac{d}{2} +2 }_{2,1} )}   + \| \big( f_2, \Lambda \u \cdot \nabla n  \big)  \|^h_{{L}^1_{t} (\dot{B}^{\frac{d}{2}+1}_{2,1} )}    \nn\\
		&\qquad +   \int_{0}^{t} \sum_{k \ge k_0 -1} 2^{(\frac{d}{2} +1)k} \Big(\| \div \u\|_{L^{\infty}} \| \dot{\Delta}_k \Lambda n\|_{L^2} + \| ([\dot{\Delta}_k, \u \cdot \nabla  ]\Lambda n) \|_{L^2}\Big) \,dt'.
	\end{align}
 Now,  applying Lemma \ref{embedding}  and \eqref{buxiu4} gives rise to
  		\begin{align}\label{last2_h}
  			\int_{0}^{t} \sum_{k \ge k_0 -1} 2^{(\frac{d}{2}+1)k} \| \div \u\|_{L^{\infty}} \| \dot{\Delta}_k \Lambda n\|_{L^2}\,dt' \lesssim 	\int_{0}^{t} \| \u\|_{\dot B^{\frac{d}{2} + 1}_{2,1}} \| \Lambda n \|^h_{\dot B^{\frac{d}{2} + 1}_{2,1}} \,dt' \lesssim \int_{0}^{t}\mathcal E_\infty(t')\mathcal E_1(t')\,dt'.
  		\end{align}
 For the last term in \eqref{le_uqt_h},
  it follows from Lemma \ref{commutator} directly that
		\begin{align}\label{last_h}
			\int_{0}^{t} \sum_{k \ge k_0 -1} 2^{(\frac{d}{2}+1)k} \|  [\dot{\Delta}_k, \u \cdot \nabla]\Lambda n \|_{L^2} \,dt'  \lesssim \int_{0}^{t}  \| \nabla \u\|_{\dot B^{\frac{d}{2} }_{2,1}} \| \la n\|_{\dot B^{\frac{d}{2} + 1}_{2,1}}  \,dt'  \lesssim \int_{0}^{t} \mathcal E_\infty(t')\mathcal E_1(t') \,dt'.
  		\end{align}
  		
  		By Lemma \ref{le_product} and \eqref{buxiu1}--\eqref{buxiu4}, we have
  		\begin{align}\label{duda_h}
  			\int_{0}^{t} \|  \la \u \cdot \nabla n \|^h_{\dot B^{\frac{d}{2} + 1}_{2,1}} \,dt' &\lesssim \int_{0}^{t}\big( \| \la \u \|_{L^{\infty}} \| \nabla n \|_{\dot B^{\frac{d}{2} + 1}_{2,1}} + \| \nabla n\|_{L^{\infty}} \| \la \u \|_{\dot B^{\frac{d}{2} + 1}_{2,1}} \big)\,dt' \nn\\
  			&\lesssim \int_{0}^{t} \big(\| \u \|_{\dot B^{\frac{d}{2} + 1}_{2,1}} \|\nabla n \|_{\dot B^{\frac{d}{2} + 1}_{2,1}} + \| \u \|_{\dot B^{\frac{d}{2} + 2}_{2,1}} \| n\|_{\dot B^{\frac{d}{2} + 1}_{2,1}}\big) \,dt'\nn\\
  			&\lesssim \int_{0}^{t} \mathcal E_\infty(t')\mathcal E_1(t')\, \,dt'.
  		\end{align}
		Similarly, for the terms in $f_1$, applying Lemmas \ref{le_product} and \ref{I(a)} and using  \eqref{buxiu1}--\eqref{eq:smalladd} gives
		\begin{align*}
			&\| (n\div \u, I(a)\div \u ) \|^h_{\dot{B}^{\frac{d}{2}+2}_{2,1} } \nn\\
			&\quad\lesssim  \| (n, I(a)) \|_{L^{\infty} } \| \div \u  \|_{\dot{B}^{\frac{d}{2}+2}_{2,1} } + \| (n, I(a)) \|_{\dot{B}^{\frac{d}{2}+2}_{2,1} } \| \div \u  \|_{L^{\infty} } \nn\\
			&\quad\lesssim\| (n, \eta)  \|_{\dot{B}^{\frac{d}{2}}_{2,1} } \|  \u  \|_{\dot{B}^{\frac{d}{2}+3}_{2,1} } + \| (n, \eta)  \|_{\dot{B}^{\frac{d}{2} +2}_{2,1} } \|  \u  \|_{\dot{B}^{\frac{d}{2}+1}_{2,1} }    \nn\\
			&\quad\lesssim \ea\eb,
		\end{align*}
		and
		\begin{align*}
			\|  \eta^2 \div \u  \|^h_{\dot{B}^{\frac{d}{2}+2}_{2,1} }
			\lesssim  &\| \eta^2 \|_{L^{\infty} } \| \div \u  \|_{\dot{B}^{\frac{d}{2}+2}_{2,1} } + \| \eta^2 \|_{\dot{B}^{\frac{d}{2}+2}_{2,1} } \| \div \u  \|_{L^{\infty} } \nn\\
			\lesssim &\| \eta  \|_{\dot{B}^{\frac{d}{2}}_{2,1} } \| \eta  \|_{\dot{B}^{\frac{d}{2}}_{2,1} } \|  \u  \|_{\dot{B}^{\frac{d}{2}+3}_{2,1} } +\| \eta  \|_{\dot{B}^{\frac{d}{2}}_{2,1} } \| \eta \|_{\dot{B}^{\frac{d}{2} +2}_{2,1} } \|  \u  \|_{\dot{B}^{\frac{d}{2}+1}_{2,1} }    \nn\\
			\lesssim&  (\ea)^2\eb.
		\end{align*}

		Hence, we get
		\begin{align}\label{first_h}
			\| f_1 \|^h_{{L}^1_{t} (\dot{B}^{\frac{d}{2} +2}_{2,1} )} \lesssim \int_{0}^{t} \big(1+\mathcal E_\infty(t')\big)\mathcal E_\infty(t')\mathcal E_1(t')\,dt'.
		\end{align}
		Finally, for the terms in $ \mathbb Q f_2$, we get by applying Lemmas \ref{bernstein}, \ref{le_product} and \eqref{buxiu1}--\eqref{buxiu4} that
		\begin{align}\label{udu_high}
		\|\u \cdot \nabla \u \|^h_{\dot B^{\frac{d}{2} + 1}_{2,1}} &\lesssim \| \u\|_{L^{\infty}} \| \nabla \u\|_{\dot B^{\frac{d}{2} + 1}_{2,1}} + \| \nabla \u\|_{L^{\infty}} \|  \u\|_{\dot B^{\frac{d}{2} + 1}_{2,1}}\lesssim \| \u\|_{\dot B^{\frac{d}{2} }_{2,1}}\| \u\|_{\dot B^{\frac{d}{2} +2}_{2,1}}\lesssim \ea\eb,
		\end{align}
		and
\begin{align}\label{kada_high}
		\|k(n,\eta) \nabla \eta \|^h_{\dot B^{\frac{d}{2} + 1}_{2,1}}
\lesssim& \| k(n,\eta)\|_{L^{\infty}} \| \nabla \eta\|_{\dot B^{\frac{d}{2} + 1}_{2,1}} + \| \nabla \eta\|_{L^{\infty}} \|  k(n,\eta)\|_{\dot B^{\frac{d}{2} + 1}_{2,1}}\nn\\
\lesssim& \| k(n,\eta)\|_{\dot B^{\frac{d}{2} }_{2,1}}\| \eta\|_{\dot B^{\frac{d}{2} +2}_{2,1}}+\| \eta\|_{\dot B^{\frac{d}{2} +1}_{2,1}}\| k(n,\eta)\|_{\dot B^{\frac{d}{2}+1 }_{2,1}}\nn\\
\lesssim& \ea\eb.
		\end{align}
Similarly,
	\begin{align}\label{Iaddu_high}
		\| k(n,\eta) (\Delta \u + \nabla \div \u) \|^h_{\dot B^{\frac{d}{2} + 1}_{2,1}} &\lesssim \| k(n,\eta) \|_{L^{\infty}} \| \u \|_{\dot B^{\frac{d}{2} + 3}_{2,1}} + \| k(n,\eta)\|_{\dot B^{\frac{d}{2} + 1}_{2,1}} \| \nabla^2 \u\|_{L^{\infty}} \nn\\
		&\lesssim \| (n, \eta)\|_{\dot B^{\frac{d}{2}}_{2,1}} \| \u \|_{\dot B^{\frac{d}{2} + 3}_{2,1}} + \| (n, \eta)\|_{\dot B^{\frac{d}{2} + 1}_{2,1}} \| \u \|_{\dot B^{\frac{d}{2} + 2}_{2,1}} \nn\\
		&\lesssim \ea\eb,
		\end{align}
		\begin{align}\label{Iadt_h}
			\| k(n,\eta) \div \tau \|^h_{\dot B^{\frac{d}{2} + 1}_{2,1}} &\lesssim \| k(n,\eta) \|_{L^{\infty}} \| \div \tau \|_{\dot B^{\frac{d}{2} + 1}_{2,1}} + \| k(n,\eta)\|_{\dot B^{\frac{d}{2} + 1}_{2,1}} \| \div \eta\|_{L^{\infty}} \nn\\
			&\lesssim \| (n,\eta) \|_{\dot B^{\frac{d}{2}}_{2,1}} \| \tau \|_{\dot B^{\frac{d}{2} + 2}_{2,1}} + \| (n,\eta) \|_{\dot B^{\frac{d}{2} + 1}_{2,1}} \| \tau \|_{\dot B^{\frac{d}{2} + 1}_{2,1}}\nn\\
			&\lesssim \ea\eb.
		\end{align}
		Together with \eqref{udu_high}--\eqref{Iadt_h}, we have
		\begin{align}\label{h2_h}
		\|    f_2 \|^h_{{L}^{1}_t ({\dot B^{\frac{d}{2} + 1}_{2,1}})} \lesssim \int_{0}^{t} \mathcal E_\infty(t')\mathcal E_1(t')\,dt'.
		\end{align}
		Regarding the high-frequency estimation of the incompressible part of velocity $\p \u$, analogous to the formulation in \eqref{eta_l}, we have
		\begin{align}\label{pu_lambda}
			&\|\p \u \|^h_{\widetilde{L}^{\infty}_{t}(\dot B^{\frac{d}{2}+1}_{2,1})}  + \|  \p\u  \|^h_{{L}^{1}_{t}(\dot B^{\frac{d}{2} + 3}_{2,1})}\lesssim \| \mathbb{P}\u_0\|^h_{\dot B^{\frac{d}{2}+1}_{2,1}} +\| \tau\|^h_{{L}^{1}_t ({\dot B^{\frac{d}{2} + 2}_{2,1}})} + \|   f_2 \|^h_{{L}^{1}_t ({\dot B^{\frac{d}{2} + 1}_{2,1}})}.
		\end{align}

	Consequently, 	inserting \eqref{last_h}, \eqref{last2_h}, \eqref{duda_h}, \eqref{first_h}, and \eqref{h2_h} into \eqref{le_uqt_h} and then simultaneously combining \eqref{pu_lambda}, we can arrive at \eqref{au_h},	which complete the proof of  Lemma \ref{le_aut_h}.
	\end{proof}
\end{lemma}
\subsubsection{The estimate of $\tau $  in the high frequency}
In this subsection, we are concerned with the
estimates of $\tau$ in the high frequency. More precisely, we have the following lemma.
\begin{lemma}\label{le_t_h}
	Under the condition in Theorem \ref{dingli1}, there exists a constant $ C$ such that
	\begin{align}\label{t_h}
	\| \Lambda\tau \|^h_{\widetilde{L}^{\infty}_{t}(\dot B^{\frac{d}{2}+1}_{2,1})} + \frac{A_0}{2\lambda_1}\| \Lambda\tau \|^h_{{L}^1_{t} (\dot{B}^{\frac{d}{2} +1 }_{2,1} )} \le  \| \Lambda\tau_0\|^h_{\dot B^{\frac{d}{2}+1}_{2,1}}   + C \int_{0}^{t} \mathcal E_\infty(t')\mathcal E_1(t')\, \,dt'.
	\end{align}
	\begin{proof}
		We  first get   by a similar derivation of \eqref{a_k} that
\begin{align*}
\partial_t \dot{\Delta}_k \Lambda \tau  + \frac{A_0}{2\lambda_1} \dot{\Delta}_k \Lambda \tau +  \alpha \u \cdot \nabla  \dot{\Delta}_k \Lambda \tau = - \alpha [\dot{\Delta}_k, \u \cdot \nabla ] \Lambda \tau - \alpha \dot{\Delta}_k ( \Lambda \u \cdot \nabla \tau) - \dot{\Delta}_k \Lambda f_3.
\end{align*}
		Multiplying the above equation by $ 2^{(\frac{d}{2}+1)k}$, taking the summation of $ k \ge k_0 $, and then integrating over $ [0, t]$, we can finally  get
		\begin{align}\label{t1_h}
		&\| \Lambda \tau\|^h_{\widetilde{L}^{\infty}_{t}(\dot B^{\frac{d}{2}+1}_{2,1})} + \frac{A_0}{2\lambda_1} \| \Lambda \tau \|^h_{{L}^1_{t} (\dot{B}^{\frac{d}{2} +1 }_{2,1} )}\nn\\ &\qquad \lesssim  \| \Lambda \tau_0 \|^h_{\dot B^{\frac{d}{2}+1}_{2,1}} + \| f_3 \|^h_{{L}^1_{t} (\dot{B}^{\frac{d}{2} +2 }_{2,1} )}   + \|  \Lambda \u \cdot \nabla \tau  \|^h_{{L}^1_{t} (\dot{B}^{\frac{d}{2}+1}_{2,1} )}    \nn\\
		&\qquad \  +  \int_{0}^{t} \sum_{k \ge k_0 -1} 2^{(\frac{d}{2} +1)k} \big(\| \div \u\|_{L^{\infty}} \| \dot{\Delta}_k \Lambda \tau\|_{L^2} + \| ([\dot{\Delta}_k, \u \cdot \nabla  ]\Lambda \tau) \|_{L^2}\big) \,dt'.
		\end{align}
		For the last term of \eqref{t1_h}, applying Lemmas \ref{commutator}, \ref{embedding} and using \eqref{buxiu1}-- \eqref{buxiu4} give rise to
		\begin{align}\label{t_last}
		& \sum_{k \ge k_0 -1} 2^{(\frac{d}{2} +1)k} \big(\| \div \u\|_{L^{\infty}} \| \dot{\Delta}_k \Lambda \tau\|_{L^2} + \| ([\dot{\Delta}_k, \u \cdot \nabla  ]\Lambda \tau) \|_{L^2}\big)\nn\\
		&\quad\lesssim \|  \u\|_{\dot B^{\frac{d}{2}+1 }_{2,1}} \| \tau \|_{\dot B^{\frac{d}{2} + 2}_{2,1}} \lesssim\ea\eb.
		\end{align}
We now deal with the term $\la \u \cdot \nabla \tau$. It follows from
		Lemma \ref{le_product} and \eqref{buxiu1}-- \eqref{buxiu4} that
		\begin{align}\label{dudt_h}
		 \|  \la \u \cdot \nabla \tau \|^h_{\dot B^{\frac{d}{2} + 1}_{2,1}}  &\lesssim \| \la \u \|_{L^{\infty}} \| \nabla \tau \|_{\dot B^{\frac{d}{2} + 1}_{2,1}} + \| \nabla \tau\|_{L^{\infty}} \| \la \u \|_{\dot B^{\frac{d}{2} + 1}_{2,1}}  \nn\\
		&\lesssim  \| \u \|_{\dot B^{\frac{d}{2} + 1}_{2,1}} \| \tau \|_{\dot B^{\frac{d}{2} + 2}_{2,1}} + \| \u \|_{\dot B^{\frac{d}{2} + 2}_{2,1}} \| \tau\|_{\dot B^{\frac{d}{2} + 1}_{2,1}}\nn\\
		&\lesssim \ea\eb .
		\end{align}
		Similarly, we have
		\begin{align}\label{N4_h}
		\| f_3 \|^h_{\dot B^{\frac{d}{2} + 2}_{2,1}}
& \lesssim \| \tau\|_{L^{\infty}} \| \u \|_{\dot B^{\frac{d}{2} + 3}_{2,1}} + \|  \tau\|_{\dot B^{\frac{d}{2} + 2}_{2,1}} \|   \nabla \u \|_{L^{\infty}} +  \| \eta\|_{L^{\infty}} \| \u \|_{\dot B^{\frac{d}{2} + 3}_{2,1}} + \|  \eta\|_{\dot B^{\frac{d}{2} + 2}_{2,1}} \|   \nabla \u \|_{L^{\infty}} \nn\\
		& \lesssim\| (\tau,\eta )\|_{\dot B^{\frac{d}{2}}_{2,1}}\| \u \|_{\dot B^{\frac{d}{2} + 3}_{2,1}} + \|  (\tau,\eta )\|_{\dot B^{\frac{d}{2} + 2}_{2,1}} \| \u \|_{\dot B^{\frac{d}{2}+1}_{2,1}} \nn\\
		&\lesssim \ea\eb .
		\end{align}
		Inserting \eqref{N4_h}, \eqref{dudt_h} and \eqref{eta_last} into \eqref{t1_h} gives rise to \eqref{t_h}. Consequently, we complete the proof of Lemma \ref{le_t_h}.
	\end{proof}
\end{lemma}
\subsubsection{The estimate of $\eta $  in the high frequency}
In this subsection, we are concerned with the
estimates of $\eta$ in the high frequency. More precisely, we have the following lemma.
\begin{lemma}\label{le_e_h}
		Under the condition in Theorem \ref{dingli1}, there exists a constant $ C$ such that
	\begin{align}\label{e_h}
	\| \Lambda\eta \|^h_{\widetilde{L}^{\infty}_{t}(\dot B^{\frac{d}{2}+1}_{2,1})}  \le  \| \Lambda\eta_0\|^h_{\dot B^{\frac{d}{2}+1}_{2,1}}   + C \int_{0}^{t} \eat\ebt\, \,dt'.
	\end{align}
	\begin{proof}
		We get   by a similar derivation of \eqref{t1_h} that
		\begin{align}\label{eta_h}
		\| \Lambda \eta\|^h_{\widetilde{L}^{\infty}_{t}(\dot B^{\frac{d}{2}+1}_{2,1})}
\lesssim&  \| \Lambda \eta_0 \|^h_{\dot B^{\frac{d}{2}+1}_{2,1}} + \| f_4 \|^h_{{L}^1_{t} (\dot{B}^{\frac{d}{2} +2 }_{2,1} )}   + \|  \Lambda \u \cdot \nabla \eta  \|^h_{{L}^1_{t} (\dot{B}^{\frac{d}{2}+1}_{2,1} )}    \nn\\
		& +   \int_{0}^{t} \sum_{k \ge k_0 -1} 2^{(\frac{d}{2} +1)k} \big(\| \div \u\|_{L^{\infty}} \| \dot{\Delta}_k \Lambda \eta\|_{L^2} + \| ([\dot{\Delta}_k, \u \cdot \nabla  ]\Lambda \eta) \|_{L^2}\big) \,dt'.
		\end{align}
		For the last two terms in  \eqref{eta_h}, we can deal with them the same as
\eqref{t_last} and \eqref{dudt_h} that
		\begin{align}\label{eta_last}
		\|  \la \u \cdot \nabla \eta \|^h_{\dot B^{\frac{d}{2} + 1}_{2,1}}
		\lesssim  \| \u \|_{\dot B^{\frac{d}{2} + 1}_{2,1}} \| \eta \|_{\dot B^{\frac{d}{2} + 2}_{2,1}} + \| \u \|_{\dot B^{\frac{d}{2} + 2}_{2,1}} \| \eta\|_{\dot B^{\frac{d}{2} + 1}_{2,1}}
		\lesssim& \ea\eb,\\
\sum_{k \ge k_0 -1} 2^{(\frac{d}{2} +1)k} \big(\| \div \u\|_{L^{\infty}} \| \dot{\Delta}_k \Lambda \eta\|_{L^2} + \| ([\dot{\Delta}_k, \u \cdot \nabla  ]\Lambda \eta) \|_{L^2}\big)
			\lesssim& \|  \u\|_{\dot B^{\frac{d}{2}+1 }_{2,1}} \| \eta \|_{\dot B^{\frac{d}{2} + 2}_{2,1}} \lesssim\ea\eb.\nn
		\end{align}
		For the term $f_4$,  we get from
		Lemma \ref{le_product} and \eqref{buxiu1}-- \eqref{buxiu4} that
\begin{align}\label{N5_h}
		\| f_4 \|^h_{\dot B^{\frac{d}{2} + 2}_{2,1}}
& \lesssim \| \eta\|_{L^{\infty}} \| \u \|_{\dot B^{\frac{d}{2} + 3}_{2,1}} + \|  \eta\|_{\dot B^{\frac{d}{2} + 2}_{2,1}} \|   \nabla \u \|_{L^{\infty}} \nn\\
		& \lesssim\| \eta \|_{\dot B^{\frac{d}{2}}_{2,1}}\| \u \|_{\dot B^{\frac{d}{2} + 3}_{2,1}} + \|  \eta \|_{\dot B^{\frac{d}{2} + 2}_{2,1}} \| \u \|_{\dot B^{\frac{d}{2}+1}_{2,1}} \nn\\
		&\lesssim \ea\eb .
		\end{align}

		Inserting  \eqref{eta_last} and \eqref{N5_h} into \eqref{eta_h} gives rise to \eqref{e_h}.
Consequently, we complete the proof of Lemma \ref{le_e_h}.
	\end{proof}
\end{lemma}

\subsubsection{The proof of global existence of Theorem \ref{dingli1}}
In this subsection, we use the continuity argument to complete the proof of global existence  of Theorem \ref{dingli1}.
Multiplying  \eqref{t_low}  by an large constant  and adding it to \eqref{aut_low}, \eqref{eta_low}, we can get the estimate of $(n,\u,\tau,\eta)$ in the low frequency:
\begin{align}\label{nuetlow}
& \| (n, \u,\eta)\|^\ell_{\widetilde{L}^{\infty}_{t}(\dot B^{\frac{d}{2} - 1}_{2,1})} +\| \tau \|^\ell_{\widetilde{L}^{\infty}_{t}(\dot B^{\frac{d}{2}}_{2,1})}  +  \| (n, \u)\|^\ell_{{L}^{1}_{t}(\dot B^{\frac{d}{2} + 1}_{2,1} )} + \frac{A_0}{2\lambda_1} \| \tau \|^\ell_{{L}^{1}_{t}(\dot B^{\frac{d}{2} }_{2,1} )}\nn\\
&\quad \le  \| (n_0, \u_0,\eta_0) \|^\ell_{\dot{B}^{\frac{d}{2}-1}_{2,1}} + \| \tau_0 \|^\ell_{\dot{B}^{\frac{d}{2}}_{2,1}} + C \int_{0}^{t} \eat\big( 1+ \eat\big)\ebt \,dt'.
\end{align}

Multiplying  \eqref{t_h}  by an large constant  and adding it to \eqref{au_h} and \eqref{e_h}, we can get the estimate of $(n,\u,\tau,\eta)$ in the high frequency:
\begin{align}\label{nuethigh}
	&\| (\Lambda n, \u,\Lambda \tau,\Lambda \eta)\|^h_{\widetilde{L}^{\infty}_{t}(\dot B^{\frac{d}{2}+1}_{2,1})}  + \| \Lambda n \|^h_{{L}^{1}_{t}(\dot B^{\frac{d}{2}+1}_{2,1})} + \|  \u \|^h_{{L}^{1}_{t}(\dot B^{\frac{d}{2} + 3}_{2,1})}+ \frac{A_0}{2\lambda_1}\| \Lambda\tau \|^h_{{L}^1_{t} (\dot{B}^{\frac{d}{2} +1 }_{2,1} )} \nn\\
	&\quad\le  \| (\Lambda n_0,\u_0,\Lambda \tau_0,\Lambda \eta_0)\|^h_{\dot B^{\frac{d}{2}+1}_{2,1}} +  C \int_{0}^{t}  \big( 1+\eat \big)\eat\ebt \, \,dt'.
	\end{align}
Now,  define $ \mathcal{E}(t)$ as
\begin{align*}
	&\mathcal{E}(t) \stackrel{\mathrm{def}}{=} \ea + \int_{0}^{t} \mathcal{E}_1(t') \,dt'.
\end{align*}

Then, we can deduce from \eqref{nuetlow} and \eqref{nuethigh} that
\begin{align}\label{energy}
	\mathcal{E}(t) \le C\mathcal{E}(0) + C\mathcal{E}^2(t) + C\mathcal{E}^3(t).
\end{align}
Under the assumption of Theorem \ref{dingli1}, there exists a constant $C_0$ such that $\mathcal E(0)\leq C_0 c_0$. Further supported by the local existence which can be  achieved through basic energy method, then there exists a positive time $T$ such that
\begin{equation}\label{re}
 \mathcal E(t) \leq 2 C_0c_0 , \quad  \forall \; t \in [0, T].
\end{equation}
Let $T^{*}$ be the maximum time of existence in Equation \eqref{re}. According to Equation \eqref{energy}, when $c_0$ is sufficiently small, the standard continuity method can extend $T^*$ to $T^{*} = \infty$. We omit the details here. Hence, we finish the proof of Theorem \ref{dingli1}. $\hspace{7cm}\square$

\subsection{The proof  of the decay rates of the solutions in Theorem \ref{dingli1}}
In this section, we shall follow the method used in \cite{xujiang2021}, \cite{ZL2021}  to get the decay rates of the solutions constructed in the previous section.  The proof mainly depends on the pure energy approach without the spectral analysis.
From Section 3.1 (see the derivation of \eqref{udelta_l}, \eqref{pu_l}, \eqref{t_l}, \eqref{eta_l}, \eqref{le_uqt_h},  \eqref{pu_lambda}, \eqref{t1_h} and \eqref{eta_h} for more details), we can  get the following low-frequency and  high-frequency estimates:
\begin{align}\label{fa1}
&\frac{d}{dt}\big(\|(n,\u )\|^{\ell}_{\dot{B}_{2,1}^{\frac{d}{2}-1}}
+\|\tau\|^\ell_{\dot{B}_{2,1}^{\frac{d}{2}}}
+\|(\Lambda n,\u,\Lambda \tau)\|^{h}_{\dot{B}_{2,1}^{\frac{d}{2}+1}}\big)
\nonumber\\
&\quad \quad+c_1\big(\|(n,\u)\|^\ell_{\dot{B}_{2,1}^{\frac{d}{2} +1}} +\|\tau\|^\ell_{\dot{B}_{2,1}^{\frac{d}{2}}}
+\|(\Lambda n,\Lambda \tau)\|^{h}_{\dot{B}_{2,1}^{\frac{d}{2}+1}}+\|\u\|^{h}_{\dot{B}_{2,1}^{\frac{d}{2} +3}}\big)\nonumber\\
&\quad\lesssim (1+\ea)\ea\big(\|(n,\u)\|^\ell_{\dot{B}_{2,1}^{\frac{d}{2} +1}} +\|\tau\|^\ell_{\dot{B}_{2,1}^{\frac{d}{2}}}
+\|(\Lambda n,\Lambda \tau)\|^{h}_{\dot{B}_{2,1}^{\frac{d}{2}+1}}+\|\u\|^{h}_{\dot{B}_{2,1}^{\frac{d}{2} +3}}\big).
\end{align}

By the proof of the global existence of Theorem \ref{dingli1},  we can choose $c_0$  small enough such that the following estimate holds:
\begin{align*}
  (1+\ea)\ea \le \frac{c_1}{2},
\end{align*}
from which and \eqref{fa1},  we find that
\begin{align}\label{fa3}
&\frac{d}{dt}\big(\|(n,\u )\|^{\ell}_{\dot{B}_{2,1}^{\frac{d}{2}-1}}
+\|\tau\|^\ell_{\dot{B}_{2,1}^{\frac{d}{2}}}
+\|(\Lambda n,\u,\Lambda \tau)\|^{h}_{\dot{B}_{2,1}^{\frac{d}{2}+1}}\big)
\nonumber\\
&\quad +\frac{c_1}{2}\big(\|(n,\u)\|^\ell_{\dot{B}_{2,1}^{\frac{d}{2} +1}} +\|\tau\|^\ell_{\dot{B}_{2,1}^{\frac{d}{2}}}
+\|(\Lambda n,\Lambda \tau)\|^{h}_{\dot{B}_{2,1}^{\frac{d}{2}+1}}+\|\u\|^{h}_{\dot{B}_{2,1}^{\frac{d}{2} +3}}\big)\le 0.
\end{align}

In order to derive the decay estimate for the solutions as stated in Theorem \ref{dingli1}, we need to obtain a Lyapunov-type differential inequality from \eqref{fa3}.
For any $\widetilde{\beta} > 0$, noticing the  small condition of the solution in \eqref{xiaonorm} and embedding relation at high frequencies, there holds
\begin{align}\label{fa4}
\|(\Lambda n,\Lambda \tau)\|^{h}_{\dot{B}_{2,1}^{\frac{d}{2}+1}}+\|\u\|^{h}_{\dot{B}_{2,1}^{\frac{d}{2} +3}} \ge C \big(\|(\Lambda n,\Lambda \tau)\|^{h}_{\dot{B}_{2,1}^{\frac{d}{2}+1}}+\|\u\|^{h}_{\dot{B}_{2,1}^{\frac{d}{2} +3}} \big)^{1+\widetilde{\beta}}.
\end{align}
Therefore, to establish the Lyapunov-type inequality of the solutions, it suffices to control the norm $\|(n,\u)\|^\ell_{\dot{B}_{2,1}^{\frac{d}{2}+1}}
+\|\tau\|^\ell_{\dot{B}_{2,1}^{\frac{d}{2}}}$  at low frequencies.
This process can be  obtained from
an interpolation inequality. Before exploiting the interpolation inequality, we need to control on the viability of the uniform bound
\begin{equation*}
\|(n,\u, \eta)(t,\cdot)\|^{\ell}_{\dot B^{-{s}}_{2,\infty}} + \|\tau(t,\cdot)\|^{\ell}_{\dot B^{-{s}+1}_{2,\infty}}\le C \,\,\mbox{ for any $1-\frac{d}{2} < {s} \le \frac{d}{2}$}.
\end{equation*}
This explains why we consider the propagation on the
regularity of the initial data with  negative   index in the following subsection.

\subsubsection{Propagation the regularity of the initial data with  negative   index}
In this subsection, we shall derive the following key proposition.

\begin{proposition}\label{propagate}
	Let $(n,\u,\tau,\eta) $ be the  solutions obtained in \eqref{xiaonorm}. For any $1-\frac{d}{2} < {s} \le \frac{d}{2},$ and $(n_0^{\ell},\u_0^{\ell}, \eta_0^\ell)\in {\dot{B}^{-{s}}_{2,\infty}}(\R^d),  \tau_0^{\ell}\in \dot{B}_{2,\infty}^{-{s}+1}(\R^d),$
	then there exists a constant $C_0>0$ depends on the norm of the initial data
	such that for all $t\geq0$,
	\begin{eqnarray}\label{fa5}
	\|(n,\u, \eta)(t,\cdot)\|^{\ell}_{\dot B^{-{s}}_{2,\infty}} + \|\tau(t,\cdot)\|^{\ell}_{\dot B^{-{s}+1}_{2,\infty}}\leq C_0.
	\end{eqnarray}
\end{proposition}
\begin{proof}
	At first,	considering the first three equations in \eqref{reformulate}, we can derive the following two inequalities through a deduction similar to \eqref{udelta_l},  \eqref{pu_l} that
	\begin{align}\label{nqu_prop}
	\| (n, \q \u ) \|^\ell_{\dot{B}^{-{s}}_{2,\infty}}   \lesssim  \| (n_0, \q \u_0) \|^\ell_{\dot{B}^{-{s}}_{2,\infty}} + \| \tau \|^\ell_{\widetilde{L}^{1}_{t}(\dot{B}^{-{s}+1}_{2,\infty})} +  \int_{0}^{t} \| (f_1, \mathbb Q f_2, \u \cdot \nabla n )\|^\ell_{\dot{B}^{-{s}}_{2,\infty}} \,dt',
	\end{align}
	and
	\begin{align}\label{pu_prop}
	\|\p \u\|^{\ell}_{\dot{B}^{-{s}}_{2,\infty}}
	\lesssim\|\p \u_0\|^{\ell}_{\dot{B}^{-{s}}_{2,\infty}} + \| \tau \|^\ell_{\widetilde{L}^{1}_{t}(\dot{B}^{-{s}+1}_{2,\infty})}
	+\int_0^t\|\mathbb P f_2\|^{\ell}_{\dot{B}^{-{s}}_{2,\infty}}\,dt'.
	\end{align}
	
	From the fourth and fifth equations of \eqref{reformulate}, we can derive results through a deduction similar to \eqref{eta_l} that
	\begin{align}\label{t_prop}
	\| \tau \|^\ell_{\dot{B}^{-{s}+1}_{2,\infty}} + \frac{A_0}{2\lambda_1} \| \tau \|^\ell_{\widetilde{L}^{1}_{t}(\dot{B}^{-{s}+1}_{2,\infty} )} &\lesssim  \| \tau_0 \|^\ell_{\dot{B}^{-{s}+1}_{2,\infty}} + \int_{0}^{t} \| \div \u \|_{L^{\infty}} \| \tau \|^\ell_{\dot{B}^{-{s}+1}_{2,\infty}} \,dt' \nn\\
	&\quad + \int_{0}^{t} \sup_{k\leq k_0} 2^{(-{s}+1)k} \| [\dot{\Delta}_k, \u \cdot \nabla ]\tau\|_{L^2} \,dt' + \int_{0}^{t} \| f_3 \|^\ell_{\dot{B}^{-{s}+1}_{2,\infty}} \,dt'
	\end{align}
	and
	\begin{align}\label{e_prop}
	\|\eta\|^{\ell}_{\dot{B}^{-{s}}_{2,\infty}}
	&\lesssim\|\eta_0\|^{\ell}_{\dot{B}^{-{s}}_{2,\infty}}
	+\int_{0}^{t} \| \div \u \|_{L^{\infty}} \| \eta \|^\ell_{\dot{B}^{-{s}}_{2,\infty}} \,dt' \nn\\
	&\quad + \int_{0}^{t} \sup_{k\leq k_0} 2^{-{s} k} \| [\dot{\Delta}_k, \u \cdot \nabla ]\eta\|_{L^2} \,dt' + \int_{0}^{t} \| f_4 \|^\ell_{\dot{B}^{-{s}}_{2,\infty}} \,dt'.
	\end{align}
	Multiplying \eqref{t_prop} by a large constant $C$ and adding it to \eqref{nqu_prop} and \eqref{pu_prop}, and subsequently combining \eqref{e_prop}, we can derive
	\begin{align}\label{nute_prop}
	&\|(n, \u, \eta) \|^{\ell}_{\dot{B}^{-{s}}_{2,\infty}}+\|\tau\|^{\ell}_{\dot{B}^{-{s}+1}_{2,\infty}}\nonumber\\
	&\quad\lesssim\|(n_0,\u_0, \eta_0)\|^{\ell}_{\dot{B}^{-{s}}_{2,\infty}}+\|\tau_0\|^{\ell}_{\dot{B}^{-{s}+1}_{2,\infty}}
	\nonumber\\&\quad\quad+\int_0^t\|(f_1,  f_2, f_4, \u \cdot \nabla n )\|^{\ell}_{\dot{B}^{-{s}}_{2,\infty}}\,dt' +\int_0^t\|f_3\|^{\ell}_{\dot{B}^{-{s}+1}_{2,\infty}}\,dt'
	\nonumber\\&\quad\quad + \int_{0}^{t} \| \div \u \|_{L^{\infty}} \| \tau \|^\ell_{\dot{B}^{-{s}+1}_{2,\infty}} \,dt' + \int_{0}^{t} \sup_{k\leq k_0} 2^{(-{s}+1)k} \| [\dot{\Delta}_k, \u \cdot \nabla ]\tau\|_{L^2} \,dt'\nn\\
	&\quad \quad +\int_{0}^{t} \| \div \u \|_{L^{\infty}} \| \eta \|^\ell_{\dot{B}^{-{s}}_{2,\infty}} \,dt' + \int_{0}^{t} \sup_{k\leq k_0} 2^{-{s} k} \| [\dot{\Delta}_k, \u \cdot \nabla ]\eta\|_{L^2} \,dt'.
	\end{align}
	
	In order to bound the terms on the right-hand side of the above inequality, we leverage the product estimate in Besov spaces which can be found in \cite{ZL2021}.
	\begin{align}
	&\|fg\|_{\dot{B}_{2,\infty}^{-{s}}}\!\!\lesssim \|f\|_{\dot{B}_{2,\infty}^{-{s}}}\|g\|_{\dot{B}_{2,1}^{\frac{d}{2}}},\quad -\frac{d}{2}\le -{s}<\frac{d}{2};\label{key0}\\
	&\|f g^\ell\|_{\dot B^{-{s}}_{2,\infty}}^\ell\!\! \lesssim\|f\|_{\dot B^{\frac{d}{2} -1}_{2,1}}\|g^\ell\|_{\dot B^{-{s}+1}_{2,\infty}},\quad -\frac{d}{2}\le-{s}<\frac{d}{2} -1;
	\label{key1}\\
	&\|f g^h\|_{\dot B^{-{s}}_{2,\infty}}^\ell \!\!\lesssim\|f\|_{\dot B^{\frac{d}{2} -1}_{2,1}}\|g\|^h_{\dot B^{\frac{d}{2}}_{2,1}},\quad -\frac{d}{2}\le-{s}<\frac{d}{2}-1.\label{key2}
	\end{align}
For the term $ \|\u\cdot\nabla n\|_{\dot B^{-{s}}_{2,\infty}}^\ell $,  we first use the decompositions $n=n^\ell+n^h$ to write
	\begin{align}\label{uda}
	\|\u\cdot\nabla n\|_{\dot B^{-{s}}_{2,\infty}}^\ell \lesssim \|\u\cdot\nabla n^\ell\|_{\dot B^{-{s}}_{2,\infty}}^\ell + \|\u\cdot\nabla n^h\|_{\dot B^{-{s}}_{2,\infty}}^\ell .
	\end{align}
	In view of \eqref{key0}, \eqref{key2} and Bernstein's inequality, one gets that
	\begin{align}\label{besI1}
	\|\u\cdot\nabla n^\ell\|_{\dot B^{-{s}}_{2,\infty}}^\ell &\lesssim \| \u\|^\ell_{\dot{B}^{-{s}}_{2,\infty}} \|\nabla n \|^\ell_{\dot{B}^{\frac{d}{2}}_{2,1}} + \|\u \|^h_{\dot{B}^{\frac{d}{2}}_{2,1}} \|\nabla n \|^\ell_{\dot{B}^{\frac{d}{2} -1}_{2,1}} \nn\\
	&\lesssim \| \u\|^\ell_{\dot{B}^{-{s}}_{2,\infty}} \| n \|^\ell_{\dot{B}^{\frac{d}{2}+1}_{2,1}} + \|\u \|^h_{\dot{B}^{\frac{d}{2} +3}_{2,1}} \| n \|^\ell_{\dot{B}^{\frac{d}{2}-1}_{2,1}}\nn\\
	&\lesssim \eb \| \u\|^\ell_{\dot{B}^{-{s}}_{2,\infty}} + \ea\eb.
	\end{align}
	Similarly,
	\begin{align*}
	\|\u\cdot\nabla n^h\|_{\dot B^{-{s}}_{2,\infty}}^\ell &\lesssim \| \u\|^\ell_{\dot{B}^{-{s}}_{2,\infty}} \|\nabla n \|^h_{\dot{B}^{\frac{d}{2} }_{2,1}} + \|\u \|^h_{\dot{B}^{\frac{d}{2}}_{2,1}} \|\nabla n\|^h_{\dot{B}^{\frac{d}{2} -1}_{2,1}} \nn\\
	&\lesssim \eb \| \u\|^\ell_{\dot{B}^{-{s}}_{2,\infty}} + \ea\eb,
	\end{align*}
	from which and \eqref{besI1} gives rise to
	\begin{align}\label{uda_prop}
	\|\u\cdot\nabla n\|_{\dot B^{-{s}}_{2,\infty}}^\ell \lesssim \eb \| \u\|^\ell_{\dot{B}^{-{s}}_{2,\infty}} + \ea\eb.
	\end{align}
	For the second term of $ f_1$ and the term $ f_4$, it follows from \eqref{key0} that
	\begin{align}\label{ndu_prop}
	\|(n\div \u, \eta \div \u)\|^{\ell}_{\dot{B}^{-{s}}_{2,\infty}}&\lesssim \|n\|_{\dot{B}^{-{s}}_{2,\infty}} \|\nabla \u \|_{\dot{B}^{\frac{d}{2}}_{2,1}} + \|\eta\|_{\dot{B}^{-{s}}_{2,\infty}} \|\nabla \u \|_{\dot{B}^{\frac{d}{2}}_{2,1}} \nn\\
	&\lesssim \big(\|(n, \eta)\|^\ell_{\dot{B}^{-{s}}_{2,\infty}} + \|(n, \eta) \|^h_{\dot{B}^{-{s}}_{2,\infty}} \big) \|\nabla \u \|_{\dot{B}^{\frac{d}{2}}_{2,1}}  \nn\\
	&\lesssim \big(\|\u \|^\ell_{\dot{B}^{\frac{d}{2} +1}_{2,1}} + \| \u \|^h_{\dot{B}^{\frac{d}{2} +3 }_{2,1}}\big) \big(\|(n, \eta)\|^\ell_{\dot{B}^{-{s}}_{2,\infty}} + \|(\Lambda n ,\Lambda \eta) \|^h_{\dot{B}^{\frac{d}{2} +1}_{2,1}} \big) \nn\\
	&\lesssim \eb \| (n, \eta)\|^\ell_{\dot{B}^{-{s}}_{2,\infty}} + \ea\eb.
	\end{align}
	So, one has
	\begin{align}\label{h5_prop}
		\|f_4\|^\ell_{\dot B^{-{s}}_{2,\infty}} \lesssim	\eb \|  \eta\|^\ell_{\dot{B}^{-{s}}_{2,\infty}} + \ea\eb.
	\end{align}
We now turn to bound the first term in $f_1$. Recalling \eqref{ine}, we can get
\begin{align}\label{dgaga}
&\|I(a) \div \u)\|^{\ell}_{\dot{B}^{-{s}}_{2,\infty}}
\lesssim\|(n \div \u, \eta  \div \u)\|^{\ell}_{\dot{B}^{-{s}}_{2,\infty}}+\|\eta^2 \div \u\|^{\ell}_{\dot{B}^{-{s}}_{2,\infty}}.
\end{align}
Noticing \eqref{ndu_prop}, we only need to bound the last term in \eqref{dgaga}. Thanks to \eqref{key0} and \eqref{buxiu1}, \eqref{buxiu3} there holds
	\begin{align}\label{nIndu_prop}
	\|\eta^2 \div \u\|^{\ell}_{\dot{B}^{-{s}}_{2,\infty}}
	\lesssim&  \|\eta\|_{\dot{B}^{-{s}}_{2,\infty}} \|\eta \nabla \u \|_{\dot{B}^{\frac{d}{2}}_{2,1}} \nn\\
	\lesssim &\| \eta\|_{\dot{B}^{-{s}}_{2,\infty}}  \| \eta \|_{\dot{B}^{\frac{d}{2}}_{2,1}} \| \u \|_{\dot{B}^{\frac{d}{2}+1}_{2,1}}  \nn\\
	\lesssim &\ea \eb \| (n, \eta)\|^\ell_{\dot{B}^{-{s}}_{2,\infty}},
	\end{align}
from which and
	 \eqref{ndu_prop}, we have
	\begin{align}\label{Indu_prop}
	\|I(a)\div \u\|^\ell_{\dot B^{-{s}}_{2,\infty}} \lesssim \big(1+\ea\big) \eb \| (n, \eta)\|^\ell_{\dot{B}^{-{s}}_{2,\infty}} + \ea\eb.
	\end{align}
Hence, the combination of \eqref{ndu_prop} and \eqref{Indu_prop} gives rise to
	\begin{align}\label{h1_prop}
	\|f_1\|^\ell_{\dot B^{-{s}}_{2,\infty}} \lesssim \big(1+\ea\big) \eb \| (n, \eta)\|^\ell_{\dot{B}^{-{s}}_{2,\infty}} + \ea\eb.
	\end{align}
	Next, we deal with the terms of $ \mathbb Q f_2$. The term $\|\u\cdot\nabla \u\|_{\dot B^{-{s}}_{2,\infty}}^\ell $ can be bounded the same as \eqref{uda_prop} such that
	\begin{align}\label{udu_prop}
	\|\u\cdot\nabla \u\|_{\dot B^{-{s}}_{2,\infty}}^\ell \lesssim \eb \| \u\|^\ell_{\dot{B}^{-{s}}_{2,\infty}} + \ea\eb.
	\end{align}
	For the term $k(a)(\Delta \u + \nabla \div \u)$, one can infer from \eqref{ine}  that
\begin{align}\label{agea}
	\|k(a)(\Delta \u + \nabla \div \u)\|_{\dot B^{-{s}}_{2,\infty}}^\ell
\lesssim \|(n \nabla^2 \u, \eta \nabla^2 \u) \|_{\dot B^{-{s}}_{2,\infty}}^\ell+ \|(n \widetilde{k}_1(n, \eta)\nabla^2 \u, \eta \widetilde{k}_2(n, \eta)\nabla^2 \u) \|_{\dot B^{-{s}}_{2,\infty}}^\ell .
	\end{align}
 Thanks to \eqref{key0}, \eqref{key2} and \eqref{buxiu1}--\eqref{buxiu4}, there holds
	\begin{align}\label{addu_prop}
	\|(n \nabla^2 \u, \eta \nabla^2 \u) \|_{\dot B^{-{s}}_{2,\infty}}^\ell
	&\lesssim \|(n^\ell \nabla^2 \u, \eta^\ell \nabla^2 \u) \|_{\dot B^{-{s}}_{2,\infty}}^\ell + \|(n^h \nabla^2 \u, \eta^h \nabla^2 \u )\|_{\dot B^{-{s}}_{2,\infty}}^\ell \nn\\
	& \lesssim \|(n, \eta) \|_{\dot B^{-{s}}_{2,\infty}}^\ell \|\nabla ^2 \u\|_{\dot{B}^{\frac{d}{2}}_{2,1}} +  \|(n, \eta)\|^h_{\dot{B}^{\f d2}_{2,1}} \|\nabla^2 \u\|_{\dot B^{\frac{d}{2}-1}_{2,1}}\nn\\
	& \lesssim (\| \u\|^\ell_{\dot{B}^{\frac{d}{2}+1}_{2,1}} + \| \u\|^h_{\dot{B}^{\frac{d}{2} +3}_{2,1}}) (\|(n, \eta) \|_{\dot B^{-{s}}_{2,\infty}}^\ell + \|\Lambda n \|^h_{\dot{B}^{\frac{d}{2} +1}_{2,1}} + \|\eta \|^h_{\dot{B}^{\frac{d}{2} +2}_{2,1}}) \nn\\
	& \lesssim \eb \| (n, \eta)\|_{\dot B^{-{s}}_{2,\infty}}^\ell + \ea\eb.
	\end{align}
	According to \eqref{key1}, \eqref{key2} and Lemma \ref{le2.5}, one has
	\begin{align*}
	&\|(n \widetilde{k}_1(n, \eta)\nabla^2 \u, \eta \widetilde{k}_2(n, \eta)\nabla^2 \u) \|_{\dot B^{-{s}}_{2,\infty}}^\ell \nn\\
	&\quad\lesssim \| (n,\eta)^\ell \widetilde{k}_1(n, \eta)\nabla^2 \u \|_{\dot B^{-{s}}_{2,\infty}}^\ell  +\| (n,\eta)^h \widetilde{k}_2(n, \eta) \nabla^2 \u \|_{\dot B^{-{s}}_{2,\infty}}^\ell\nn\\
	&\quad \lesssim \| (n,\eta)\|_{\dot B^{-{s}+1}_{2,\infty}}^\ell \| \widetilde{k}_1(n, \eta) \nabla^2 \u\|_{\dot{B}^{\frac{d}{2} -1}_{2,1}} + \|(n,\eta)\|^h_{\dot{B}^{\frac{d}{2}}_{2,1}} \| \widetilde{k}_2(n, \eta)\nabla^2 \u\|_{\dot{B}^{\frac{d}{2} -1}_{2,1}}\nn\\
	&\quad \lesssim \| (n, \eta)\|_{\dot B^{-{s}}_{2,\infty}}^\ell \| \widetilde{k}_1(n, \eta)\|_{\dot{B}^{\frac{d}{2}}_{2,1}} \| \nabla^2 \u\|_{\dot{B}^{\frac{d}{2} -1}_{2,1}} +  \|(n, \eta)\|^h_{\dot{B}^{\frac{d}{2}}_{2,1}} \| \widetilde{k}_1(n, \eta)\|_{\dot{B}^{\frac{d}{2}}_{2,1}} \| \nabla^2 \u\|_{\dot{B}^{\frac{d}{2} -1}_{2,1}}\nn\\
	&\quad \lesssim \big(\| (n, \eta)\|_{\dot B^{-{s}}_{2,\infty}}^\ell + \|\Lambda n \|^h_{\dot{B}^{\frac{d}{2} +1}_{2,1}} + \|\eta \|^h_{\dot{B}^{\frac{d}{2} +2}_{2,1}}\big)\nn\\
 &\qquad\times\big( \|(n, \eta)\|^\ell_{\dot{B}^{\frac{d}{2}-1}_{2,1}} + \|\Lambda n \|^h_{\dot{B}^{\frac{d}{2} +1}_{2,1}} + \|\eta \|^h_{\dot{B}^{\frac{d}{2} +2}_{2,1}}\big) \big(\| \u\|^\ell_{\dot{B}^{\frac{d}{2}+1}_{2,1}} + \| \u\|^h_{\dot{B}^{\frac{d}{2} +3}_{2,1}} \big)\nn\\
	&\quad \lesssim \ea\eb \| (n,\eta)\|_{\dot B^{-{s}}_{2,\infty}}^\ell + \ea\ea\eb,
	\end{align*}
	from which and \eqref{addu_prop},  we can get that
	\begin{align}\label{Iaddu_prop}
	\|k(a)(\Delta \u + \nabla \div \u) \|_{\dot B^{-{s}}_{2,\infty}}^\ell \lesssim (1+\ea)\eb \| (n,\eta)\|_{\dot B^{-{s}}_{2,\infty}}^\ell  + (1+\ea)\ea\eb.
	\end{align}
	Similarly, we have
	\begin{align}\label{Iada_prop}
	\|k(n, \eta)\nabla n \|_{\dot B^{-{s}}_{2,\infty}}^\ell \lesssim (1+\ea)\eb \| (n,\eta)\|_{\dot B^{-{s}}_{2,\infty}}^\ell  + (1+\ea)\ea\eb
	\end{align}
	and
	\begin{align}\label{Iadt_prop}
	\|k(n, \eta)\div \tau \|_{\dot B^{-{s}}_{2,\infty}}^\ell \lesssim (1+\ea)\eb \| (n,\eta)\|_{\dot B^{-{s}}_{2,\infty}}^\ell  + (1+\ea)\ea\eb.
	\end{align}

	Collecting \eqref{udu_prop}, \eqref{Iaddu_prop}, \eqref{Iada_prop} and \eqref{Iadt_prop} together, we arrive at
	\begin{align}\label{h2_prop}
	 \| f_2\|^{\ell}_{\dot{B}^{-{s}}_{2,\infty}}
 \lesssim (1+\ea)\eb \| (a, \u, \eta)\|_{\dot B^{-{s}}_{2,\infty}}^\ell  + (1+\ea)\ea\eb.
	\end{align}
	For the second term of $ f_3$,  it follows from \eqref{ndu_prop} that
	\begin{align}\label{edu_prop}
		\| \eta (\nabla \u + (\nabla \u)^{\top} ) \|^{\ell}_{\dot{B}^{-{s}+1}_{2,\infty}} \lesssim \eb \| \eta\|^\ell_{\dot{B}^{-{s}}_{2,\infty}} + \ea\eb.
	\end{align}
For the rest two term of $ f_3$,
 it follows from \eqref{key0} that
	\begin{align}\label{tdu_prop}
		\| \nabla \u \tau + \tau (\nabla \u)^{\top} + \tau \div \u \|^{\ell}_{\dot{B}^{-{s}+1}_{2,\infty}} &\lesssim \| \nabla \u \tau + \tau (\nabla \u)^{\top} + \tau \div \u \|^{\ell}_{\dot{B}^{ -{s}}_{2,\infty}} \lesssim \| \nabla \u \|_{\dot B^{-{s}}_{2,\infty}} \| \tau\|_{\dot{B}^{\frac{d}{2}}_{2,1}}\nn\\
		&\lesssim (\| \tau\|^\ell_{\dot{B}^{\frac{d}{2}}_{2,1}} + \| \tau\|^h_{\dot{B}^{\frac{d}{2}+2}_{2,1}} ) (\|  \u \|^\ell_{\dot B^{-{s}}_{2,\infty}} + \| \u\|^h_{\dot{B}^{\frac{d}{2}+1}_{2,1}} )\nn\\
		& \lesssim \eb \|  \u \|^\ell_{\dot B^{-{s}}_{2,\infty}} + \ea\eb.
	\end{align}
	Hence, the combination of \eqref{tdu_prop} and \eqref{edu_prop} yields that
	\begin{align}\label{h4_prop}
	\| f_3\|^\ell_{\dot B^{-{s}+1 }_{2,\infty}} \lesssim \eb \|(\u, \eta)\|^{\ell}_{\dot{B}^{-{s}}_{2,\infty}}   + \ea\eb.
	\end{align}
	Finally, we still need to deal with the remaining terms in \eqref{nute_prop}. By Lemma \ref{commutator}, one can get
	\begin{align}\label{te_commutaor_prop}
		 &\| \div \u \|_{L^{\infty}} (\| \tau \|^\ell_{\dot{B}^{-{s}+1}_{2,\infty}} + \| \eta \|^\ell_{\dot{B}^{-{s}}_{2,\infty}}) +  \sup_{k\leq k_0} 2^{(-{s}+1)k} \| [\dot{\Delta}_k, \u \cdot \nabla ]\tau\|_{L^2}    +  \sup_{k\leq k_0} 2^{-{s} k} \| [\dot{\Delta}_k, \u \cdot \nabla ]\eta\|_{L^2}\nn\\
		 &\quad\lesssim (\| \tau \|_{\dot{B}^{-{s}+1}_{2,\infty}} + \| \eta \|_{\dot{B}^{-{s}}_{2,\infty}}) \|\nabla \u \|_{\dot{B}^{\frac{d}{2}}_{2,1}} \nn\\
		 &\quad \lesssim (\| \tau \|^\ell_{\dot{B}^{-{s}+1}_{2,\infty}} + \| \eta \|^\ell_{\dot{B}^{-{s}}_{2,\infty}} + \| (\tau, \eta) \|^h_{\dot{B}^{\frac{d}{2}+2}_{2,1}}  ) (\| \u\|^\ell_{\dot{B}^{\frac{d}{2}+1}_{2,1}} + \| \u\|^h_{\dot{B}^{\frac{d}{2} +3}_{2,1}}  )\nn\\
		 &\quad \lesssim \eb (\| \tau \|^\ell_{\dot{B}^{-{s}+1}_{2,\infty}} + \| \eta \|^\ell_{\dot{B}^{-{s}}_{2,\infty}}) + \ea\eb.
	\end{align}
	Inserting \eqref{uda_prop}, \eqref{h5_prop}, \eqref{h1_prop}, \eqref{h2_prop}, \eqref{h4_prop},  and \eqref{te_commutaor_prop} into \eqref{nute_prop}, we can get
	\begin{align}\label{nute}
	&\|(n,\u, \eta) \|^{\ell}_{\dot{B}^{-{s}}_{2,\infty}}+\|\tau\|^{\ell}_{\dot{B}^{-{s}+1}_{2,\infty}}\nonumber\\
	&\quad\lesssim\|(n_0,\u_0, \eta_0)\|^{\ell}_{\dot{B}^{-{s}}_{2,\infty}}+\|\tau_0\|^{\ell}_{\dot{B}^{-{s}+1}_{2,\infty}}+\int_{0}^{t} \big(1+\eat\big)\eat\ebt \,dt'\nn\\
	&\quad \quad +  \int_{0}^{t}
	\big(1+\eat\big)\ebt (\|(n^{\ell},\u^{\ell}, \eta^\ell)\|_{\dot{B}^{-{s}}_{2,\infty}}
	+\|\tau^{\ell}\|_{\dot{B}^{-{s}+1}_{2,\infty}})\,dt' .
	\end{align}
		Thus, in view of the global solutions obtained in Theorem \ref{dingli1}, and applying the Gronwall inequality to \eqref{nute} gives rise to
	\begin{eqnarray}\label{prop}
	\|(n,\u,\eta)(t,\cdot)\|^{\ell}_{\dot B^{-{s}}_{2,\infty}}+\|\tau(t,\cdot)\|^{\ell}_{\dot B^{-{s}+1}_{2,\infty}}\leq \bar{C}.
	\end{eqnarray}
	for all $t\geq0$, where $\bar{C}>0$ depends on the norm of the initial data.
	
	Consequently, we complete the proof of Proposition \ref{propagate}.
\end{proof}

\subsubsection{Lyapunov-type differential inequality}
In this section, we establish the Lyapunov-type inequality in time for energy norms, which leads
to the time-decay estimates. On the one hand, for any $-\frac d2\le{s}<{\frac{d}{2}-1}$,  we get by interpolation inequality that
\begin{align*}
\|(n,\u)\|^\ell_{\dot{B}_{2,1}^{\frac{d}{2}-1}}
\le& C \big(\|(n,\u)\|^\ell_{\dot{B}_{2,\infty}^{-{s}}}\big)^{\theta_{1}}
\big(\|(n,\u)\|^\ell_{\dot{B}_{2,1}^{\frac{d}{2}+1}}\big)^{1-\theta_{1}},
\quad \theta_1=\frac{4}{d+2{s}+2}\in(0,1),
\end{align*}
this together with Proposition \ref{propagate}  implies
there exists a constant $\widetilde{C}$ such that
\begin{align}\label{sa51}
\|(n,\u)\|^\ell_{\dot{B}_{2,1}^{ \frac{d}{2} +1}}\ge  \widetilde{C}\big(\|(a,\u)\|^\ell_{\dot{B}_{2,1}^{ \frac{d}{2}-1}}\big)^{\frac{1}{1-\theta_{1}}}.
\end{align}
On the other hand, by using  smallness condition of $\|\tau\|^\ell_{\dot{B}_{2,1}^{\frac{d}{2}}}$ obtained in \eqref{xiaonorm}, there holds
\begin{align}\label{sa54}
\|\tau\|^\ell_{\dot{B}_{2,1}^{\frac{d}{2}}}\ge \widetilde{C}\big(\|\tau\|^\ell_{\dot{B}_{2,1}^{\frac{d}{2}}}\big)^{\frac{1}{1-\theta_{1}}}.
\end{align}
Thus,  setting $\widetilde{\beta} = \frac{\ta_1}{1-\ta_1} >0$ in \eqref{fa4} and combining  with \eqref{sa51} and \eqref{sa54}, we conclude that there exists a constant $\widetilde{C} > 0$ such that the following Lyapunov-type inequality in time holds
\begin{align}\label{sa59}
&\frac{d}{dt}\big(\|(n,\u )\|^{\ell}_{\dot{B}_{2,1}^{\frac{d}{2}-1}}
+\|\tau\|^\ell_{\dot{B}_{2,1}^{\frac{d}{2}}}+\|(\Lambda n,\u,\Lambda \tau )\|^{h}_{\dot{B}_{2,1}^{\frac{d}{2} +1}}
\big)
\nonumber\\
&\quad +\widetilde{C}\big(\|(n,\u )\|^{\ell}_{\dot{B}_{2,1}^{\frac{d}{2}-1}}
+\|\tau\|^\ell_{\dot{B}_{2,1}^{\frac{d}{2}}}+\|(\Lambda n,\u,\Lambda \tau )\|^{h}_{\dot{B}_{2,1}^{\frac{d}{2} +1}}\big)^{1+\frac{4}{d+2{s}-2}}\le 0.
\end{align}
\subsubsection{Decay estimate}
Solving the differential inequality  obtained in \eqref{sa59} directly, we obtain
\begin{align}\label{decay}
\|(n,\u )\|^{\ell}_{\dot{B}_{2,1}^{\frac{d}{2}-1}}
+\|\tau\|^\ell_{\dot{B}_{2,1}^{\frac{d}{2}}}+\|(\Lambda n,\u,\Lambda \tau )\|^{h}_{\dot{B}_{2,1}^{\frac{d}{2} +1}}
\le C(1+t)^{-\frac{d+2{s}-2}{4}}.
\end{align}
For any $ - {s} < \beta_1 < \frac{d}{2} -1,$ by  using Lemma \ref{embedding} and Bernstein's inequality, we have
\begin{align*}
\|(n,\u)\|^\ell_{\dot{B}_{2,1}^{\beta_1}}\le C\big(\|(n,\u)\|^\ell_{\dot{B}_{2,\infty}^{-{s}}}\big)^{\kappa_{1}} \big(\|(n,\u)\|^\ell_{\dot{B}_{2,1}^{\frac{d}{2}-1}}\big)^{1-\kappa_{1}},\quad \kappa_{1}=\frac{d -2-2\beta_1}{d -2+2{s}}\in (0,1),
\end{align*}
from  which and Proposition \ref{propagate} gives rise to
\begin{align}\label{sa60}
\|(n,\u)\|^\ell_{\dot{B}_{2,1}^{\beta_1}}
\le C(1+t)^{  - \frac{\beta_1+{s}}{2}}.
\end{align}
Since $\beta_1< \frac{d}{2} -1,$ we see that
$$\|\Lambda n\|^h_{\dot{B}_{2,1}^{\beta_1-1}}+\|\u\|^h_{\dot{B}_{2,1}^{\beta_1}}\le C \|(\Lambda n,\u)\|^{h}_{\dot{B}_{2,1}^{\frac{d}{2} +1}}
 \le C(1+t)^{-\frac{d+2{s}-2}{4}},
$$
from which and \eqref{sa60} gives
\begin{align*}
\|(n,\u)\|_{\dot{B}_{2,1}^{\beta_1}}=&\|\Lambda n\|_{\dot{B}_{2,1}^{\beta_1-1}}+\|\u\|_{\dot{B}_{2,1}^{\beta_1}}\nonumber\\
\le& C(1+t)^{  - \frac{\beta_1+{s}}{2}}+C(1+t)^{-\frac{d+2{s}-2}{4}}\nonumber\\
\le& C(1+t)^{  - \frac{\beta_1+{s}}{2}}.
\end{align*}
Based on the embedding relationship of $\dot{B}^{0}_{2,1}(\R^d)\hookrightarrow L^2(\R^d)$, we can obtain
\begin{align*}
\|\Lambda^{\beta_1} (n,\u)\|_{L^2}\le C \|(n,\u)\|_{\dot{B}_{2,1}^{\beta_1}}
\le& C(1+t)^{  - \frac{\beta_1+{s}}{2}}.
\end{align*}
For any $1 - {s}<\beta_2<\frac{d}{2},$ by the interpolation inequality and Bernstein's inequality, we have
\begin{align*}
\|\tau\|^\ell_{\dot{B}_{2,1}^{\beta_2}}
\le C\big(\|\tau\|^\ell_{\dot{B}_{2,\infty}^{-{s}+1}}\big)^{\kappa_{2}} \big(\|\tau\|^\ell_{\dot{B}_{2,1}^{\frac{d}{2}}}\big)^{1-\kappa_{2}},\quad \kappa_{2}=\frac{d -2\beta_2}{d + 2{s} -2}\in (0,1),
\end{align*}
which  gives
\begin{align}\label{sa61+3}
\|\tau\|^\ell_{\dot{B}_{2,1}^{\beta_2}}
\le C(1+t)^{ - \frac{\beta_2 + {s} -1}{2}}.
\end{align}
In  view of $\beta_2<\frac{d}{2},$ we can get in the high frequency that
\begin{align}\label{sa61+61+3}
\|\tau\|^{h}_{\dot{B}_{2,1}^{\beta_2}}\le C\|\Lambda\tau\|^{h}_{\dot{B}_{2,1}^{\frac{d}{2}+1}}\le C(1+t)^{-\frac{d+2{s}-2}{4}}.
\end{align}
From \eqref{sa61+3} and \eqref{sa61+61+3}, we have
\begin{align*}
\|\tau\|_{\dot{B}_{2,1}^{\beta_2}}
&\le C(1+t)^{ - \frac{\beta_2 + {s} -1}{2}}+C(1+t)^{-\frac{d+2{s}-2}{4}}\nonumber\\
&\le C(1+t)^{  - \frac{\beta_2+{s}-1}{2}}.
\end{align*}

Finally, we obtain the decay rate of $\tau$ that
\begin{align*}
\|\Lambda^{\beta_2}\tau\|_{L^2}
\le C(1+t)^{ - \frac{\beta_2 + {s} -1}{2}}.
\end{align*}
This completes the proof of the Theorem \ref{dingli1}. $\hspace{7.2cm} \qquad \quad \square$

\vskip .2in

\section{The proof of Theorem \ref{dingli2}}
This section is devoted to proving Theorem \ref{dingli2}. The proof is long
and is thus divided into several
subsections for the sake of clarity.

\subsection{Local well-posedness}\label{loc}
Given the initial data $(P_0-\bar{P},\u_0, \eta_0-\bar{\eta},\tau_0)\in H^{3}(\T^d)$,
 the local well-posedness of \eqref{zhouqim3} could be proven
by using the standard energy method (see, e.g.,  \cite{Kawashima}). Thus, we may assume that there exists $T > 0$  such that the system \eqref{zhouqim3} has a unique solution
$(P-\bar{P},\u, \eta-\bar{\eta},\tau)\in C([0,T];H^{3})$. Moreover,
\begin{align}\label{youjiexing}
\frac12c_0\le\rho(t,x),\eta(t,x)\le 2c_0^{-1},\quad\hbox{for any $t\in[0,T]$}.
\end{align}
We use the bootstrapping argument to show that this local solution can be extended into a global one.
The goal is to derive {\itshape a priori} upper bound. To initiate the bootstrapping argument, we make the
ansatz that
\begin{align*}
\sup_{t\in[0,T]}(\norm{P-\bar{P}}{H^{3}}+\norm{\u}{H^{3}}
     +\norm{\eta-\bar{\eta}}{H^{3}}+\norm{\tau}{H^{3}})\le \delta,
\end{align*}
where $0<\delta<1$ obeys requirements to be specified later. In the following subsections we prove that, if the initial norm is taken to be sufficiently small, namely
$$
\|P_0-\bar{P}\|_{H^{3}} +\| \u_0\|_{H^{3}}+\|\eta_0-\bar{\eta}\|_{H^{3}}+\|\tau_0\|_{H^{3}}\le \varepsilon,
$$
with sufficiently small $\varepsilon>0$, then
$$
\sup_{t\in[0,T]}(\norm{P-\bar{P}}{H^{3}}+\norm{\u}{H^{3}}
     +\norm{\eta-\bar{\eta}}{H^{3}}+\norm{\tau}{H^{3}})\le \frac{\delta}{2}.
$$
The bootstrapping argument then leads to the desired global bound.

\subsection{Energy estimates for $(P-\bar{P},\u, \eta-\bar{\eta},\tau)$}
We first show the energy estimates
which contains the dissipation estimate for $\u$  only.
Without loss of generality,  we set $R=1$ in
$P(\rho)=R\rho^\gamma$ and let $\bar \rho=\bar \eta=1.$
By introducing
\begin{align*}
p\stackrel{\mathrm{def}}{=}P -1\quad\hbox{and}
\quad b\stackrel{\mathrm{def}}{=}\eta-1,
\end{align*}
we can reformulate the system \eqref{zhouqim3} into the following form:
\begin{eqnarray}\label{mm}
\left\{\begin{aligned}
&\partial_tp+\gamma\div \u+\u\cdot\nabla p+\gamma p\div \u=0,\\
&\partial_t \u-\div(\bar{\mu}(\rho)\nabla\u)-\nabla(\bar{\lambda}(\rho)\div\u)+\nabla\aa=\widetilde{f}_1,\\
&\partial_t{b}+\div \u+\u\cdot\nabla {b}+{b}\div \u=0,\\
&\partial_t\tau + \tau+\div(\tau \u) =  0,
\end{aligned}\right.
\end{eqnarray}
where
\begin{align*}
&\bar{\mu}(\rho)\stackrel{\mathrm{def}}{=}\frac{\mu}{\rho},\quad
\bar{\lambda}(\rho)\stackrel{\mathrm{def}}{=}\frac{\lambda+\mu}{\rho},\nn\\
&\aa\stackrel{\mathrm{def}}{=}P+q(\eta)-K(L-1)-1-\zeta=p+K(L-1)b+\zeta b(b+2),
\end{align*}
and
\begin{align}\label{f1}
\widetilde{f}_1\stackrel{\mathrm{def}}{=}
-\u\cdot\nabla \u+J(a)\nabla \aa+\nabla J(a)\left(\mu\nabla\u+(\lambda+\mu)\div \u\right)
+\frac{1}{1+a} \nabla \tau,
\end{align}
with $a=\rho-1$ and $J(a)=\frac{a}{1+a}$.
\vskip .1in
In this subsection, we shall prove the following crucial lemma.
\begin{lemma}\label{gaojie}
Let $(p,\u,b,\tau) \in C([0, T];H^3)$ be a solution to the  system \eqref{mm}, there holds
\begin{align}\label{gaojie1}
&\frac12\frac{d}{dt}\norm{\left(\frac{1}{\sqrt{\gamma}}\,p,\u,\sqrt{K(L-1)+2\zeta}\,b,\sqrt{1+c_1}\,\tau\right)}{H^{3}}^2
     -\frac{1}{2\gamma}\frac{d}{dt}\int_{\T^d}\frac{p}{1+p}(\Lambda^3p)^2dx\nn\\
     &\qquad
      -\frac{K(L-1)}{2}\frac{d}{dt}\int_{\T^d}\frac{b}{1+b}(\Lambda^3b)^2dx
      +\mu\norm{\nabla\u}{H^{3}}^2
     +(\lambda+\mu)\norm{\div\u}{H^{3}}^2+\norm{\tau}{H^{3}}^2\nn\\
& \quad\leq C\left(\norm{\u}{H^3}+(\norm{p}{H^3}^2+\norm{b}{H^3}^2)\norm{\u}{H^3}
       +\norm{(\u,\aa,\tau)}{H^3}^2\right)
      \|( p, \u, b,\tau)\|_{H^{3}}^2,
\end{align}
where $c_1$ will be given in the following proof.
\end{lemma}
\begin{proof}
We can use the standard energy method to get the $L^2$ estimates, here we omit the details. Let $s=1,2,3.$
Applying  operator $\la^s$  to the equations of \eqref{mm} and then taking $L^2$ inner product with $(\frac{1}{\gamma}\la^s p, \la^s\u,(K(L-1)+2\zeta)\la^s b,(1+c_1)\la^s \tau)$ yield
\begin{align}\label{gaojie2}
&\frac12\frac{d}{dt}\norm{\left(\frac{1}{\sqrt{\gamma}}\, \la^s  p,\la^s \u,
\sqrt{K(L-1)+2\zeta}\,\la^s b,\sqrt{1+c_1}\,\la^s \tau\right)}{L^{2}}^2
-\int_{\T^d}\la^{s}\div(\bar{\mu}(\rho)\nabla\u)\cdot\la^{s} \u\,dx\nn\\
&\qquad-\int_{\T^d}\la^{s}\nabla(\bar{\lambda}(\rho)\div\u)\cdot\la^{s} \u\,dx
   +(1+c_1)\norm{\la^s \tau}{L^{2}}^2\nn\\
&\quad= -\frac{1}{\gamma}\int_{\T^d}\la^{s} (\u\cdot\nabla p)\cdot\la^s p\,dx
        -\int_{\T^d}\la^{s}\div\u\cdot \la^s p\,dx
        -\int_{\T^d}\la^{s}(p\,\div\u)\cdot \la^s p\,dx\nn\\
&\qquad
-\int_{\T^d}\la^{s} \nabla\aa\cdot\la^{s} \u\,dx
+\int_{\T^d}\la^{s} \widetilde{f}_1\cdot\la^{s} \u\,dx
-(K(L-1)+2\zeta)\int_{\T^d}\la^{s} \div\u\cdot\la^{s} b\,dx\nn\\
&\qquad
-(K(L-1)+2\zeta)\int_{\T^d}\la^{s} (\u\cdot\nabla b)\cdot\la^s b\,dx
-(K(L-1)+2\zeta)\int_{\T^d}\la^{s} (b\,\div \u)\cdot\la^{s} b\,dx\nn\\
&\qquad
-(1+c_1)\int_{\T^d}\la^{s}\big(\div(\tau\u)\big)\cdot \la^{s}\tau\,dx.
\end{align}
Due to $$
\nabla\aa=\nabla p+(K(L-1)+2\zeta)\nabla b+2\zeta b\nabla b
$$
and the  cancellations
\begin{align*}
&\int_{\T^d}\la^{s}\div\u\cdot \la^{s}p\,dx+\int_{\T^d}\la^{s}\nabla p\cdot\la^{s}\u\,dx=0,\nn\\
&\int_{\T^d}\la^{s}\div\u\cdot \la^{s}b\,dx+\int_{\T^d}\la^{s}\nabla b\cdot\la^{s}\u\,dx=0,
\end{align*}
we can further rewrite \eqref{gaojie2} into
\begin{align}\label{gaojie3}
&\frac12\frac{d}{dt}\norm{(\frac{1}{\sqrt{\gamma}}\, \la^s p,\la^s \u,
\sqrt{K(L-1)+2\zeta}\,\la^s b,\sqrt{1+c_1}\,\la^s \tau)}{L^{2}}^2
-\int_{\T^d}\la^{s}\div(\bar{\mu}(\rho)\nabla\u)\cdot\la^{s} \u\,dx\nn\\
&\qquad-\int_{\T^d}\la^{s}\nabla(\bar{\lambda}(\rho)\div\u)\cdot\la^{s} \u\,dx
   +(1+c_1)\norm{\la^s \tau}{L^{2}}^2\nn\\
&\quad= -\frac{1}{\gamma}\int_{\T^d}\la^{s} (\u\cdot\nabla p)\cdot\la^s p\,dx
        -\int_{\T^d}\la^{s}(p\,\div\u)\cdot \la^s p\,dx
-2\zeta\int_{\T^d}\la^{s} (b\nabla b)\cdot\la^{s} \u\,dx\nn\\
&\qquad
+\int_{\T^d}\la^{s} \widetilde{f}_1\cdot\la^{s} \u\,dx
-(K(L-1)+2\zeta)\int_{\T^d}\la^{s} (\u\cdot\nabla b)\cdot\la^s b\,dx\nn\\
&\qquad
-(K(L-1)+2\zeta)\int_{\T^d}\la^{s} (b\,\div \u)\cdot\la^{s} b\,dx
-(1+c_1)\int_{\T^d}\la^{s}\big(\div(\tau\u)\big)\cdot \la^{s}\tau\,dx.
\end{align}
The second term on the left-hand side can be written as
\begin{align}\label{gaojie4}
&-\int_{\T^d}\la^{s}\div(\bar{\mu}(\rho)\nabla\u)\cdot\la^{s} \u\,dx
=\int_{\T^d}\la^{s}(\bar{\mu}(\rho)\nabla\u)\cdot\nabla\la^{s} \u\,dx\nn\\
&\quad=\int_{\T^d}[\la^{s},\,\bar{\mu}(\rho)\nabla]\u\cdot\nabla\la^{s} \u\,dx
        +\int_{\T^d}\bar{\mu}(\rho)\nabla\la^{s}\u\cdot\nabla\la^{s} \u\,dx.
\end{align}
Due to \eqref{youjiexing}, we have for any $t\in[0,T]$ that
\begin{align}\label{gaojie5}
 \int_{\T^d}\bar{\mu}(\rho)\nabla\la^{s}\u\cdot\nabla\la^{s} \u\,dx
 \ge \frac{\mu c_0}{2}\norm{\la^{s+1}\u}{L^{2}}^2.
\end{align}
For the first term in \eqref{gaojie4}, we first rewrite this term into
\begin{align}\label{gaojie6}
\int_{\T^d}[\la^{s},\bar{\mu}(\rho)\nabla]\u\cdot\nabla\la^{s} \u\,dx
=&\int_{\T^d}[\la^{s},\,(\bar{\mu}(\rho)-\mu+\mu)\nabla]\u\cdot\nabla\la^{s} \u\,dx\nn\\
=&-\int_{\T^d}[\la^{s},\,\mu J(a)\nabla]\u\cdot\nabla\la^{s} \u\,dx.
\end{align}
Bounding nonlinear terms involving composition functions in $\eqref{gaojie6}$ is more elaborate.
Throughout we make the assumption that
\begin{equation}\label{axiao}
\sup_{t\in\R^+,\, x\in\T^d} |a(t,x)|\leq \frac{1}{100},
\end{equation}
which will enable us to use freely the composition estimate stated in Lemma \ref{fuhe}.
Note that $H^2(\T^d)\hookrightarrow L^\infty(\T^d),$ condition \eqref{axiao} will be ensured by the
fact that the constructed solution about $a$ has small norm.
It follows from Lemma \ref{fuhe} that
\begin{equation}\label{eq:smalla}
\|J(a)\|_{H^s}\le C\|a\|_{H^s}, \quad\hbox{for any $s>0$}.
\end{equation}
Moreover,  from  $p = P(1+a)-P(1)$, when \eqref{axiao} holds, $a$ can be expressed by a smooth function of $p$, hence
we can use Lemma \ref{fuhe} again to deduce that
\begin{equation}\label{smalla}
\|a\|_{H^3}\le C\|p\|_{H^3}.
\end{equation}
Then, using the fact $H^2(\T^d)\hookrightarrow L^\infty(\T^d)$,
with the aid of Lemmas \ref{jiaohuanzi}, \eqref{eq:smalla} and \eqref{smalla}, we have
\begin{align}\label{gaojie9}
\Big|\int_{\T^d}[\la^{s},\mu J(a)\nabla]\u\cdot\nabla\la^{s} \u\,dx\Big|
&\quad \le C\Big(\norm{\nabla J(a) }{L^\infty}\norm{\la^{s} \u}{L^2}
                 +\norm{\nabla \u}{L^\infty}\norm{\la^{s}J(a)}{L^2}\Big)
         \norm{\nabla\la^{s}\u}{L^{2}}\nn\\
&\quad \le C\Big(\norm{J(a) }{H^3}\norm{\la^{s} \u}{L^2}
           +\norm{\nabla \u}{L^\infty}\norm{J(a)}{H^3}\Big)
         \norm{\nabla\la^{s}\u}{L^{2}}\nn\\
&\quad \le \frac{\mu c_0}{4}\norm{\la^{s+1}\u}{L^{2}}^2+C\norm{ p }{H^3}^2\norm{ \u}{H^3}^2.
\end{align}
Inserting \eqref{gaojie5}, \eqref{gaojie6}, and \eqref{gaojie9} into \eqref{gaojie4} leads to
\begin{align}\label{gaojie10}
-\int_{\T^d}\la^{s}\div(\bar{\mu}(\rho)\nabla\u)\cdot\la^{s} \u\,dx
\ge&\frac{\mu c_0}{4}\norm{\la^{s+1}\u}{L^{2}}^2
     -C\norm{p}{H^3}^2\norm{ \u}{H^3}^2.
\end{align}
Similarly, we have
\begin{align}\label{gaojie11}
-\int_{\T^d}\la^{s}\nabla(\bar{\lambda}(\rho)\div\u)\cdot\la^{s} \u\,dx
\ge&\frac{(\lambda+\mu)c_0}{4}\norm{\la^{s}\div\u}{L^{2}}^2
     -C\norm{p}{H^3}^2\norm{ \u}{H^3}^2.
\end{align}

Hence, plugging \eqref{gaojie10} and \eqref{gaojie11} into \eqref{gaojie3}, we obtain
\begin{align}\label{gaojie13}
&\frac12\frac{d}{dt}\norm{(\frac{1}{\sqrt{\gamma}}\,\la^s p,\la^s \u,
\sqrt{K(L-1)+2\zeta}\,\la^s b,\sqrt{1+c_1}\,\la^s \tau)}{L^{2}}^2\nn\\
&\qquad
+\frac{\mu c_0}{4}\norm{\la^{s+1}\u}{L^2}^2
     +\frac{(\lambda+\mu)c_0}{4}\norm{\la^{s}\mbox{div}\u}{L^2}^2
   +(1+c_1)\norm{\la^s \tau}{L^{2}}^2\nn\\
&\quad\leq C\norm{p}{H^3}^2\norm{ \u}{H^3}^2
-\frac{1}{\gamma}\int_{\T^d}\la^{s} (\u\cdot\nabla p)\cdot\la^s p\,dx
        -\int_{\T^d}\la^{s}(p\,\div\u)\cdot \la^s p\,dx\nn\\
&\qquad
-2\zeta\left(\int_{\T^d}\la^{s} (b\nabla b)\cdot\la^{s} \u\,dx
              +\int_{\T^d}\la^{s} (b\,\div \u)\cdot\la^{s} b\,dx\right)\nn\\
&\qquad
-(K(L-1)+2\zeta)\int_{\T^d}\la^{s} (\u\cdot\nabla b)\cdot\la^s b\,dx
-K(L-1)\int_{\T^d}\la^{s} (b\,\div \u)\cdot\la^{s} b\,dx
\nn\\
&\qquad-(1+c_1)\int_{\T^d}\la^{s}\big(\div(\tau\u)\big)\cdot \la^{s}\tau\,dx
+\int_{\T^d}\la^{s} \widetilde{f}_1\cdot\la^{s} \u\,dx
\nn\\
&\quad \stackrel{\mathrm{def}}{=}C\norm{p}{H^3}^2\norm{ \u}{H^3}^2+M_1+M_2+M_3+M_4+M_5+M_6+M_7.
\end{align}
For the first  term $M_1$, it follows from integration by parts that
\begin{align}\label{gaojie14}
M_1=& -\frac{1}{\gamma}\int_{\T^d}\la^{s} (\u\cdot\nabla p)\cdot\la^s p\,dx\nn\\
= &-\frac{1}{\gamma}\int_{\T^d} [\la^{s},\,\u\cdot\nabla] p\cdot\la^s p\,dx
-\frac{1}{\gamma}\int_{\T^d}\u\cdot\nabla\la^{s}p\cdot\la^s p\,dx.
\end{align}
By Lemma \ref{jiaohuanzi}, there holds
\begin{align}\label{gaojie15}
\left|\int_{\T^d}[\la^{s},\,\u\cdot\nabla] p\cdot\la^s p\,dx\right|
\le& C\norm{[\la^s,\u\cdot\nabla ]p}{L^2}\norm{\la^{s}p}{L^2}\nn\\
\le&C\big(\norm{\nabla \u}{L^\infty}\norm{\la^{s}p}{L^2}
+\norm{\la^{s} \u}{L^2}\norm{\nabla p}{L^\infty}\big)\norm{\la^{s}p}{L^2}\nn\\
\le&C\norm{\u}{H^3}\norm{p}{H^3}^2.
\end{align}
By integration by parts, there holds
\begin{align*}
\left|\int_{\T^d}\u\cdot\nabla\la^{s}p\cdot\la^s p\,dx\right|
\le& C\norm{\nabla \u}{L^\infty}\norm{\la^{s}p}{L^2}^2
\le C\norm{\u}{H^3}\norm{p}{H^3}^2,
\end{align*}
from which and \eqref{gaojie15}, we get
\begin{align}\label{I1}
\left|M_1\right|
\le&C\norm{\u}{H^3}\norm{p}{H^3}^2.
\end{align}
In the same manner, there holds
\begin{align}\label{I4}
\left|M_4\right|\le C\left|\int_{\T^d}\la^{s} (\u\cdot\nabla b)\cdot\la^{s} b\,dx\right|
\le C\norm{\u}{H^3}\norm{b}{H^3}^2.
\end{align}
Similarly,
\begin{align}\label{gaojie19}
&\int_{\T^d}\la^{s} (b\,\div \u)\cdot\la^{s} b\,dx
+\int_{\T^d}\la^{s} (b\nabla b)\cdot\la^{s} \u\,dx\nn\\
&\quad=\int_{\T^d}[\la^{s},\, b\div] \u\cdot\la^{s} b\,dx
+\int_{\T^d}[\la^{s},\,b\nabla] b\cdot\la^{s} \u\,dx\nn\\
&\qquad+\int_{\T^d}b\div\la^{s} \u\cdot\la^{s} b\,dx+\int_{\T^d}b \nabla\la^{s} b\cdot\la^{s} \u\,dx\nn\\
&\quad=\int_{\T^d}[\la^{s},\, b\div] \u\cdot\la^{s} b\,dx+\int_{\T^d}[\la^{s},\,b\nabla] b\cdot\la^{s} \u\,dx+\int_{\T^d}b\div (\la^{s}b\la^{s}\u)\,dx\nn\\
&\quad=\int_{\T^d}[\la^{s}, \,b\div] \u\cdot\la^{s} b\,dx+\int_{\T^d}[\la^{s},\,b\nabla] b\cdot\la^{s} \u\,dx-\int_{\T^d}\la^{s}b\la^{s}\u\cdot\nabla b\,dx.
\end{align}
By Lemma \ref{jiaohuanzi}, we have
\begin{align}\label{gaojie20}
\int_{\T^d}[\la^{s}, b\div] \u\cdot\la^{s} b\,dx
\le& C\norm{[\la^s,b\div  ]\u}{L^2}\norm{\la^{s}b}{L^2}\nn\\
\le&C\Big(\norm{\nabla \u}{L^\infty}\norm{\la^{s}b}{L^2}
  +\norm{\la^{s} \u}{L^2}\norm{\nabla b}{L^\infty}\Big)\norm{\la^{s}b}{L^2}\nn\\
\le&C\norm{\u}{H^3}\norm{b}{H^3}^2
\end{align}
and
\begin{align}\label{gaojie21}
\int_{\T^d}[\la^{s}, b\nabla] b\cdot\la^{s} \u\,dx
\le& C\norm{[\la^s,b \nabla] b}{L^2}\norm{\la^{s}\u}{L^2}\nn\\
\le&C\norm{\nabla b}{L^\infty}\norm{\la^{s}b}{L^2}\norm{\la^{s}\u}{L^2}\nn\\
\le&C\norm{\u}{H^3}\norm{b}{H^3}^2.
\end{align}
For the last term in \eqref{gaojie19},
\begin{align}\label{gaojie22}
\int_{\T^d}\la^{s}b\la^{s}\u\cdot\nabla b\,dx
\le& C\norm{\la^{s}b}{L^2}\norm{\la^{s}\u}{L^2}\norm{\nabla b}{L^\infty}\nn\\
\le&C\norm{\u}{H^3}\norm{b}{H^3}^2.
\end{align}
Hence, combining with \eqref{gaojie19}--\eqref{gaojie22}, we can get
\begin{align}\label{I3}
\left|M_3\right|\leq C\left|\int_{\T^d}\la^{s} (b\nabla b)\cdot\la^{s} \u\,dx
+\int_{\T^d}\la^{s} b\,\div \u\cdot\la^{s} b\,dx
\right|
\le C\norm{\u}{H^3}\norm{b}{H^3}^2.
\end{align}
Next, we have to bound the most trouble term
 \begin{align}\label{gaojie24}
 \int_{\T^d}\la^{s}  (p\,\div \u)\cdot\la^s p\,dx
 =\int_{\T^d}[\la^{s}, \, p\,\div] \u\cdot\la^s p\,dx
 +\int_{\T^d}p\,\div \la^{s}\u\cdot\la^s p\,dx.
\end{align}
Then, by Lemma \ref{jiaohuanzi}, we have
\begin{align}\label{gaojie25}
\left|-\int_{\T^d}[\la^{s},\, p\,\div] \u\cdot\la^s p\,dx\right|
\le& C\norm{[\la^s,\, p\,\div ] \u}{L^2}\norm{\la^{s}p}{L^2}\nn\\
\le&C(\norm{\nabla \u}{L^\infty}\norm{\la^{s}p}{L^2}
+\norm{\la^{s} \u}{L^2}\norm{\nabla p}{L^\infty})\norm{\la^{s}p}{L^2}\nn\\
\le&C\norm{\u}{H^3}\norm{p}{H^3}^2.
\end{align}
Finally, we have to deal with the last term on the right hand side of \eqref{gaojie24}. In fact, for $s=0,1,2$, we can bound this term directly as follows
\begin{align}\label{gaojie26}
\left|-\int_{\T^d}p\,\div\la^{s} \u\cdot\la^s p\,dx\right|
\leq&C\norm{p}{L^\infty}\norm{\la^{s}\div \u}{L^2}\norm{\la^{s}p}{L^2}\nn\\
\leq&C\norm{p}{H^2}\norm{\div\u}{H^2}\norm{p}{H^2}\nn\\
\leq&C\norm{\u}{H^3}\norm{p}{H^3}^2.
\end{align}

However, if $s=3$,  the term $\norm{\div\u}{H^3}$ must be appeared which we cannot control it.
To overcome the difficulty, we  deduce  from the first equation of \eqref{mm} that
 $$\div\u=-\frac{\partial_t p+\u\cdot\nabla p}{\gamma(p+1)},$$
 from which we have
\begin{align}\label{gaojie27}
-\int_{\T^d}p\la^{{3}} \div\u\cdot\la^{{3}} p\,dx
=&\frac{1}{\gamma}\int_{\T^d}p\la^{{3}} \left(\frac{\partial_t p+\u\cdot\nabla p}{p+1}\right)\cdot\la^{{3}} p\,dx \nn\\
=&\frac{1}{\gamma}
\left(\underbrace{\int_{\T^d}p\la^{{3}} \left(\frac{\partial_t p}{p+1}\right)\cdot\la^{{3}} p\,dx }_{D_1} +\underbrace{\int_{\T^d}p\la^{{3}} \left(\frac{\u\cdot\nabla p}{p+1}\right)\cdot\la^{{3}} p\,dx}_{D_2}
\right)
.
\end{align}
For the first term $D_1$, we have
\begin{align}\label{gaojie28}
D_1=&\int_{\T^d}p\la^{{3}} \left(\frac{\partial_t p}{p+1}\right)\cdot\la^{{3}} p\,dx\nn\\
=&
\int_{\T^d}\frac{p}{p+1}\la^{{3}} \left({\partial_t p}\right)\cdot\la^{{3}} p\,dx
+\int_{\T^d}p\sum_{\ell=0}^2 C_3^\ell\la^{{\ell}}\partial_t p
 \la^{{3-\ell}}\left(\frac{1}{p+1}\right)
 \cdot\la^{{3}} p\,dx
\nn\\
=&\frac{1}{2}\int_{\T^d}\frac{p}{p+1}\partial_t(\la^{{3}}p)^2 \,dx
+\int_{\T^d}p\sum_{\ell=0}^2 C_3^\ell\la^{{\ell}}\partial_t p
 \la^{{3-\ell}}\left(\frac{1}{p+1}\right)
\cdot\la^{{3}} p\,dx
\nn\\
=&\frac{1}{2}\frac{d}{dt}\int_{\T^d}\frac{p}{p+1}(\la^{{3}}p)^2 \,dx
-\frac{1}{2}\int_{\T^d}\frac{1}{(p+1)^2}\partial_t p(\la^{{3}}p)^2\,dx\nn\\
&+
\int_{\T^d}p\sum_{\ell=0}^2 C_3^\ell\la^{{\ell}}\partial_t p
\la^{{3-\ell}}\left(\frac{1}{p+1}\right)
 \cdot\la^{{3}} p\,dx.
\end{align}
Using the first equation of \eqref{mm}, we  can bound the second term on the right hand side of \eqref{gaojie28} as
\begin{align}\label{gaojie29}
-\frac{1}{2}\int_{\T^d}\frac{1}{(p+1)^2}\partial_t p(\la^{{3}}p)^2\,dx
=&\frac{1}{2}\int_{\T^d}\frac{1}{(p+1)^2}
(\u\cdot\nabla p+\gamma p\,\div \u+\gamma\div \u)(\la^{{3}}p)^2\,dx\nn\\
\le&C\Big((1+\norm{p}{L^\infty})\norm{\nabla \u}{L^\infty}+\norm{\nabla p}{L^\infty}
\norm{ \u}{L^\infty}\Big)\norm{\la^{{3}}p}{L^2}^2\nn\\
\le&C\Big(1+\norm{p}{H^3}\Big)\norm{\u}{H^3}\norm{p}{H^3}^2.
\end{align}
By the H\"older inequality and embedding inequality, the last term in \eqref{gaojie28} can be controlled as
\begin{align*}
&\int_{\T^d}p\sum_{\ell=0}^2 C_3^\ell\la^{{\ell}}\partial_t p
 \la^{{3-\ell}}\left(\frac{1}{p+1}\right)
 \cdot\la^{{3}} p\,dx\\
& \leq C\norm{p}{L^\infty}\norm{\la^{3}p}{L^2}
 \Bigg(\norm{\partial_tp}{L^\infty}\norm{\la^{{3}}\left(\frac{1}{p+1}\right)}{L^2}\\
 &\qquad
  +\norm{\la\partial_tp}{L^3}\norm{\la^{{2}}\left(\frac{1}{p+1}\right)}{L^6}
   +\norm{\la^2\partial_tp}{L^2}\norm{\la\left(\frac{1}{p+1}\right)}{L^\infty}
  \Bigg)\\
&\leq C\norm{p}{L^\infty}\norm{\la^{3}p}{L^2}\norm{\partial_tp}{H^2}
   \left(\norm{\la^{3}p}{L^2}+\norm{\la^{2}p}{H^1}+\norm{\la p}{H^2}\right)\\
&\leq C\norm{\partial_tp}{H^2}\norm{ p}{H^3}^3.
\end{align*}
Since
\begin{align*}
\norm{\partial_t p}{H^2}
\leq C\norm{\u\cdot\nabla p+ \gamma a\,\div \u+\gamma\div \u}{H^2}
\leq C(\norm{\u}{H^3} +\norm{\u}{H^3}\norm{p}{H^3}).
\end{align*}
Hence, we have
\begin{align}\label{gaojie33}
\int_{\T^d}p\sum_{\ell=0}^2 C_3^\ell\la^{{\ell}}\partial_t p
 \la^{{3-\ell}}\left(\frac{1}{p+1}\right)
 \cdot\la^{{3}} p\,dx
 \le C(1+\norm{p}{H^3})\norm{\u}{H^3}\norm{p}{H^3}^3.
\end{align}
Combining  \eqref{gaojie29} with \eqref{gaojie33}, we get
\begin{align}\label{gaojie34}
D_1
\le&\frac{1}{2}\frac{d}{dt}\int_{\T^d}\frac{p}{p+1}(\la^{{3}}p)^2 \,dx
+C(1+\norm{p}{H^3}^2)\norm{\u}{H^3}\norm{p}{H^3}^2.
\end{align}
For the term $D_2$, we infer that
\begin{align}\label{gaojie35}
D_2=&\int_{\T^d}p\la^{{3}} \left(\frac{\u\cdot\nabla p}{p+1}\right)\cdot\la^{{3}} p\,dx\nn\\
=&\int_{\T^d}\frac{p}{p+1}\la^{{3}}(\u\cdot\nabla p)\cdot\la^{{3}} p\,dx+\int_{\T^d}p\sum_{\ell=0}^2 C_3^\ell\la^{{\ell}}(\u\cdot\nabla p)\la^{{3-\ell}}\left(\frac{1}{p+1}\right)
\cdot\la^{{3}} p\,dx\nn\\
=&D_{2,1}+D_{2,2}.
\end{align}
We can use the commutator again   to rewrite $D_{2,1}$ into
\begin{align*}
D_{2,1}=&\int_{\T^d}\frac{p}{p+1}
        \left[\la^{{3}},\,\u\cdot\nabla\right] p\cdot \la^{{3}} p\,dx+
\int_{\T^d}\frac{p}{p+1}\u\cdot\nabla\la^{{3}} p\cdot\la^{{3}} p\,dx.
\end{align*}
Thanks to Lemma \ref{jiaohuanzi}, we get
\begin{align*}
\left|\int_{\T^d}\frac{p}{p+1}
      \left[\la^{{3}},\,\u\cdot\nabla\right] p\cdot \la^{{3}} p\,dx\right|
\le& C\norm{\frac{p}{p+1}}{L^\infty}\norm{[\la^3,\u\cdot\nabla ]p}{L^2}\norm{\la^{{3}}p}{L^2}\nn\\
\le&C\Big(\norm{\nabla \u}{L^\infty}\norm{\la^{{3}}p}{L^2}+\norm{\la^{{3}} \u}{L^2}\norm{\nabla p}{L^\infty}\Big)\norm{\la^{{3}}p}{L^2}\nn\\
\le&C\norm{\u}{H^3}\norm{p}{H^3}^2.
\end{align*}
By using the integration by parts, we have
\begin{align*}
\left|
\int_{\T^d}\frac{p}{p+1}\u\cdot\nabla\la^{{3}} p\cdot\la^{{3}} p\,dx\right|
\le& C\norm{\div\left(\frac{p\u}{p+1}\right)}{L^\infty}\norm{\la^{{3}}p}{L^2}^2\nn\\
\le&C\Big(\norm{\nabla \u}{L^\infty}+\norm{ \u}{L^\infty}\norm{\nabla p}{L^\infty}\Big)\norm{\la^{{3}}p}{L^2}^2\nn\\
\le&C\big(1+\norm{p}{H^3}\big)\norm{\u}{H^3}\norm{p}{H^3}^2,
\end{align*}
from which we get
\begin{align}\label{gaojie36}
D_{2,1}\le&C\big(1+\norm{p}{H^3}\big)\norm{\u}{H^3}\norm{p}{H^3}^2.
\end{align}
Thanks to the H\"older inequality and embedding inequality, we can get
\begin{align*}
D_{2,2}=&
\int_{\T^d}p\sum_{\ell=0}^2 C_3^\ell\la^{{\ell}}(\u\cdot\nabla p)
   \la^{{3-\ell}}\left(\frac{1}{p+1}\right)
\cdot\la^{{3}} p\,dx\nn\\
\leq& C\norm{p}{L^\infty}\norm{\la^{3}p}{L^2}
    \Bigg(\norm{\u\cdot\nabla p}{L^\infty}\norm{\la^{{3}}\left(\frac{1}{p+1}\right)}{L^2}\\
   &\qquad
  +\norm{\la(\u\cdot\nabla p)}{L^3}\norm{\la^{{2}}\left(\frac{1}{p+1}\right)}{L^6}
   +\norm{\la^2(\u\cdot\nabla p)}{L^2}\norm{\la\left(\frac{1}{p+1}\right)}{L^\infty}
  \Bigg)\\
\leq&C\norm{p}{L^\infty}\norm{\la^{3}p}{L^2}\norm{\u\cdot\nabla p}{H^2}
   \left(\norm{\la^{3}p}{L^2}+\norm{\la^{2}p}{H^1}+\norm{\la p}{H^2}\right)\\
\leq&C\norm{\u}{H^3}\norm{ p}{H^3}^4,
\end{align*}
which combines with \eqref{gaojie36} implies that
\begin{align}\label{gaojie37}
D_{2}
\le C(1+\norm{p}{H^3}^2)\norm{\u}{H^3}\norm{p}{H^3}^2.
\end{align}
Inserting \eqref{gaojie34} and \eqref{gaojie37} into \eqref{gaojie27} leads to
\begin{align}\label{gaojie38}
-\int_{\T^d}p\la^{{3}} \div\u\cdot\la^{{3}} p\,dx
\le&\frac{1}{2\gamma}\frac{d}{dt}\int_{\T^d}\frac{p}{p+1}(\la^{{3}}p)^2 \,dx+C(1+\norm{p}{H^3}^2)\norm{\u}{H^3}\norm{p}{H^3}^2.
\end{align}
Consequently, taking the estimates \eqref{gaojie25}, \eqref{gaojie26} and \eqref{gaojie38} into \eqref{gaojie24}, we get
\begin{align}\label{I2}
M_2
\le&\frac{1}{2\gamma}\frac{d}{dt}\int_{\T^d}\frac{p}{p+1}(\la^{3}p)^2 \,dx
+C(1+\norm{p}{H^3}^2)\norm{\u}{H^3}\norm{p}{H^3}^2.
\end{align}
In the same manner, there holds
\begin{align}\label{I5}
M_5
\le&\frac{K(L-1)}{2}\frac{d}{dt}\int_{\T^d}\frac{b}{1+b}(\la^{3}b)^2 \,dx
+C(1+\norm{b}{H^3}^2)\norm{\u}{H^3}\norm{b}{H^3}^2.
\end{align}
To bound $M_6$, we rewrite
\begin{align*}
M_6
=-(1+c_2)\left(\int_{\T^d}\la^{s}\big(\tau\div\u\big)\cdot \la^{s}\tau\,dx
+\int_{\T^d}\la^{s}\big(\u\div\tau\big)\cdot \la^{s}\tau\,dx\right).
\end{align*}
By using Lemma \ref{daishu}, we have
\begin{align*}
(1+c_2)\left|\int_{\T^d}\la^{s}\big(\tau\div\u\big)\cdot \la^{s}\tau\,dx\right|
\leq&C(\|\tau\|_{L^\infty}\norm{\div \u}{H^s}
+\|\div \u\|_{L^\infty}\norm{\tau}{H^s})\norm{\la^{s}\tau}{L^2}\\
\leq&C\norm{\div \u}{H^3}\norm{\tau}{H^3}^2\\
\leq&\frac{(\lambda+\mu)c_0}{16}\norm{\div \u}{H^3}^2+C\norm{\tau}{H^3}^4.
\end{align*}
It follows from Lemma \ref{jiaohuanzi} that
\begin{align*}
(1+c_2)\left|\int_{\T^d}\la^{s}\big(\u\div\tau\big)\cdot \la^{s}\tau\,dx\right|
\leq&C\left(\left|\int_{\T^d}
      \left[\la^{{s}},\,\u\div\right] \tau\cdot \la^{{s}} \tau\,dx\right|
+\left|\int_{\T^d}
      \u\div\la^{{s}} \tau\cdot \la^{{s}} \tau\,dx\right|
\right)\\
\leq&C(\|\nabla\tau\|_{L^\infty}\norm{\u}{H^s}
+\|\nabla \u\|_{L^\infty}\norm{\tau}{H^s})\norm{\la^{s}\tau}{L^2}\\
\leq&C\norm{ \u}{H^3}\norm{\tau}{H^3}^2.
\end{align*}
Then we get
\begin{align}\label{I6}
|M_6|\leq\frac{(\lambda+\mu)c_0}{16}\norm{\div \u}{H^3}^2
+C\norm{\tau}{H^3}^4+C\norm{ \u}{H^3}\norm{\tau}{H^3}^2.
\end{align}
 In the following, we bound the terms of $\widetilde{f}_1$ in \eqref{gaojie13}. To do this, we  write
\begin{align}\label{gaojie40}
M_7=\int_{\T^d}\la^{s} \widetilde{f}_1\cdot\la^{s} \u\,dx\stackrel{\mathrm{def}}{=}\widetilde{A}_1+\widetilde{A}_2+\widetilde{A}_3+\widetilde{A}_4+\widetilde{A}_5
\end{align}
with
\begin{align*}
&\widetilde{A}_1\stackrel{\mathrm{def}}{=}-\int_{\T^d}\la^{s} \big(\u\cdot\nabla \u\big)\cdot\la^{s} \u\,dx,\nn\\
&\widetilde{A}_2\stackrel{\mathrm{def}}{=}\int_{\T^d}\la^{s} \big(J(a)\nabla \aa\big)\cdot\la^{s} \u\,dx,\nn\\
&
\widetilde{A}_3\stackrel{\mathrm{def}}{=}\int_{\T^d}\la^{s} \big(\mu\nabla J(a)\nabla\u\big)\cdot\la^{s} \u\,dx,\nn\\
&\widetilde{A}_4\stackrel{\mathrm{def}}{=}\int_{\T^d}\la^{s}\big( (\lambda+\mu)
      \nabla J(a)\div\u\big)\cdot\la^{s} \u\,dx,\nn\\
&\widetilde{A}_5\stackrel{\mathrm{def}}{=}\int_{\T^d}\la^{s} \big(\frac{1}{1+a}\nabla\tau\big)\cdot\la^{s} \u\,dx.
\end{align*}
The term $\widetilde{A}_1$ can be bounded as in \eqref{I1}
\begin{align*}
\big|\widetilde{A}_1\big|\le C\norm{\u}{H^3}^3.
\end{align*}
By Lemmas \ref{daishu} and (\ref{eq:smalla}), we have
\begin{align*}
\big|\widetilde{A}_2\big|
\leq&C\Big(\|\nabla \aa\|_{L^{\infty}}\|J(a)\|_{H^{s-1}}
        +\|\nabla \aa\|_{H^{s-1}}\|J(a)\|_{L^{\infty}}\Big)\norm{\la^{s+1}\u}{L^2}\nonumber\\
\leq&\frac{\mu c_0}{64}\norm{\la^{s+1}\u}{L^2}^2
         +C\norm{a}{H^{s-1}}^2\norm{\aa}{H^3}^2
         +C\|J(a)\|_{H^{2}}^2\norm{\aa}{H^s}^2\nn\\
\leq&\frac{\mu c_0}{64}\norm{\la^{s+1}\u}{L^2}^2
         +C\norm{p}{H^{3}}^2\norm{\aa}{H^3}^2.
\end{align*}
Similarly,
\begin{align*}
\big|\widetilde{A}_3\big|+\big|\widetilde{A}_4\big|
\leq&C\Big(\|\nabla J(a)\|_{L^{\infty}}\norm{\la^{s}\u}{L^2}
        +\|\nabla J(a)\|_{H^{s-1}}\|\nabla \u\|_{L^{\infty}}\Big)\norm{\la^{s+1}\u}{L^2}\nonumber\\
\leq&\frac{\mu{c_0}}{64}\norm{\la^{s+1}\u}{L^2}^2
   +C\norm{p}{H^3}^2\norm{ \u}{H^3}^2
\end{align*}
and
\begin{align*}
\big|\widetilde{A}_5\big|
\leq&C\norm{\frac{1}{1+a}\nabla\tau}{H^{s-1}}\norm{\la^{s+1}\u}{L^2}
\leq C\|\frac{1}{1+a}\|_{H^{2}}\|\tau\|_{H^{3}}\norm{\la^{s+1}\u}{L^2}\\
\leq&C\left(1+\|a\|_{H^{3}}\right)\|\tau\|_{H^{3}}\norm{\la^{s+1}\u}{L^2}
\leq\frac{\mu c_0}{64}\norm{\la^{s+1}\u}{L^2}^2
         +c_1(1+\norm{a}{H^{3}}^2)\norm{\tau}{H^3}^2\\
\leq&\frac{\mu c_0}{64}\norm{\la^{s+1}\u}{L^2}^2
         +c_1(1+\norm{p}{H^{3}}^2)\norm{\tau}{H^3}^2.
\end{align*}
Inserting the bounds for $\widetilde{A}_1$ through $\widetilde{A}_5$ into \eqref{gaojie40}, we get
\begin{align}\label{I7}
|M_7|
\leq\frac{3\mu{c_0}}{64}\norm{\la^{s+1}\u}{L^2}^2
+C\norm{\u}{H^3}^3
+C\norm{(\aa,\u,\tau)}{H^3}^2\norm{p}{H^3}^2+c_1\norm{\tau}{H^3}^2.
\end{align}
Finally, inserting \eqref{I1}, \eqref{I4}, \eqref{I3},
   \eqref{I2}, \eqref{I5}, \eqref{I6} and \eqref{I7} into \eqref{gaojie13}
and then summing up  \eqref{gaojie13} over $1\leq s\leq3$, we obtain \eqref{gaojie1}.
Consequently, we prove the Lemma \ref{gaojie}.
\end{proof}
\subsection{Energy estimates for $(\aa,\u, \tau)$ }
In this subsection, we shall give the energy estimates for the unknown
good function $\aa$,
We need to reformulate \eqref{mm} in terms of variables $\aa$, $\u$ and $\tau$.
Precisely, one has
\begin{eqnarray}\label{energy1}
\left\{\begin{aligned}
&\partial_t \aa+ (\gamma+2\ze+K(L-1))\div\u  =\widetilde{f}_2,\\
&\partial_t \u-\div(\bar{\mu}(\rho)\nabla\u)-\nabla(\bar{\lambda}(\rho)\div\u)+\nabla \aa
=\widetilde{f}_1,\\
&\partial_t\tau + \tau+\div(\tau \u) =  0,\\
&(\aa,\u,\tau)|_{t=0}=(\aa_0,\u_0,\tau_0),
\end{aligned}\right.
\end{eqnarray}
with
\begin{align}\label{aa}
&\aa\stackrel{\mathrm{def}}{=}p+K(L-1)b+\zeta\eta^2-\zeta=p+K(L-1)b+\zeta b(b+2),\\
\label{f2}
&\widetilde{f}_2\stackrel{\mathrm{def}}{=}-\u\cdot\nabla \aa
-\gamma\aa\div \u+(\gamma-2)\zeta b^2\div\u+\big(2(\gamma-2)\zeta +K(\gamma-1)(L-1)\big)b\div\u,
\end{align}
and $\widetilde{f}_1$ is defined in \eqref{f1}.
 Then, we will present the energy estimates for $(\aa,\u, \tau)$ in the following Lemma.
\begin{lemma}\label{haosan}
Let $(\aa,\u,\tau) \in C([0, T];H^3)$ be a solution to the  system \eqref{energy1}, there holds
\begin{align}\label{energy2}
&
\frac12\frac{d}{dt}\norm{\left(\frac{1}{\sqrt{\gamma+2\ze+K(L-1)}}\,\aa,\u,\sqrt{1+c_2}\,\tau\right)}{H^{3}}^2
+\mu\norm{\nabla\u}{H^{3}}^2
+(\lambda+\mu)\norm{\div\u}{H^{3}}^2+\norm{\tau}{H^{3}}^2\nn\\
& \quad\leq C\left(\norm{(\u,p)}{H^3}
       +(1+\norm{b}{H^3}^2)\norm{b}{H^3}^2
       +\norm{(p,\aa,\tau}{H^3}^2\right)
      \|(\aa, \u,\tau)\|_{H^{3}}^2,
\end{align}
where $c_2$ will be given later.
\end{lemma}
\begin{proof}
We start with the $L^2$ estimate.
Taking inner product with $(\frac{1}{\gamma+2\ze+K(L-1)}\aa,\u,(1+c_2)\tau)$ for the equations in \eqref{energy1} gives
\begin{align}\label{energy3}
&\frac12\frac{d}{dt}\norm{\left(\frac{1}{\sqrt{\gamma+2\ze+K(L-1)}}\,\aa,\u,\sqrt{1+c_2}\,\tau\right)}{L^{2}}^2
    -\int_{\T^d}\div(\bar{\mu}(\rho)\nabla\u)\cdot \u\,dx\nn\\
&\qquad-\int_{\T^d}\nabla(\bar{\lambda}(\rho)\div\u)\cdot \u\,dx+(1+c_2)\int_{\T^d}\tau^2\,dx\nn\\
&\quad= \frac{1}{\gamma+2\ze+K(L-1)}\int_{\T^d} \widetilde{f}_2\cdot \aa\,dx
      +\int_{\T^d} \widetilde{f}_1\cdot \u\,dx-(1+c_2)\int_{\T^d} \div(\tau \u)\cdot \tau\,dx,\nn\\
&\quad\stackrel{\mathrm{def}}{=} B_{11}+B_{12}+B_{13}
\end{align}
where we have used the following cancellations
\begin{align*}
\int_{\T^d}\div\u\cdot \aa\,dx+\int_{\T^d}\nabla \aa\cdot\u\,dx=0.
\end{align*}
For the last three terms on the left hand side of \eqref{energy3}, we get by integration by parts  and \eqref{youjiexing}  that
\begin{align}\label{energy4}
-\int_{\T^d}\div(\bar{\mu}(\rho)\nabla\u)\cdot \u\,dx
=&\int_{\T^d}\bar{\mu}(\rho)\nabla\u\cdot \nabla\u\,dx
\ge \frac{\mu c_0} {2}\norm{\nabla\u}{L^{2}}^2
\end{align}
and
\begin{align}\label{energy5}
-\int_{\T^d}\nabla(\bar{\lambda}(\rho)\div\u)\cdot \u\,dx
=&\int_{\T^d}\bar{\lambda}(\rho)\div\u\cdot \div\u\,dx
  \ge \frac{\left(\lambda+\mu\right) c_0} {2}\norm{\div\u}{L^{2}}^2.
\end{align}
Next, we shall estimate each term on the right hand side of \eqref{energy3}.
First, it follows from integration by parts and the H\"older inequality that
\begin{align}\label{energy6}
|B_{11}|=&\Big|\frac{1}{\gamma+2\ze+K(L-1)}\int_{\T^d} \widetilde{f}_2\cdot \aa\,dx\Big|\nn\\
\le&C\left(|\int_{\T^d} (\u\cdot \nabla\aa)\cdot\aa\,dx|
   +|\int_{\T^d} \left(\aa\div\u\right)\cdot\aa\,dx|
   +|\int_{\T^d}b^2\div\u\cdot\aa\,dx|
   +|\int_{\T^d}b\div\u\cdot\aa\,dx|\right)\nn\\
\le&C\norm{\div\u}{L^{\infty}}\norm{\aa}{L^{2}}^2
    +C\left(\norm{b}{L^{\infty}}+\norm{b}{L^{\infty}}^2\right)\norm{\div\u}{L^{2}}\norm{\aa}{L^{2}}\nn\\
\le&\frac{\left(\lambda+\mu\right) c_0} {16}\norm{\div\u}{L^{2}}^2
      + C \left(\norm{\u}{H^{3}}
      +(1+\norm{b}{H^{3}}^2)\norm{b}{H^{3}}^2\right)\norm{\aa}{L^{2}}^2.
\end{align}
For the first term in $\widetilde{f}_1$,
\begin{align}\label{energy7}
\Big| \int_{\T^d} \u\cdot\nabla \u\cdot \u\,dx\Big|\le C \norm{\nabla\u}{L^{\infty}}\norm{\u}{L^{2}}^2
\le C \norm{\u}{H^{3}}^3.
\end{align}
Due to $\int_{\mathbb{T}^2}\rho \u \,dx=0,$ one can deduce from Lemma \ref{lem-Poi} that
\begin{align*}
\|(\sqrt{\rho }\u)(t)\|_{L^2}^2\le C\|\nabla \u(t)\|_{L^2}^2,
\end{align*}
which combines with Lemma \ref{lem2.2} further implies that
\begin{align}\label{P}
\|\u(t)\|_{L^2}^2\le C\|\nabla \u(t)\|_{L^2}^2.
\end{align}
Thanks to \eqref{P} and Lemma \ref{fuhe}, we have
\begin{align}\label{energy9}
\Big| \int_{\T^d} J(a)\nabla \aa\cdot \u\,dx\Big|
\le& C \norm{J(a)}{L^{\infty}}\norm{ \nabla \aa}{L^{2}}\norm{\u}{L^{2}}\nn\\
\le& \frac{\mu c_0} {32}\norm{\u}{L^{2}}^2+C \norm{J(a)}{H^{2}}^2\norm{ \aa}{H^{1}}^2,\nn\\
\le& \frac{\mu c_0} {32}\norm{\nabla\u}{L^{2}}^2+C \norm{p}{H^{3}}^2\norm{ \aa}{H^{3}}^2
\end{align}
and
\begin{align}
\label{energy10}
\Big| \int_{\T^d} \nabla J(a)(\mu\nabla\u+(\lambda+\mu)\div u)\cdot \u\,dx\Big|
\le& C \norm{\nabla J(a)}{L^{\infty}}\norm{\nabla \u}{L^{2}}\norm{\u}{L^{2}}\nn\\
\le&\frac{\mu c_0} {32}\norm{\nabla\u}{L^{2}}^2+C\norm{\nabla J(a)}{H^{2}}^2\norm{\u}{L^{2}}^2\nn\\
\le&\frac{\mu c_0} {32}\norm{\nabla\u}{L^{2}}^2+C\norm{p}{H^{3}}^2\norm{\u}{H^{3}}^2.
\end{align}
From integration by parts, there holds
\begin{align}\label{energy11}
\Big| \int_{\T^d} \frac{1}{1+a}\nabla \tau\cdot \u\,dx\Big|
\le& C\Big| \int_{\T^d} \frac{1}{1+a}\tau\cdot \div \u\,dx\Big|
+C\Big| \int_{\T^d}\tau\u\cdot \nabla \frac{1}{1+a}\,dx\Big|\nn\\
\le&C\norm{\frac{1}{1+a}}{L^{\infty}}\norm{\tau}{L^{2}}\norm{\nabla\u}{L^{2}}
      +C\norm{\nabla \frac{1}{1+a}}{H^{2}}\norm{\tau}{L^{2}}\norm{\u}{L^{2}}\nn\\
\le&\frac{\mu c_0} {32}\norm{\nabla\u}{L^{2}}^2
     +c_2\norm{\tau}{L^{2}}^2
     +C\norm{\nabla J(a)}{H^{2}}(\norm{\u}{L^{2}}^2+\norm{\tau}{L^{2}}^2)\nn\\
\le&\frac{\mu c_0} {32}\norm{\nabla\u}{L^{2}}^2
     +c_2\norm{\tau}{L^{2}}^2
     +C\norm{p}{H^{3}}(\norm{\u}{L^{2}}^2+\norm{\tau}{L^{2}}^2).
\end{align}
Combining with  \eqref{energy7}, \eqref{energy9}--\eqref{energy11} gives
\begin{align}\label{energy12}
|B_{12}|=\Big| \int_{\T^d} \widetilde{f}_1\cdot \u\,dx\Big|
\le&\frac{3\mu c_0} {32}\norm{\nabla\u}{L^{2}}^2+c_2\norm{\tau}{L^{2}}^2\nn\\
   &+C\Big(\norm{(\u,p)}{H^{3}}+  \norm{p }{H^{3}}^2\Big)
  \norm{(\aa,\u,\tau)}{H^3}^2.
\end{align}
For the last term of \eqref{energy3}, by integration by parts, we have
\begin{align}\label{energy13}
|B_{13}|=&(1+c_2)\left|\int_{\T^d} \u\div\tau \cdot \tau\,dx
       +\int_{\T^d} \tau \div\u\cdot \tau\,dx\right|\nn\\
\le&C\left|\int_{\T^d} \div\u |\tau|^2\,dx\right|\nn\\
\le&C\norm{\u}{H^{3}}\norm{\tau}{H^{3}}^2.
\end{align}
Inserting  \eqref{energy6}, \eqref{energy12}, and \eqref{energy13} into \eqref{energy3}
 and using \eqref{energy4}-\eqref{energy5}, we arrive at a basic energy inequality
\begin{align}\label{energy14}
&\frac12\frac{d}{dt}\norm{
\left(\frac{1}{\sqrt{\gamma+2\ze+K(L-1)}}\,\aa,\u,\sqrt{1+c_2}\,\tau\right)}{L^{2}}^2
     +\frac{\mu c_0} {4}\norm{\nabla\u}{L^{2}}^2
     + \frac{(\lambda+\mu) c_0} {4}\norm{\div\u}{L^{2}}^2
     +\norm{\tau}{L^{2}}^2\nn\\
&\quad\leq C\Big(\norm{(\u,p)}{H^{3}}+  \norm{p }{H^{3}}^2+(1+\norm{b}{H^{3}}^2)\norm{b}{H^{3}}^2\Big)
  \norm{(\aa,\u,\tau)}{H^3}^2.
\end{align}

 Next, we are concerned with the higher energy estimates.
 Applying  $\la^s$ with $1\le s\le 3$ to \eqref{energy1}
 and then taking $L^2$ inner product with
 $(\frac{1}{\gamma+2\ze+K(L-1)}\la^s\aa,\la^s\u,(1+c_2)\la^s\tau)$ yields
\begin{align}\label{energy16}
&\frac12\frac{d}{dt}\norm{\left(
\frac{1}{\sqrt{\gamma+2\ze+K(L-1)}}\,\la^s\aa,\la^s\u,\sqrt{1+c_2}\,\la^s\ta\right)}{L^2}^2
-\int_{\T^d}\la^{s}\div(\bar{\mu}(\rho)\nabla\u)\cdot\la^{s} \u\,dx\nn\\
&\qquad-\int_{\T^d}\la^{s}\nabla(\bar{\lambda}(\rho)\div\u)\cdot\la^{s} \u\,dx
   +(1+c_2)\int_{\T^d}\la^{s}\tau\cdot\la^{s} \tau\,dx\\
&\quad= \frac{1}{\gamma+2\ze+K(L-1)}\int_{\T^d}\la^{s} \widetilde{f}_2\cdot\la^{s} \aa\,dx
+\int_{\T^d}\la^{s} \widetilde{f}_1\cdot\la^{s} \u\,dx
-(1+c_2)\int_{\T^d}\la^{s} \div(\tau \u)\cdot \la^{s}\tau\,dx.\nn
\end{align}
The last three terms on the left hand side of \eqref{energy16}
can be dealt from \eqref{gaojie10}-\eqref{gaojie11}, then we get
\begin{align}\label{energy17}
&\frac12\frac{d}{dt}\norm{\left(\frac{1}{2}\la^s\aa,\la^s\u,\la^s\ta\right)}{L^2}^2
   +\frac{\mu c_0}{4}\norm{\la^{s+1}\u}{L^2}^2
   +\frac{(\lambda+\mu)c_0}{4}\norm{\la^{s}\div\u}{L^2}^2
   +(1+c_2)\norm{\la^{s}\tau}{L^2}^2 \nn\\
&\quad\le  \frac{1}{\gamma+2\ze+K(L-1)}\int_{\T^d}\la^{s} \widetilde{f}_2\cdot\la^{s} \aa\,dx
\nn\\
&\qquad+\int_{\T^d}\la^{s} \widetilde{f}_1\cdot\la^{s} \u\,dx
-(1+c_2)\int_{\T^d}\la^{s} \div(\tau \u)\cdot \la^{s}\tau\,dx
 + C\norm{p}{H^3}^2 \norm{ \u}{H^3}^2.
\end{align}
We now estimate successively terms on the right hand side of \eqref{energy17}.
To bound the first term in $\widetilde{f}_2$, we rewrite it into
\begin{align*}
\int_{\T^d}\la^{s} (\u\cdot\nabla \aa)\cdot\la^{s} \aa\,dx
=&\int_{\T^d}[\la^{s},\, \u\cdot\nabla] \aa\cdot\la^{s} \aa\,dx
+\int_{\T^d}\u\cdot\nabla\la^{s}\aa\cdot\la^{s} \aa\,dx.
\end{align*}
Then according to  Lemma \ref{jiaohuanzi} and integration by parts, we have
\begin{align}\label{energy18}
\Big|\int_{\T^d}\la^{s} (\u\cdot\nabla \aa)\cdot\la^{s} \aa\,dx\Big|
\le& C\norm{[\la^s,\,\u\cdot\nabla]\aa}{L^2}\norm{\la^{s}\aa}{L^2}
      +C\norm{\div\u}{L^\infty}\norm{\la^{s}\aa}{L^2}^2\nn\\
\le&C\big(\norm{\nabla \u}{L^\infty}\norm{\la^{s}\aa}{L^2}
      +\norm{\la^{s} \u}{L^2}\norm{\nabla \aa}{L^\infty}\big)\norm{\la^{s}\aa}{L^2}\nn\\
\le&C\norm{ \u}{H^3}\norm{\aa}{H^3}^2.
\end{align}
For the second term in $\widetilde{f}_2$, it follows from Lemma \ref{daishu} that
\begin{align}\label{energy19}
\frac{\gamma}{\gamma+2\ze+K(L-1)}\int_{\T^d}\la^{s} (\aa\div\u)\cdot\la^{s} \aa\,dx\leq&C\big(\|\nabla\u\|_{L^{\infty}}\|\aa\|_{H^{s}}
        +\|\nabla\u\|_{H^{s}}\|\aa\|_{L^{\infty}}\big)\norm{\la^{s}\aa}{L^2}\nonumber\\
         \leq&\frac{\mu c_0}{64}\|\nabla\u\|_{H^3}^2+C\|\aa\|_{H^{3}}^4.
\end{align}
Similarly, we have
\begin{align}\label{energy20}
\frac{2(\gamma-2)\zeta+K(\gamma-1)(L-1)}{\gamma+2\ze+K(L-1)}
     \int_{\T^d}\la^{s} (b\div\u)\cdot\la^{s} \aa\,dx
\leq &C\big(\|\nabla\u\|_{L^{\infty}}\|b\|_{H^{s}}
        +\|\nabla\u\|_{H^{s}}\|b\|_{L^{\infty}}\big)\norm{\la^{s}\aa}{L^2}\nonumber\\
\leq&\frac{\mu c_0}{64}\|\nabla\u\|_{H^3}^2+C\norm{b}{H^3}^2\norm{\aa}{H^3}^2
\end{align}
and
\begin{align}\label{energy21}
\frac{(\gamma-2)\zeta}{\gamma+2\ze+K(L-1)}\int_{\T^d}\la^{s} (b^2\div\u)\cdot\la^{s} \aa\,dx\leq&C\big(\|\nabla\u\|_{L^{\infty}}\|b^2\|_{H^{s}}
        +\|\nabla\u\|_{H^{s}}\|b^2\|_{L^{\infty}}\big)\norm{\la^{s}\aa}{L^2}\nonumber\\
         \leq&\frac{\mu c_0}{64}\|\nabla\u\|_{H^3}^2+C\norm{b}{H^3}^4\norm{\aa}{H^3}^2.
\end{align}
Collecting \eqref{energy18}-\eqref{energy21}, we can get
\begin{align}\label{energy22}
\frac{1}{\gamma+2\ze+K(L-1)}
    \int_{\T^d}\la^{s} \widetilde{f}_2\cdot\la^{s} \aa\,dx
&\le\frac{3\mu{c_0}}{64}\norm{\nabla\u}{H^3}^2\nn\\
&+C\left(\|\u\|_{H^3}+\|\aa\|_{H^3}^2+(1+\norm{b}{H^3}^2)\norm{b}{H^3}^2\right)
     \norm{\aa}{H^3}^2.
\end{align}
We get by a  similar  derivation of \eqref{I7} and \eqref{I6} that
\begin{align}\label{energy23}
\left|\int_{\T^d}\la^{s} \widetilde{f}_1\cdot\la^{s} \u\,dx\right|
\leq&\frac{3\mu{c_0}}{64}\norm{\la^{s+1}\u}{L^2}^2
+C\norm{\u}{H^3}^3
+C\norm{(\aa,\u,\tau)}{H^3}^2\norm{p}{H^3}^2+c_2\norm{\tau}{H^3}^2
\end{align}
and
\begin{align}\label{energy24}
(1+c_2)\left|\int_{\T^d}\la^{s} \div(\tau \u)\cdot \la^{s}\tau\,dx\right|
\leq&\frac{(\lambda+\mu)c_0}{16}\norm{\div \u}{H^3}^2
+C\norm{\tau}{H^3}^4+C\norm{ \u}{H^3}\norm{\tau}{H^3}^2.
\end{align}
Plugging \eqref{energy22}-\eqref{energy24} into \eqref{energy17} and combining with \eqref{energy14}, we arrive at the desired estimate \eqref{energy2}. This completes the proof of Lemma \ref{haosan}.
\end{proof}

\subsection{Dissipation estimates for $(\aa,\u,\tau,\mathbf{\widetilde{G}})$}
Next, we find the hidden dissipation of $\aa.$
We rewrite \eqref{energy1} into
 \begin{eqnarray}\label{dissp1}
    \left\{\begin{aligned}
    &\partial_t \aa+ (\gamma+2\ze+K(L-1))\div\u  =\widetilde{f}_2,\\
    &\partial_t\u - \mu\Delta \u -(\lambda+\mu)\nabla \div \u
            +\nabla\aa=\widetilde{f}_3,\\
    &\partial_t\tau+\tau+ \div (\tau\u)=0,\\
    &(\aa,\u,\tau)|_{t=0}=(\aa_0,\u_0,\tau_0),
    \end{aligned}\right.
    \end{eqnarray}
  where
   \begin{align}\label{f3}
    &\widetilde{f}_3\stackrel{\mathrm{def}}{=}- \u\cdot \nabla \u
    +J(a)\nabla\aa-J(a)\Big(\mu\Delta \u + (\lambda+\mu)\nabla \div \u\Big)+\frac{1}{1+a}\nabla\tau.
    \end{align}
Denote
\begin{align}\label{dissp4}
{\mathbf{\widetilde{G}}}\stackrel{\mathrm{def}}{=} \mathbb{Q}\u-\frac{1}{\nu}\Delta^{-1}\nabla {\aa}
\end{align}
with $\nu=\lambda+2\mu>0$.
Then, we find out that ${\mathbf{\widetilde{G}}}$ satisfies
\begin{align}\label{dissp5}
\partial_t{\mathbf{\widetilde{G}}}-\nu\Delta {\mathbf{\widetilde{G}}}
   =\frac{\gamma+2\ze+K(L-1)}{\nu}\q \u
   -\frac{1}{\nu}\Delta^{-1}\nabla \widetilde{f}_2+\q \widetilde{f}_3.
\end{align}
The goal of this subsection is to establish the dissipation estimates for $(\aa,\u,\tau,\mathbf{\widetilde{G}})$.
\begin{lemma}\label{miduhaosan}
Let $(\aa,\u,\tau,\mathbf{\widetilde{G}}) \in C([0, T];H^3)$ be a solution
  to the  system \eqref{dissp1} and  \eqref{dissp5},
 there holds
\begin{align}\label{dissp6}	
&\frac12\frac{d}{dt}\norm{(\aa,\u,\tau,\mathbf{\widetilde{G}})}{H^{{3}}}^2+\norm{(\aa,\tau)}{H^{{3}}}^2
   +\norm{(\nabla\u,\nabla\mathbf{\widetilde{G}})}{H^{{3}}}^2\\
 &\quad  \le C\norm{(p,\u)}{H^{{3}}}^2\norm{\nabla\u}{H^{{3}}}^2
+C\Big(\|\u\|_{H^3}+\norm{ (p,\aa,\tau)}{H^3}^2
+(1+\norm{b}{H^3}^2)\norm{b}{H^3}^2\Big)
\norm{(\aa,\u,\tau)}{H^{3}}^2.\nn
\end{align}
\end{lemma}
\begin{proof}
For $s=0,$ we can follow  the derivation of \eqref{energy14} to get the desired estimates, here we omit the details. For $s=1,2,3$,
applying ${\la^s} $ to  the above equation \eqref{dissp5},
 and  taking the $L^2$-inner product with ${\la^s}\mathbf{\widetilde{G}}$ give
\begin{align*}
	&\frac12\frac{d}{dt}\norm{\la^{s} \mathbf{\widetilde{G}}}{L^{2}}^2+
       \nu\norm{\la^{s+1} \mathbf{\widetilde{G}}}{L^{2}}^2\nn\\
	&\quad\le C\left|\int_{\T^d}\la^s\q \u\cdot\la^s \mathbf{\widetilde{G}}\,dx\right|
     +C\left|\int_{\T^d}\la^{s-1}\widetilde{f}_2\cdot\la^s \mathbf{\widetilde{G}}\,dx\right|
   +C\left|\int_{\T^d}\la^s\mathbb{Q}\widetilde{f}_3\cdot\la^s \mathbf{\widetilde{G}}\,dx\right|.
\end{align*}
With the aid of the Young inequality,  we obtain
\begin{align}\label{dissp7}
	\frac12\frac{d}{dt}\norm{\la^{s} \mathbf{\widetilde{G}}}{L^{2}}^2+\frac{\nu}{2}\norm{\la^{s+1} \mathbf{\widetilde{G}}}{L^{2}}^2
\le C\big(\norm{\la^{s-1}\q\u}{L^{2}}^2+\norm{\la^{s-2}\widetilde{f}_2}{L^{2}}^2
	+\norm{\la^{s-1}\widetilde{f}_3}{L^{2}}^2\big).
\end{align}
By using the fact
$$\div\q\u=\div\u,$$
we infer from the first equation in \eqref{dissp1} and \eqref{dissp4}
that $\aa$ satisfies a damped transport equation
\begin{align*}
\partial_t{\aa}+\frac{\gamma+2\ze+K(L-1)}{\nu}{\aa}=-(\gamma+2\ze+K(L-1))\div {\mathbf{\widetilde{G}}}+\widetilde{f}_2.
\end{align*}
For the above equation, we get by a similar derivation of \eqref{dissp7} that
\begin{align*}
\frac12\frac{d}{dt}\norm{\la^s {\aa}}{L^{2}}^2
+\frac{\gamma+2\ze+K(L-1)}{2\nu}\norm{\la^s {\aa}}{L^{2}}^2
\leq C\|\nabla\mathbf{\widetilde{G}} \|_{H^3}^2
+\left|\int_{\T^d}\la^s\widetilde{f}_2\cdot\la^s {\aa}\,dx\right|.
\end{align*}
The last term on the right hand side of the above equality can be bounded the same as \eqref{energy22}
\begin{align*}
\left|\int_{\T^d}\la^{s} \widetilde{f}_2\cdot\la^{s} \aa\,dx\right|
\le\varepsilon\|\nabla\u\|_{H^3}^2+C\left(\|\u\|_{H^3}+\|\aa\|_{H^3}^2
  +(1+\norm{b}{H^3}^2)\norm{b}{H^3}^2\right)\norm{\aa}{H^3}^2,
\end{align*}
from which, we have
\begin{align}\label{dissp10}
&\frac12\frac{d}{dt}\norm{\la^s {\aa}}{L^{2}}^2
   +\frac{\gamma+2\ze+K(L-1)}{2\nu}  \norm{\la^s {\aa}}{L^{2}}^2\nn\\
 &\quad\le\varepsilon\|\nabla\u\|_{H^3}^2+C\norm{\nabla\mathbf{\widetilde{G}}}{H^{{3}}}^2
+ C\left(\|\u\|_{H^3}+\norm{\aa}{H^3}^2
      +(1+\norm{b}{H^3}^2)\norm{b}{H^3}^2\right)\norm{\aa}{H^3}^2.
\end{align}
By multiplying  suitable large constant,
 it follows from  \eqref{dissp7} and \eqref{dissp10} that
\begin{align}\label{dissp101}
&\frac12\frac{d}{dt}\norm{(\aa,\mathbf{\widetilde{G}})}{H^{3}}^2+\norm{{\aa}}{H^{3}}^2+\norm{\nabla\mathbf{\widetilde{G}}}{H^{3}}^2\nn\\
&\quad\le \varepsilon\|\nabla\u\|_{H^3}^2+C(\|\u\|_{H^3}^2+\|\widetilde{f}_2\|_{H^1}^2+\|\widetilde{f}_3\|_{H^2}^2)\nn\\
&\qquad+ C\left(\|\u\|_{H^3}+\norm{\aa}{H^3}^2
      +(1+\norm{b}{H^3}^2)\norm{b}{H^3}^2\right)\norm{\aa}{H^3}^2.
\end{align}
Using the fact that $H^3(\T^d)$ is Banach algebra, the nonlinear terms in \eqref{dissp101} can be estimated as follows. Due to,
\begin{align*}
\norm{\u\cdot\nabla\aa}{H^{{1}}}^2+\norm{{\aa} \div\u}{H^{{1}}}^2
\le C\norm{\u}{H^{{2}}}^2\norm{\nabla\aa}{H^{{2}}}^2+C\norm{\aa}{H^{{2}}}^2\norm{\div\u}{H^{{2}}}^2
\leq C\norm{\u}{H^{{3}}}^2\norm{\aa}{H^{{3}}}^2
\end{align*}
and
\begin{align*}
\norm{b \div\u}{H^{{1}}}^2+\norm{b^2 \div\u}{H^{{1}}}^2
\le C\norm{b}{H^{{3}}}^2\norm{\u}{H^{{3}}}^2+\norm{b}{H^{{3}}}^4\norm{\u}{H^{{3}}}^2,
\end{align*}
we obtain
\begin{align}\label{dissp13}
\norm{\widetilde{f}_2}{H^{{1}}}^2
\le C\left(\norm{\aa}{H^{{3}}}^2+(1+\norm{b}{H^{{3}}}^2)\norm{b}{H^{{3}}}^2
    \right)\norm{\u}{H^{{3}}}^2.
\end{align}
From \eqref{f3}, by using Lemma \ref{fuhe}, we get
\begin{align}\label{dissp14}
\norm{\widetilde{f}_3}{H^{{2}}}^2
\le& C\norm{\u}{H^{{2}}}^2\norm{\nabla\u}{H^{{2}}}^2
   +C\norm{a}{H^{{2}}}^2\norm{\nabla \aa}{H^{{2}}}^2
    +C\norm{a}{H^{{2}}}^2\norm{\nabla\u}{H^{{3}}}^2
     +C\norm{1-J(a)}{H^{{2}}}^2\norm{\tau}{H^{{3}}}^2\nn\\
\le& C\norm{(p,\u)}{H^{{3}}}^2\norm{\nabla\u}{H^{{3}}}^2
      +C\norm{p}{H^{{3}}}^2\norm{(\aa,\tau)}{H^{{3}}}^2+C\norm{\tau}{H^{{3}}}^2.
\end{align}
Inserting \eqref{dissp13}-\eqref{dissp14} into \eqref{dissp101} leads to
\begin{align}\label{dissp102}
&\frac12\frac{d}{dt}\norm{(\aa,\mathbf{\widetilde{G}})}{H^{3}}^2+\norm{{\aa}}{H^{3}}^2+\norm{\nabla\mathbf{\widetilde{G}}}{H^{3}}^2\nn\\
&\quad\le \varepsilon\|\nabla\u\|_{H^3}^2+C(\|(\u,\tau)\|_{H^3}^2+\|(p,\u)\|_{H^3}^2\|\nabla\u\|_{H^3}^2)
\nn\\
&\qquad+ C\left(\|\u\|_{H^3}+\norm{(p,\aa)}{H^3}^2
      +(1+\norm{b}{H^3}^2)\norm{b}{H^3}^2\right)\norm{(\aa,\u,\tau)}{H^3}^2.
\end{align}
From the third equation in \eqref{dissp1} and \eqref{I6}, we get
\begin{align*}
	&\frac12\frac{d}{dt}\norm{\la^{s} \tau}{L^{2}}^2+\norm{\la^{s} \tau}{L^{2}}^2
\le\varepsilon\|\nabla\u\|_{H^3}^2 +C\left(\norm{\u}{H^{3}}
   +\norm{\tau}{H^{3}}^2\right)\norm{\tau}{H^{3}}^2.
\end{align*}
Then, we have
\begin{align}\label{dissp11}
	&\frac12\frac{d}{dt}\norm{ \tau}{H^{3}}^2+\norm{ \tau}{H^{3}}^2
\le\varepsilon\|\nabla\u\|_{H^3}^2 +C\left(\norm{\u}{H^{3}}
   +\norm{\tau}{H^{3}}^2\right)\norm{\tau}{H^{3}}^2.
\end{align}
Multiplying \eqref{dissp11} by a suitable large constant and
then adding to  \eqref{dissp102}, we can finally get  that
\begin{align}\label{dissp16}
&\frac12\frac{d}{dt}\norm{(\aa,\tau,\mathbf{\widetilde{G}})}{H^{{3}}}^2+\norm{(\aa,\tau)}{H^{{3}}}^2
	+\norm{\nabla\mathbf{\widetilde{G}}}{H^{{3}}}^2\nn\\
&\quad\le \varepsilon\|\nabla\u\|_{H^3}^2+C\norm{\u}{H^{{3}}}^2
       +C\|(p,\u)\|_{H^3}^2\|\nabla\u\|_{H^3}^2\nn\\
&\qquad+ C\left(\|\u\|_{H^3}+\norm{(p,\aa,\tau)}{H^3}^2
      +(1+\norm{b}{H^3}^2)\norm{b}{H^3}^2\right)\norm{(\aa,\u,\tau)}{H^3}^2.
\end{align}
It follows from \eqref{energy2} that
\begin{align}\label{dissp17}
&
\frac12\frac{d}{dt}\norm{\left(\frac{1}{\sqrt{\gamma+2\ze+K(L-1)}}\,\aa,\u,\sqrt{1+c_2}\,\tau\right)}{H^{3}}^2
+\mu\norm{\nabla\u}{H^{3}}^2
+(\lambda+\mu)\norm{\div\u}{H^{3}}^2+\norm{\tau}{H^{3}}^2\nn\\
& \quad\leq C\left(\norm{\u}{H^3}
       +(1+\norm{b}{H^3}^2)\norm{b}{H^3}^2
       +\norm{(p,\aa,\tau}{H^3}^2\right)
      \|(\aa, \u,\tau)\|_{H^{3}}^2,
\end{align}
Thus,  multiplying the above inequality \eqref{dissp17} by a suitable large constant and
then adding to  \eqref{dissp16}, we can finally get  that
\begin{align*}	
&\frac12\frac{d}{dt}\norm{(\aa,\u,\tau,\mathbf{\widetilde{G}})}{H^{{3}}}^2+\norm{(\aa,\tau)}{H^{{3}}}^2
   +\norm{(\nabla\u,\nabla\mathbf{\widetilde{G}})}{H^{{3}}}^2\nn\\
&\quad\le C\|(p,\u)\|_{H^3}^2\|\nabla\u\|_{H^3}^2
+C
\big(\|\u\|_{H^3}+\norm{( p,\aa,\tau)}{H^3}^2
+(1+\norm{b}{H^3}^2)\norm{b}{H^3}^2\big)
\norm{\aa,\u,\tau}{H^{3}}^2.
\end{align*}
This completes the proof of Lemma \ref{miduhaosan}.
\end{proof}

\subsection{Bootstrap argument}
In this section, we prove Theorem \ref{dingli2}.
For any $(p_0, \u_0,\tau_0, b_0)\in H^{{3}}({ \mathbb{T} }^d)$,  we can prove the local well-posedness of \eqref{mm}
by using the  standard energy method.  To prove the global well-posedness, the aim is
 on the global bound of
$(p,\u,\tau,b)$ in $H^3(\mathbb T^d)$. We use the bootstrapping argument and start by making the
ansatz that
\begin{align}\label{boot1}
	\sup_{t\in[0,T]}\norm{(p,\u,\tau,b)}{H^{3}}\le \delta,
\end{align}
for suitably chosen $0<\delta<1$. The main efforts are devoted to proving that, if the initial norm is taken to be sufficiently small, namely
\begin{align*}
\norm{(p_0,\u_0,\tau_0,b_0)}{H^{{3}}}\le\varepsilon,
\end{align*}
with sufficiently small $\varepsilon>0$, then
\begin{align*}
\sup_{t\in[0,T]}\norm{(p,\u,\tau,b)}{H^{3}}\le \frac{\delta}{2}.
\end{align*}
Under the assumption of \eqref{boot1}, we infer from \eqref{dissp6} that
\begin{align}\label{boot2}	&\frac12\frac{d}{dt}\norm{(\aa,\u,\tau,\mathbf{\widetilde{G}})}{H^{{3}}}^2+\norm{(\aa,\tau)}{H^{{3}}}^2
+\norm{(\nabla\u,\nabla\mathbf{\widetilde{G}})}{H^{{3}}}^2\nn\\
&\quad\le C
\delta(\delta^3+\delta+1)\norm{(\aa,\u,\tau)}{H^{3}}^2+C\delta^2\norm{\nabla\u}{H^{{3}}}^2,
\end{align}
where we use the fact
\begin{align*}
\norm{\aa}{H^{{3}}}^2=&\norm{p+K(L-1)b+\zeta b(b+2)}{H^{{3}}}^2\\
 \le&\norm{p}{H^{{3}}}^2+\zeta\norm{b^2}{H^{{3}}}^2+(K(L-1)+2\zeta)\norm{b}{H^{{3}}}^2\\
 \le&\norm{a}{H^{{3}}}^2+\zeta\norm{b}{H^{{3}}}^4+(K(L-1)+2\zeta)\norm{b}{H^{{3}}}^2\\
 \le&C\delta^2.
\end{align*}
Denote
\begin{align*}
	{\mathcal{\widetilde{E}}(t)}=&\norm{(\aa,\u,\tau,\mathbf{\widetilde{G}})}{H^{{3}}}^2
\end{align*}
and
\begin{align*}
	{\mathcal{\widetilde{D}}(t)}=&\norm{(\aa,\tau)}{H^{{3}}}^2+\norm{(\nabla\u,\nabla\mathbf{\widetilde{G}})}{H^{{3}}}^2.
\end{align*}
From \eqref{P}, choosing $\delta$ small enough in \eqref{boot2} implies that
\begin{align}\label{boot3}
	\frac{d}{dt}{\mathcal{\widetilde{E}}(t)}+\frac12{\mathcal{\widetilde{D}}(t)}\le 0.
\end{align}
On the one hand, notice that $\mathbf{\widetilde{G}}$ is a potential function, the following Poincar\'{e} inequality holds\begin{equation*}
		\begin{split}
			\norm{\mathbf{\widetilde{G}}}{L^{{2}}}\leq C\norm{\nabla\mathbf{\widetilde{G}}}{L^{{2}}}.
		\end{split}
	\end{equation*}
This gives rise  to
\begin{align*}
\mathcal{\widetilde{E}}(t)\le{C\mathcal{\widetilde{D}}(t)},
\end{align*}
from which and \eqref{boot3}  we get
\begin{align*}
	\frac{d}{dt}{\mathcal{\widetilde{E}}(t)}+c{\mathcal{\widetilde{E}}(t)}\le 0.
\end{align*}
Solving this inequality yields
\begin{align}\label{boot4}
	{\mathcal{\widetilde{E}}(t)}\le Ce^{-ct}.
\end{align}
Hence, we get
\begin{align}\label{boot5}
\int_0^t\left(\norm{ \aa(t')}{H^3}^2+\norm{ \u(t')}{H^3}^2+\norm{ \tau(t')}{H^3}^2\right)\,dt'\le C.
\end{align}
Due to \eqref{youjiexing}, $$ \frac{1}{2}c_0\le\rho,\eta\le 2c_0^{-1},$$
we have
\begin{align*}
 \tilde{c}_0\le\frac{1}{1+p},\ \frac{1}{1+b}\le \tilde{c}_0^{-1}.
\end{align*}
Hence, there holds
\begin{align*}
\frac12\norm{p}{H^{3}}^2-\frac{1}{2}\int_{\T^d}\frac{p}{1+p}(\la^{3}p)^2 \,dx
\ge C\norm{p}{H^{3}}^2,\\
\frac12\norm{b}{H^{3}}^2-\frac{1}{2}\int_{\T^d}\frac{b}{1+b}(\la^{3}b)^2 \,dx
\ge C\norm{b}{H^{3}}^2,
\end{align*}
from which and the Lemma \ref{gaojie}, we have
\begin{align}\label{boot6}
&\norm{(p,\u,\tau,b)}{H^{3}}^2
\leq \norm{(p_0,\u_0,\tau_0,b_0)}{H^{3}}^2\nn\\
&\quad+ C\int_0^t\left(\norm{\u}{H^3}+(\norm{p}{H^3}^2+\norm{\u}{H^3}^2)\norm{\u}{H^3}
       +\norm{(\aa,\u,\tau)}{H^3}^2\right) \norm{( p, \u,\tau, b)}{H^{3}}^2\,dt'.
\end{align}
 From \eqref{boot4} and \eqref{boot5}, exploiting
the Gronwall inequality to \eqref{boot6}  implies that
\begin{align*}
\norm{(p,\u,\tau,b)}{H^{3}}^2
\le&C\norm{(p_0,\u_0,\tau_0,b_0)}{H^{3}}^2\\
&\quad\times\exp
\left\{C\int_0^T\left(\norm{\u}{H^3}
      +(\norm{p}{H^3}^2+\norm{b}{H^3}^2)\norm{\u}{H^3}
        +\norm{(\aa,\u,\tau)}{H^3}^2\right)\,dt'\right\}\\
\le& C\varepsilon^2.
\end{align*}
Taking $\varepsilon$ small enough so that $C\varepsilon^2\le \delta^2/4$, we deduce from a continuity argument that the local solution can be extended as a global one in time.
 Consequently, we complete the proof of Theorem \ref{dingli2}.\hspace{15.4cm}$\square$

\bigskip
 \section*{Acknowledgement.}
{X. Zhai would like to thank Professor Yong Lu for his value discussions and suggestions.
Zhao is supported by the China Postdoctoral Science Foundation under grant 2023TQ0309.
Zhai is supported by  the Guangdong Provincial Natural Science Foundation under grant 2022A1515011977 and 2024A1515030115.}

\bigskip
\noindent{Conflict of Interest:} The authors declare that they have no conflict of interest.

\bigskip
\noindent{Data availability statement:}
 Data sharing not applicable to this article as no datasets were generated or analysed during the
current study.

\vskip .3in
\bibliographystyle{abbrv}

\end{document}